\documentclass[a4paper,12pt,oneside]{article}

\usepackage[utf8]{inputenc}
\usepackage[english]{babel}
\usepackage{amsthm}
\usepackage{amssymb}
\usepackage{amsmath}
\usepackage{hyperref}
\usepackage{mathrsfs}

\usepackage{pstricks}

 \usepackage{tikz}

\usepackage{graphicx}
\graphicspath{{images/}}

\usepackage{longtable,geometry}
\usepackage{color}

\usepackage{hyperref}

\newtheoremstyle{note}{12pt}{12pt}{}{}{\bfseries}{.}{.5em}{}
\title{\LARGE\textbf{Invariant Manifolds for Non-differentiable Operators}}
\author{Marco Martens and Liviana Palmisano}

\makeatletter
\newtheorem{theo}[equation]{Theorem}
\newtheorem{prop}[equation]{Proposition}

\newtheorem{defin}[equation]{Definition}

\newtheorem{rem}[equation]{Remark}
\newtheorem{cor}[equation]{Corollary}

\numberwithin{equation}{section}
\newtheorem{lem}[equation]{Lemma}

\usepackage{babel}

\newcommand{\N}{{\mathbb N}}
\newcommand{\Z}{{\mathbb Z}}
\newcommand{\R}{{\mathbb R}}

\renewcommand{\S}{{\mathbb S}^1}

\newcommand{\Ct}{{{\mathcal C}^3}}
\newcommand{\Cr}{{{\mathcal C}^r}}
\newcommand{\vers}{\longrightarrow}

\newcommand{\Cuno}{{{\mathcal C}^1}}

\renewcommand{\L}{{\mathscr L}}
\newcommand{\W}{{\mathscr W}}

\begin{document}
\maketitle
\author
\textcolor{blue}{}\global\long\def\sbr#1{\left[#1\right] }
\textcolor{blue}{}\global\long\def\cbr#1{\left\{  #1\right\}  }
\textcolor{blue}{}\global\long\def\rbr#1{\left(#1\right)}
\textcolor{blue}{}\global\long\def\ev#1{\mathbb{E}{#1}}
\textcolor{blue}{}\global\long\def\R{\mathbb{R}}
\textcolor{blue}{}\global\long\def\E{\mathbb{E}}
\textcolor{blue}{}\global\long\def\norm#1#2#3{\Vert#1\Vert_{#2}^{#3}}
\textcolor{blue}{}\global\long\def\pr#1{\mathbb{P}\rbr{#1}}
\textcolor{blue}{}\global\long\def\qq{\mathbb{Q}}
\textcolor{blue}{}\global\long\def\aa{\mathbb{A}}
\textcolor{blue}{}\global\long\def\ind#1{1_{#1}}
\textcolor{blue}{}\global\long\def\pp{\mathbb{P}}
\textcolor{blue}{}\global\long\def\cleq{\lesssim}
\textcolor{blue}{}\global\long\def\ceq{\eqsim}
\textcolor{blue}{}\global\long\def\Var#1{\text{Var}(#1)}
\textcolor{blue}{}\global\long\def\TDD#1{{\color{red}To\, Do(#1)}}
\textcolor{blue}{}\global\long\def\dd#1{\textnormal{d}#1}
\textcolor{blue}{}\global\long\def\eqdef{:=}
\textcolor{blue}{}\global\long\def\ddp#1#2{\left\langle #1,#2\right\rangle }
\textcolor{blue}{}\global\long\def\En{\mathcal{E}_{n}}
\textcolor{blue}{}\global\long\def\Z{\mathbb{Z}}
\textcolor{blue}{{} }

\textcolor{blue}{}\global\long\def\nC#1{\newconstant{#1}}
\textcolor{blue}{}\global\long\def\C#1{\useconstant{#1}}
\textcolor{blue}{}\global\long\def\nC#1{\newconstant{#1}\text{nC}_{#1}}
\textcolor{blue}{}\global\long\def\C#1{C_{#1}}
\textcolor{blue}{}\global\long\def\meas{\mathcal{M}}
\textcolor{blue}{}\global\long\def\cSpace{\mathcal{C}}
\textcolor{blue}{}\global\long\def\pspace{\mathcal{P}}

\begin{abstract} 
A general invariant manifold theorem is needed to study the topological classes of smooth dynamical systems. These classes are often invariant under renormalization. The classical invariant manifold theorem cannot be applied, because the renormalization operator for smooth systems is not differentiable and sometimes does not have an attractor. Examples are the renormalization operator for general smooth dynamics, such as unimodal dynamics, circle dynamics, Cherry dynamics, Lorenz dynamics, H\'enon dynamics, etc.
A general method to construct invariant manifolds of non-differentiable non-linear operators is presented. An application is that the $\mathcal C^{4+\epsilon}$ Fibonacci Cherry maps form a $\mathcal C^1$ codimension one manifold. 
\end{abstract}

\section{Introduction}
The classical invariant manifold theorem assures the existence of a submanifold preserved by a given map. The map has to satisfy two crucial conditions, namely, it has to be differentiable and it has to have an hyperbolic invariant set (for exampe an hyperbolic fixed point). The invariant manifold theorem has many applications in smooth hyperbolic dynamics where it supplies the framework for a complete topological description of the dynamics in phase space.

The invariant manifold theorem also plays an important role in the description of parameter space of families of dynamical systems. Namely, in one and two dimensional dynamics, the topology of systems can sometimes be characterized in terms of renormalization schemes. This is the case for circle diffeomorphisms, critical circle maps, unimodal maps, quadratic-like maps and dissipative H\'enon maps at the boundary of chaos. A topological class is a collection of systems which share a topological property. For example, systems which are conjugate on their attractors form topological classes. A topological class is often an invariant manifold of renormalization. The renormalization operator acts on parts of the corresponding spaces of systems. 

In the analytic setting of circle diffeomorphisms, unimodal maps, quadratic-like maps or strongly dissipative H\'enon-like maps at the boundary of chaos, the classical invariant manifold theorem can be applied and it assures that many topological classes are indeed finite codimension analytic manifolds. In these contexts the renormalization operators are indeed differentiable and have hyperbolic attractors.

One expects that many topological classes are also smooth submanifolds in the context of smooth systems,  like the ones mentioned before, as well as Lorenz dynamics, H\'enon dynamics, Cherry dynamics, etc.

Unfortunately, in the smooth context, the renormalization operators are not differentiable and sometimes they don't have attractors at all. The only results in this direction so far has been realized for the particular case of smooth unimodal dynamics, see \cite{Davie, dFdMP}. In this case the renormalization operator is not differentiable, however it has an hyperbolic attractor which is part of the space of analytic maps. The convergence of renormalization towards the space of analytic systems allows to extend a topological class of analytic unimodal maps into the corresponding topological class of smooth unimodal maps. 

Anticipating that renormalization will be a powerful tool to describe the topological classes of dynamical systems in many different contexts, one would like to have a general invariant manifold theorem which gives a method to prove that the topological classes are smooth manifolds also when the renormalization operator is not differentiable, like in smooth dynamics and when the renormalization does not have an attractor, like in Fibonacci unimodal dynamics, \cite{MilnorLyubic}, Lorenz dynamics, \cite{MW} and Cherry dynamics, \cite{5aut, P1, P2}. The main theorem presented here, namely Theorem \ref{InvMan}, provides such a method. 

The method has two parts. The first one is very general and is the same in all different contexts. The second part has a quantitative aspect which depends on the specific setting. We illustrate it in the most difficult situation where the renormalization operator is not differentiable and does not have an attractor, namely in Fibonacci Cherry dynamics. A very similar situation occurs for Fibonacci unimodal dynamics and Lorenz dynamics. However, the method is applicable in much broader contexts, for example smooth one dimensional dynamics and H\'enon dynamics. One has to adapt the quantitative aspect of the method following the guidline given in the application to Cherry dynamics.

The two main ingredients of the classical invariant manifold theorem, such as the differentiability and the hyperbolicity of the operator are substituted by weaker versions: jump-out differentiability, see Definition \ref{def:jumpoutdiff} and topological hyperbolicity, see Theorem \ref{InvMan}.
One of the reason why the renormalization operator is not differentiable in the smooth context is that composition is not differentiable. We replace the space of systems by decomposed systems in which renomalization is jump-out differentiable and the jump-out derivative is described, see Section \ref{deco}. The construction of the invariant manifold uses the classical graph transform which is studied with the non conventional method of curve dynamics. Curve dynamics does rely on Lipschitz regularity instead of differentiability.  

The main theorem, Theorem \ref{InvMan}, is stated in the context of Cherry dynamics. Renormalization of Cherry dynamics is introduced in Section \ref{section:renorm}. This concerns the dynamics of circle maps with a flat spot and critical exponents greater than one. For those maps whose critical exponent is between one and two and the rotation number is the Fibonacci number, the renormalizations diverge and in particular the Fibonacci renormalization operator does not have an attractor. The specific degeneration of the renormalizations is quantified by three invariants given in Proposition \ref{superformula}. These invariants determine the specific form of the quantitative aspects in the Invariant Manifold Theorem \ref{InvMan}, which applied in this case gives the following. Refer to Theorem \ref{manifold}.
\paragraph{Theorem A.}\label{manifoldintr}The $\mathcal C^{4+\epsilon}$ Fibonacci Cherry maps form a $\Cuno$ codimension one  manifold.

\bigskip

The ingredients developed to construct the smooth Fibonacci class of Theorem A. can also be used to study one of the fundamental questions in low-dimensional dynamics whether two systems with the same topological properties have also 
the same geometrical properties.
More precisely, consider two dynamical systems defined by two functions $f$ and $g$ and suppose that there exists an 
homeomorphism $h$ which conjugates $f$ and $g$ on their attractors. Is $h$ a $\mathcal C^{1+\beta}$, $\beta>0$, diffeomorphism? 
Such a regularity of the conjugacy implies that the geometry of the two systems is rigid, it is not possible to modify it 
on asymptotical small scales. The fact that the conjugacy has any regularity is by itself very surprising. 
It tells that as soon the topology of a systems is known, the asymptotic small scale geometry is also determined.
The rigidity question has been studied for circle diffeomorphisms
\cite{H79} \cite{Yo}, critical circle homeomorphisms \cite{dFdM}, \cite{Y}, \cite{Y1}, \cite{GdM}, \cite{GMdM}, unimodal maps
\cite{Lan}, \cite{S}, \cite{Mc}, \cite{Lyu1}, \cite{Lyu2}, \cite{dMP}, \cite{dFdMP}, 
circle maps with breakpoints \cite{KT1}, \cite{KK14} and for Kleinian groups \cite{Mo}. %and for analytic H\'enon-like maps at the boundary of chaos \cite{CLM05}, \cite{ML}.

In all these cases the topological conjugacy is indeed smooth, $\mathcal C^{1+\beta}$, $\beta>0$. In the context of Fibonacci Cherry dynamics one encounter a different behavior, also detected in unimodal Fibonacci dynamics, see \cite{MilnorLyubic}. Given a Fibonacci Cherry map $f$ with critical exponent bigger than one and smaller than two, the smoothness of the conjugacy is determined by the three geometrical invariants, $C(f)=\left(C_u(f), C_-(f),C_s(f)\right)$. Refer to Theorem \ref{LM}. 

\paragraph{Theorem B.} Two $\mathcal C^{4+\epsilon}$ Fibonacci Cherry maps, $f$ and $g$ with the same critical exponent bigger than one and smaller than two,  are $\mathcal C^{1+\beta}$, $\beta>0$, conjugate if and only if the three geometrical invariants are the same, $C(f)=C(g)$.

\bigskip

Different phenomena concerning rigidity of systems has been detected also for strongly dissipative H\'enon-like maps, see \cite{CLM05, ML} and Lorenz maps, see \cite{MW}. A new rigidity conjecture, taking into account the recent phenomena detected for one and two dimensional dynamical systems, has been formulated in \cite{MPW}. 
%\begin{lem}\label{nonlinearitylemma}
%Let $I=[a,b]\subset[0,1]$ be an interval and let $\varphi_1,\varphi_2, \varphi\in\text{Diff}^2\left([0,1]\right).$ Then 
%\begin{enumerate}
%\item $\text{dist}\left(\varphi\right)\leq\left|\varphi\right|$
%\item $\left|\varphi-\text{id}\right|_{\Cd}\leq\left|\varphi\right|$
%\item $\left|Z_{I}\varphi\right|\leq\left|I\right|\cdot\left|\varphi\right|$
%\item $\left|\varphi_1\circ\varphi_2\right|\leq\left|\varphi_1\right|+\left|\varphi_2\right|$
%\end{enumerate}
%\end{lem}
%The proof of the previous Lemma can be found in \cite{M}.
\paragraph {Standing notation.} Let $\alpha_n$ and $\beta_n$ be two sequences of positive numbers. We say that  $\alpha_n$ is of the order of $\beta_n$ if there exists an uniform constant $K>0$ such that, for $n$ big enough  $\alpha_n<K \beta_n$. We will use the notation $\alpha_n=O(\beta_n).$ Moreover we denote by $[a, b] = [b, a]$ the shortest interval
between $a$ and $b$ regardless of the order of these two points. The length of that interval
in the natural metric will be denoted by $\left|[a , b]\right|$. 
\paragraph{Acknowledgements.}
The authors would like to thank IMPAN for its hospitality. The work was initiated and mostly developed at IMPAN. The first author was partially supported by the NSF grant 1600554 and the second author was partially supported by the Leverhulme Trust through
the Leverhulme Prize of C. Ulcigrai and by the Trygger foundation, Project CTS
17:50.
\section{Smooth invariant manifolds} 
The classical invariant manifold theorem assures the existence of a submanifold preserved by a given map. The map has to satisfy two crucial conditions, namely, it has to be differentiable and it has to have an hyperbolic invariant set. Here a general invariant manifold theorem is given. The two crucial conditions needed for the classical invariant manifold theorem are substituted by weaker ones: jump-out differentiability and topological hyperbolocity, see Theorem \ref{InvMan}. 

\subsection{Jump-out differentiability}
\begin{defin}\label{def:jumpoutdiff}
Let $\left(B_0, \|\cdot\|_0\right)$ and $\left(B_1, \|\cdot\|_1\right)$ be two normed spaces such that $B_1\subset B_0$ and $\|b\|_0\leq\|b\|_1$ for all $b\in B_1$. Let $U\subset B_1$ be an open subset and let $$F:U\subset B_1\to B_1\subset B_0.$$ 
We say that $F$ is jump-out-differentiable if, for every $p\in U$, $$F:U\subset B_1\to B_0$$ is differentiable in $p$. Moreover the derivative $DF_p$ extends to a bounded operator $$DF_p:B_0\to B_0,$$ and  $DF_p$ depends continuously on $p$.
\end{defin}
\begin{rem}
The word "jump-out-differentiability" describes the following situation. The differentiable map $F:U\to B_0$ is not necessarily differentiable as map $F:U\to B_1$, although $F(U)\subset B_1$. 
The derivative jumps-out $B_1$, namely for some $p\in B_1$, $\text{Image} (DF_p(B_1))\not\subset B_1$. Observe that in finite dimensional vector spaces all norms are equivalent. As consequence the notion of jump-out-differentiability and differentiability are in fact equivalent when the dimension of $B_0$ is finite. The jump-out-differentiability is purely an infinite dimensional phenomenon.  

The following example will illustrate jump-out differentiability. It is similar to what we will encounter when discussing renormalization in Section \ref{jumpoutdiff}. Let $\eta\in C^1(\R,\R)$ and $Z:\R\to C^1(\R,\R)$ defined by
$$
Z(t)(x)=t\eta(tx).
$$
Then $Z$ is not differentiable. However,
$$
Z:\R\to C^0(\R,\R)
$$
is differentialble with derivative $DZ_t: \R\to C^0(\R,\R)$ given by
$$
DZ_t(x)=txD\eta(tx)+\eta(tx).
$$
Observe that
$$
\text{Image}(DZ_t)\not\subset C^1(\R,\R).
$$
\end{rem}
\begin{rem}\label{chainrule}
The chain rule holds for jump-out differentiable maps. More in detail, the composition of two jump-out differentiable maps $F$ and $G$ is again jump-out differentiable and the extension of the derivative of the composition is the composition of the extended derivatives of $F$ and $G$. 
\end{rem}
\subsection{Cone field}
In the following we introduce a cone field which will be used to study the graph transform in the proof of the main theorem. The role of the cone field here is the same as in the proof of the classical invariant manifold theorem. Observe that the cone field in Definition \ref{degconefield} is degenerating when $y$ goes to infinity. The form of the cone is inspired by the specific degeneration of renormalization of Fibonacci Cherry dynamics. This specific choice of the cone field is one of the quantitative aspect of the method discussed in the introduction which has to be adapted in other applications.

We fix two normed spaces $\left(B_0, \|\cdot\|_0\right)$ and $\left(B_1, \|\cdot\|_1\right)$ with $B_1\subset B_0$ and $\|b\|_0\leq\|b\|_1$ for all $b\in B_1$.  If $\left(B, \|\cdot\|\right)$ is a normed space then
 we will use the norm $|p|=|y|+\|b\|$ with $p=(y,b)\in \R\times B$ on $\R\times B$.
Moreover, we fix $\theta>0$, $0<\kappa<1$, we chose an open set  $U\subset B_1$ and we define a cone field on $\R\times B_0$.

\begin{defin}\label{degconefield}
Let $p=(y,b)\in \R\times B_0$. The cone at $p$ is defined as
$$C_p=\left\{(\Delta y,\Delta b)\in \R\times B_0 | \text{ }\theta |y|^{\kappa}\|\Delta b \|_0 < | \Delta y |\right\}.$$ 
\end{defin}

\subsection{Almost horizontal curves and almost vertical graps}
\begin{defin}
An almost horizontal curve $\gamma$ is the graph of a continuous function $\hat\gamma:[t_-,t_+]\subset [0,1]\to B_1$ having the following properties:
\begin{itemize}
\item $\hat\gamma:[t_-,t_+]\to B_0$ is continuously differentiable,
\item for every $p\in\gamma$, the tangent space satisfies $T_p\gamma\subset C_p$.
\end{itemize} 
The set of almost horizontal curves is denoted by $\Gamma_0$.
\end{defin}
\begin{lem}\label{perturbation}
Let $\gamma\in\Gamma_0$ and $p_1,p_2\in\gamma$. Then there exist $V_1$ neighborhood of $p_1$ and $V_2$ neighborhood of $p_2$ such that, for all $p'_1\in V_1$ and $p'_2\in V_2$ there exists $\gamma'\in\Gamma_0$ such that $p'_1,p'_2\in\gamma'$.
\end{lem}
\begin{proof} Let $p'_1=(y_1, b_1)$ and $p'_2=(y_2, b_2)$ any two points. Let $\varphi$ be the affine function in $y$ with $\varphi(y_1)=b_1-\hat\gamma(y_1)$ and $\varphi(y_2)=b_2-\hat\gamma(y_2)$. Consider the function $$\hat\gamma'(y)=\hat\gamma(y)+\varphi(y).$$ Then $\gamma'=\text{graph}(\hat\gamma')$ is a smooth curve passing trough $p'_1$ and $p'_2$. 
Let $y$ be in the domain of $\hat\gamma$ and $\left(\Delta y, \Delta b\right)$ be a tangent vector to the curve $\gamma'$ in the point $\left(y, \hat\gamma'(y)\right)$. Then $\Delta b=\left(D\hat\gamma+D\varphi\right)\Delta y$. By continuity, there exists $\epsilon>0$, independent on $y$, such that $\theta y^{\kappa}\|D\hat\gamma\|_0\leq 1-\epsilon$. As consequence $$\theta y^{\kappa}\|\Delta b\|_0\leq \left(1-\epsilon +\theta y^{\kappa} \|D\varphi\|_0\right)|\Delta y|.$$ Notice that if $p'_1=p_1$ and $p'_2=p_2$, then $\varphi\equiv 0$. Hence, for $p'_1$ close enough to $p_1$ and $p'_2$ close enough to $p_2$,  $\theta y^{\kappa}\|\Delta b\|_0\leq |\Delta y|.$ As consequence $\gamma'\in\Gamma_0$.
\end{proof}

\begin{defin}
An almost vertical graph $\omega$ is the graph of a continuous function $\hat\omega: U\subset B_1\to\R$ such that, for all almost horizontal curve $\gamma\in\Gamma_0$, $\gamma\cap\omega$ is at most one point.  The set of almost vertical graphs is denoted by $\Omega_0$.
\end{defin}
\subsection{Invariant Manifold Theorem}
We are now ready to state our main theorem.
\begin{theo}[\bf Invariant Manifold Theorem]\label{InvMan}
Let $\left(B_0, \|\cdot\|_0\right)$ and $\left(B_1, \|\cdot\|_1\right)$ be two normed spaces with $B_1\subset B_0$ and $\|b\|_0\leq\|b\|_1$ for all $b\in B_1$. Let $U\subset B_1$ be open, $\partial_\pm\in \Omega_0$ be the almost vertical graphs corresponding to $\hat\partial_\pm$ with $1>\hat\partial_+(b)>\hat\partial_-(b)>0$, for $b\in U$ and 
$$
D=\{(y,b)\in [0,1]\times U| \hat\partial_-(b)\le y \le \hat\partial_+(b)\}.
$$
For all $p=(y,b)\in D$ with $F(p)=(\tilde y,\tilde b)$ assume the map
$$F:D\to \R\times B_1\subset \R\times B_0$$ 
has the following properties. 
\begin{itemize}
\item $F$ is jump-out-differentiable.
\item $F$ has derivatives of the form: there exist $E\neq 0$ and $0\leq\kappa<1$ such that, if $DF_p\left(\Delta y,\Delta b\right)=\left(\Delta\tilde y,\Delta\tilde b\right)$, then 
\begin{equation}\label{cond2}
\left\{\begin{matrix}
\Delta\tilde y&=&\frac{E_p}{y}\Delta y+O\left(\|\Delta b\|_0\right)\\
&&\\
\|\Delta\tilde b \|_0&=&O\left(\frac{1}{\tilde y^{\kappa}}|\Delta y|+\|\Delta b\|_0\right)
\end{matrix}\right.
\end{equation}
where ${1}/{E}\leq |E_p|\leq E$.
\end{itemize}
Let ${1}/{4}>\delta>0$ and let $U_{\delta}\subset U$ such that, for all $b\in U_{\delta}$, $\hat\partial_{+}(b)<\delta$. Let $D_{\delta}\subset D$ be the set of $(y,b)\in D$, with $b\in  U_{\delta}$. 

\begin{itemize}
\item $F:D_{\delta}\to\R\times B_1\subset \R\times B_0$ is topologically hyperbolic for all $\delta>0$, i.e.
\begin{eqnarray*}
\text{If } F(\hat\partial_+(b),b)&=&(\tilde{y},\tilde{b}) \text{ then }
\tilde{y}\le \hat\partial_-(\tilde{b}),\\
\text{If } F(\hat\partial_-(b),b)&=&(\tilde{y},\tilde{b})\text{ then }\tilde{y}\ge \hat\partial_+(\tilde{b}),\\
\text{If } F(y,b)&=&(\tilde{y},\tilde{b})\text{ then }\tilde{b}\in U_{\delta}.
\end{eqnarray*}
\item $F$ has vertical $\xi$-expansion, $\xi>0$, i.e.
\begin{eqnarray}
\label{cond3}
\frac{y}{\tilde y^{\kappa}}&\geq &\xi.
\end{eqnarray}
 
\item $F$ has $\eta$-dominating horizontal expansion, $\eta>0$,  i.e.
\begin{eqnarray}
\label{cond4}
\frac{y^2}{\tilde y^{\kappa}}&\leq &\eta.
\end{eqnarray}

\end{itemize}
For $\theta>0$, $\delta>0$, $\eta>0$ small enough and $\xi>0$ large enough, the invariant set $$W=\left\{p\in D | \forall n\in\N\text{ }F^n(p)\in D\right\}$$ is an almost vertical graph of a $\Cuno$ function $\hat\omega:U_{\delta}\to \R$.
\end{theo}
\begin{rem}
Lemma \ref{lemma1} states that the expansion along vertical graphs is of the order ${y}/{\tilde y^{\kappa}}$. This motivates the name of "vertical $\xi$-expansion". Lemma \ref{lemma1} and Lemma \ref{yexpansion} state that the expansion along almost horizontal curves dominates the expansion along almost vertical graphs by a factor of order ${y^2}/{\tilde y^{\kappa}}$. This motivates the name of "$\eta$-dominating horizontal expansion".
\end{rem}
\begin{rem}
Observe that $W$ is not necessarily the stable manifold of an hyperbolic fixed point, as in the most classical context. Theorem \ref{InvMan} will in fact be applied to the context of Cherry maps whose renormalization does not have a fixed point. $W$ will correspond to the class of maps with Fibonacci rotation number.  The renormalizations diverge to infinity. 
\end{rem}
From now on we assume the conditions of the Theorem \ref{InvMan}. The proof involves adjustments of parameters $\delta$, $\theta$, $\chi$, $\xi$ and $\eta$. When these adjustments are required in a proof of a lemma, then they are expressed in the statement of the lemma and assumed in the sequel. 
\begin{lem}\label{lemma1}
For $\theta>0$ small enough and $\xi>0$ large enough, there exists $K>0$ such that if $p=(y,b)\in D$ with $\tilde p=F(p)=(\tilde y, \tilde b)\in D$ and $\Delta q=\left(\Delta y, \Delta b\right)$ with $\Delta\tilde q=DF_p\left(\Delta q\right)\notin C_{F(p)}$, then 
\begin{enumerate}
\item $|\Delta y|\leq K y\|\Delta b\|_{0}$,
\item $|\Delta\tilde q|\leq K\left(\frac{y}{\tilde y^{\kappa}}\right)|\Delta q|$.
\end{enumerate}
\end{lem}
\begin{proof}
Let $\Delta\tilde q=\left(\Delta \tilde y, \Delta\tilde b\right)$. By (\ref{cond2}) and the fact that $\Delta\tilde q\notin C_{F(p)}$ we get
\begin{eqnarray*}
\frac{E_p}{ y}|\Delta y|+ O\left(\| \Delta b\|_0\right)\leq |\Delta\tilde y|< \theta\tilde y^{\kappa} \| \Delta\tilde b\|_0
= \theta\tilde y^{\kappa}O\left(\frac{1}{\tilde y^{\kappa}}|\Delta y|+\| \Delta b\|_0\right).
\end{eqnarray*}
As consequence, for $\theta>0$ small enough we get property $1$. For proving property $2$, it is enought to use (\ref{cond2}), (\ref{cond3}) and property $1$, namely
$$|\Delta\tilde q|\leq\frac{E_p}{y}|\Delta y|+O\left(\frac{1}{\tilde y^{\kappa}}|\Delta y|+\|\Delta_b\|_0\right)=O\left(\frac{y}{\tilde y^{\kappa}}\|\Delta b\|_0\right)=O\left(\frac{y}{\tilde y^{\kappa}}|\Delta q|\right).$$
\end{proof}
The following lemma states that the cone field $C_p$ is expanding and invariant.
\begin{lem}\label{yexpansion}
For ${1}/{4}>\delta>0$ small enough, the following holds. Let $p=(y,b)\in D$ and $\Delta q=(\Delta y, \Delta b)\in C_p$ with $\Delta \tilde q=DF_p(\Delta x)=(\Delta\tilde y, \Delta\tilde b)$, then $\Delta \tilde q\in C_{F(p)}$ and $$|\Delta\tilde y|\geq\frac{1}{2Ey}|\Delta y|\geq 2 |\Delta y|.$$ 
\end{lem}
\begin{proof}
By (\ref{cond2}) and by the fact that $\Delta x\in C_p$ we get 
$$|\Delta\tilde y|\geq \frac{1}{Ey}|\Delta y|+O\left(|\Delta b\|_{0}\right)=\frac{1}{Ey}|\Delta y|+O\left(\frac{1}{\theta y^{\kappa}}|\Delta y|\right)=\frac{1}{Ey}|\Delta y|\left(1+O\left(\frac{y^{1-\kappa}}{\theta}\right)\right).$$ The expansion estimate follows by taking $\delta>0$ small enough. Left is to show the cone invariance.
From $(\ref{cond2})$, the fact that $\Delta q\in C_p$ and the expansion estimate we get 
$$\theta \tilde y^{\kappa}\|\Delta\tilde b\|_0=O\left(1+\frac{\tilde y^{\kappa}}{y^{\kappa}}\right)|\Delta y|=O\left(y+{\tilde y^{\kappa}}{y^{1-\kappa}}\right)|\Delta\tilde y|.$$ The cone invariance follows by taking $\delta$ small enough. 
\end{proof}

We denote by $\Gamma$ the subset of $\Gamma_0$ of all almost horizontal curves having the additional property that $\gamma\cap\partial_-$ and $\gamma\cap\partial_+$ consist each of exactly one point. 
We denote by $\Omega$ the subset of $\Omega_0$ of all almost vertical graphs having the additional property that $\omega\subset D_{\delta}$. 
The next lemma states that the set of almost vertical graphs is invariant. This crucial property reduces the method to curve dynamics.
\begin{lem}\label{Fgamma}
If $\gamma\in\Gamma$, then $F(\gamma\cap D_{\delta})\in\Gamma$.
\end{lem}
\begin{proof}
Observe that, by the definition of $\Omega$, $\gamma\cap D_{\delta}$ is itself and almost horizontal curve, it only contains one component. Hence, we may assume that $\gamma\subset D_{\delta}$. We introduce the map $$F{\gamma}:[t_-,t_+]\ni t\mapsto F\left(\left(t,\hat\gamma(t)\right)\right)\subset B_1.$$
By Remark \ref{chainrule}, $F\gamma:[t_-,t_+]\to B_0$ is continuously differentiable and by the cone invariance, Lemma \ref{yexpansion}, for every $p\in\gamma$, $T_{F(p)}F\left(\gamma\right)=DF_p\left(T_p\gamma\right)\subset C_{F(p)}$. Suppose now that $F(\gamma)$ is not a graph. Then, there exists $\tilde p\in F(\gamma)$ such that the tangent vector of $F(\gamma)$ in $\tilde p$ is of the form $(0,\Delta b)$ which is not in the cone $C_{\tilde p}$. This contradict the cone invariance. We proved that $F(\gamma)\in\Gamma_0$. 
Because $F$ is topologically hyperbolic we have that $F(\gamma)\cap\partial_{\pm}\neq\emptyset$. Hence $F(\gamma)\in\Gamma$.
\end{proof}
\begin{cor}\label{Fgamma0}
If $\gamma\in\Gamma_0$, then $F(\gamma\cap D_{\delta})\in\Gamma_0$.
\end{cor}
\begin{cor}\label{Finj}
If $\gamma\in\Gamma$, then $F_{|\gamma\cap D_{\delta}}$ is injective.
\end{cor}
The proof of Theorem \ref{InvMan} is divided into $2$ steps inspired by the graph transform method. In the following $2$ subsections, discussing these steps, we assume the hypotheses of Theorem \ref{InvMan}.

\subsection{The graph transform}
In this section we introduce the graph transform. The invariant manifold of $F$ will be the fixed point of the graph transform. 

We define a distance $d$ on $\Omega$. Let $\omega_1,\omega_2\in\Omega$ and $\gamma\in\Gamma$. Let $p_1=(y_1, b_1)=\gamma\cap\omega_1$ and let $p_2=(y_2, b_2)=\gamma\cap\omega_2$.  We define the distance $$d\left(\omega_1,\omega_2\right)=\sup_{\gamma\in\Gamma}|y_1-y_2|\leq 1.$$

\begin{lem}
$\left(\Omega, d\right)$ is complete.
\end{lem}
\begin{proof}
Notice that the uniform distance on the functions $\hat\omega$ is bounded by $d$. As consequence, any Cauchy sequence $\omega_n$ in $\Omega$ has a limit $\omega$ which is the graph of a continuous function $\hat\omega$. Because $D_{\delta}$ is closed, $\hat\omega\in D_{\delta}$. Moreover, for any $\gamma\in\Gamma$, $\gamma\cap\omega\neq\emptyset$. It remains to  prove that this intersection consists of only one point. By contradiction, suppose there is a $\gamma\in\Gamma$ such that $\gamma\cap\omega\supset \left\{p_1,p_2\right\}$. By Lemma \ref{perturbation}, there exists $\tilde\gamma\in\Gamma$ such that, for $n$ large enough, $\tilde\gamma\cap\omega_n$ contains at least two points. This is a contradiction and hence $\omega\in\Omega$.
\end{proof}

We are ready to define the graph transform $T:\Omega\to \Omega$. Let $\omega\in\Omega$, $b\in U_{\delta}$ and $\gamma_b$ the (almost) horizontal curve at $b$, given by $\hat\gamma_b(t)=(t,b)$. By Lemma \ref{Fgamma}, we know that $F\gamma_b\in\Gamma$, hence $F\gamma_b\cap\omega=\left\{\tilde p\right\}$ and by Corollary \ref{Finj}, $F^{-1}(\omega)$ is the graph of a function. We call this graph ${T\omega}$.

\begin{lem}
 $T{\omega}\in\Omega$.
\end{lem}
\begin{proof}
Let $\gamma\in\Gamma$. We have to prove that $\gamma\cap T\omega$ is a point. By Lemma \ref{Fgamma}, $F\gamma\in\Gamma$, hence $F\gamma\cap\omega=\left\{\tilde p\right\}$. As a consequence, $$\gamma\cap T\omega=F^{-1}(\tilde p),$$ which is a unique point by Corollary \ref{Finj}.
\end{proof}
The graph transform $T:\Omega\to \Omega$ is defined by $\omega\mapsto T\omega$.
\begin{lem}
For $\delta>0$ small enough, $T$ is a contraction.
\end{lem}
\begin{proof}
 Let $\omega_1,\omega_2\in\Omega$ and $\gamma\in\Gamma$. Denote by $p_1=(y_1,b_1)=\gamma\cap T\omega_1$, by $p_2=(y_2,b_2)=\gamma\cap T\omega_2$, by $\tilde p_1=(\tilde y_1,\tilde b_1)=F\gamma\cap \omega_1$ and by $\tilde p_2=(\tilde y_2,\tilde b_2)=F\gamma\cap \omega_2$.
 By Lemma \ref{yexpansion}, for $\delta$ small enough,  $$|y_1-y_2|\leq {2Ey} |\tilde y_1-\tilde y_2| \leq\frac{1}{2} |\tilde y_1-\tilde y_2|,$$ and as a consequence 
 $$d\left(T\omega_1,T\omega_2\right)\leq\frac{1}{2} d\left(\omega_1,\omega_2\right).$$
 \end{proof}
 The graph transform is then a contraction on a complete space. It follows the next proposition:
 \begin{prop}
 $T$ has a unique fixed point $\omega^*\in\Omega$.
 \end{prop}
\begin{lem}\label{omegaW}
 Let $\omega^*$ be the fixed point of $T$. Then $$\omega^*=W=\left\{p\in D_{\delta} | \forall n\in\N\text{ }F^n(p)\in D_{\delta}\right\}.$$
\end{lem}
\begin{proof}
Take $p\in\omega^*\subset D_{\delta}$. Because $\omega^*$ is a fixed point for $T$, $\omega^*=F^{-1}\omega^*$. As a consequence, $F(p)\in\omega^*\subset D_{\delta}.$ Hence $\omega^*\subset W$. Take now $p=(y,b)\in W$. Notice that, for all $n\in\N$, $p_n=F^n(p)=(y_{n},b_n)\in D_{\delta}$. Let $\hat p=(\hat y,b)\in \omega^*\subset D_{\delta}$ with $\hat y=\omega^*(b)$. Observe that, because $\omega^*$ is a fixed point for $T$, for all $n\in\N$, $\hat p_n=F^n(\hat p)=(\hat y_n,b_n)\in\omega^*\subset D_{\delta}$. In particular $\hat y_n\leq \frac{1}{4}$. We want to prove that $p=\hat p\in\omega^*$. Suppose not. Let $\gamma$ be the line segment between $p$ and $\hat p$ and let $\gamma_n=F^n\gamma$. Observe that $\gamma_n\subset D_{\delta}$.  By Lemma \ref{yexpansion}, for all $n\in\N$, $$|y_n-\hat y_n|\geq 2^n |y-\omega^*(b)|.$$ 
We get to a contradiction, noticing that, for $n$ big enough, either $p_n$ or $\hat p_n$ is not in $D_{\delta}$. 
 \end{proof}

\subsection{Differentiability}
In this section we prove that $W=\omega^*$ is a $\Cuno$ codimension one manifold. This will conclude the proof of Theorem \ref{InvMan}.
\bigskip

A {\it plane} is a codimension one subspace of $\R\times B_0$ which is a the graph of a functional $b^*\in \text{Dual}(B_0)$. We identify the plane with the corresponding functional $b^*$. In particular, $\text{Dual}(B_0)$ is the space of planes and carries a corresponding complete distance $d^*_{B_0}$. Fix a constant $\chi>0$.
\begin{defin}
Let $p=(y,b)\in\omega^*$. A plane $V_p$ is admissible for $p$ if it has the following properties:
\begin{enumerate}
\item if $\left(\Delta y, \Delta b\right)\in V_p$ then $|\Delta y|\leq
\chi y\|\Delta b\|_0$,
\item $V_p$ depends continuously on $p$ with respect to $d^*_{B_0}$. 
\end{enumerate}
The set of planes admissible for $p$ is denoted by $\text{Dual}_p(B_0)$. A plane field is a continuous assignment $V :p\mapsto V_p$ where $V_p$ is admissible for $p$. The set of plane fields is denoted by $\Omega_1$.
\end{defin}
\begin{rem}\label{intlinecone}
Observe that, for $\delta>0$ small enough, if $p\in\omega^*$ and $V_p$ is an admissible plane for $p$, then for all straight lines $\gamma$ with direction in $C_p$, $\gamma\cap V_p$ is a unique point.
\end{rem}
We fix $p=(y,b)\in\omega^*$ and we define a distance on $\text{Dual}_p(B_0)$. Let $V_p,V'_p\in \text{Dual}_p(B_0)$ and let $\gamma$ be a straight line with direction in $C_p$ not necessarily passing trough the origin.  Let $\delta>0$ be small enough such that Remark \ref{intlinecone} can be applied. Denote by $\Delta q=(\Delta y,\Delta b)=V_p\cap\gamma$, $\Delta q'=(\Delta y',\Delta b')=V'_p\cap\gamma$. The distance $d_{1,p}$ on $\text{Dual}_p(B_0)$ is defined as
$$d_{1,p}\left(V_p,V'_p\right)=\sup_{\gamma}\frac{|\Delta y-\Delta y'|}{\min\left\{|\Delta q|,|\Delta q'|\right\}}.$$
This is a complete metric. The distance $d_1$ on $\Omega_1$ is the corresponding uniform distance. Namely, if $V,V'\in \Omega_1$ then
$$d_1\left(V,V'\right)=\sup_{p}d_{1,p}\left(V_p,V'_p\right).$$

\begin{lem}
For $\delta>0$ small enough, $d_1$ is a distance. 
\end{lem}
\begin{proof}
We only have to show that $d_1$ is bounded. We use the notations from the definition of $d_1$. Without lose of generality we assume $|\Delta q|\leq|\Delta q'|$ and we estimate ${|\Delta y|+|\Delta y'|}/{|\Delta q|}$. By the definition of $V_p$ and $V'_p$ we get
\begin{equation}\label{admisplane}
|\Delta y|+|\Delta y'|\leq \chi y \left(\|\Delta b\|_0+\|\Delta b'\|_0\right).
\end{equation}
Define $\Delta=\Delta b'-\Delta b$. Because the direction of $\gamma$ is in the cone $C_p$ and by (\ref{admisplane}), we get
\begin{equation*}
\|\Delta\|_0\leq\chi\frac{y^{1-\kappa}}{\theta }\left(\|\Delta b\|_0+\|\Delta b'\|_0\right),
\end{equation*}
and 
\begin{equation*}
\|\Delta b'\|_0\leq \|\Delta b\|_0 +\chi\frac{y^{1-\kappa}}{\theta }\left(\|\Delta b\|_0+\|\Delta b'\|_0\right).
\end{equation*}
As a consequence, for $\delta$ small enough,
\begin{equation}\label{admisplane1}
\frac{1}{2}\leq\frac{\|\Delta b'\|_0}{\|\Delta b\|_0}\leq 2.
\end{equation}
Observe that, for $\delta$ small enough,
\begin{equation}\label{admisplane2}
|\Delta q|\geq\frac{1}{2}\|\Delta b\|_0,
\end{equation}
and by (\ref{admisplane}), (\ref{admisplane1}) and (\ref{admisplane2}) we get
$$\frac{|\Delta y|+|\Delta y'|}{|\Delta q|}\leq 2.$$
\end{proof}
\begin{lem}
For $\delta>0$ small enough, $\left(\Omega_1, d_1\right)$ is complete.
\end{lem}
\begin{proof}
Let $V(n)\in\Omega_1$ be a $d_1$-Cauchy sequence. Then, for all $p$, $V_p(n)$ is a $d_{1,p}$-Cauchy sequence. Hence $V_p(n)$ converges to $V_p$ in $\text{Dual}_p(B_0)$. Let $V:\omega^*\to \text{Dual}(B_0)$ be the point-wise limit. Observe now that, for $\delta$ small enough, $$d^*_{B_0}\left(V_p,V'_p\right)\leq 2 d_{1,p}\left(V_p,V'_p\right),$$ for all $p\in\omega^*$. Hence $V(n)$ is also a Cauchy sequence in the space of continuous functions $\omega^*\to \text{Dual}(B_0)$ carrying the uniform distance corresponding to $d^*_{B_0}$. This space is complete, hence $V$ is continuous.
\end{proof}
The next step is to define the plane field transform $T_1: \Omega_1\to \Omega_1$. Let $V\in \Omega_1$ and $p\in\omega^*$, then  
$$T_1V_p=DF^{-1}_p\left(V_{F(p)}\right).$$
\begin{lem}
For $\delta>0$ small enough and for $\chi>0$ large enough, $T_1V\in\Omega_1$, for all $V\in\Omega_1$.
\end{lem}
\begin{proof}
We prove that, for all $p\in\omega^*$, $T_1V_p$ is an admissible plane for $p$. Consider the vectors $\left(\Delta y, \Delta b\right)\in T_1V_p$ and $DF_p\left(\Delta y, \Delta b\right)=(\Delta\tilde y, \Delta\tilde b )\in V_{F(p)}$. From (\ref{cond2}) and from the fact that $V_{F(p)}$ is admissible in ${F(p)}$ we get $$\frac{1}{Ey}|\Delta y|+ \chi\tilde y^{1-\kappa}O\left(|\Delta y|\right)\leq \chi\tilde y O\left(\|\Delta b\|_0\right)+O\left(\|\Delta b\|_0\right).$$ As consequence, for $\delta>0$ small enough and $\chi>0$ large enough, we get 
$$|\Delta y|\leq \chi y \|\Delta_b\|_0.$$ In particular $DF_p^{-1}\left(V_{F(p)}\right)$ is a codimension one subspace and this implies that $DF_p$ is transversal to $V_{F(p)}$. This transversality gives that $T_1V_p$ depends continuously on $p$.
\end{proof}
\begin{lem}
For $\eta>0$ small enough, $T_1$ is a contraction.
\end{lem}
\begin{proof}
We use the notations from the definition of $d_1$ and we denote by $\Delta\tilde q=DF_p\left(\Delta q\right)=\left(\Delta\tilde y, \Delta\tilde b\right)$ and by $\Delta\tilde q'=DF_p\left(\Delta q'\right)=\left(\Delta\tilde y', \Delta\tilde b'\right)$.  By Lemma \ref{lemma1} and Lemma \ref{yexpansion} we get 
$$\frac{|\Delta y'-\Delta y|}{\min\left\{|\Delta q|, |\Delta q'|\right\}}= O\left(\frac{y^2}{\tilde y^{\kappa}}\right)\frac{|\Delta\tilde y'-\Delta\tilde y|}{\min\left\{|\Delta\tilde q|, |\Delta\tilde q'|\right\}}\leq \frac{1}{2}d_{1, F(p)}\left(V_{F(p)},V'_{F(p)}\right),$$ when $\eta$ is small enough. Hence $d_{1}\left(T_1V, T_1V'\right)\leq\frac{1}{2}d_{1}\left(V, V'\right).$
\end{proof}
\begin{prop}
For $\delta>0$ small enough and $\chi>0$ large enough, $T_1$ has a unique fixed point $V^*\in\Omega_1$.
\end{prop}

Although $V^*_p$ is a subspace of $\R\times B_0$, in the sequel we abuse the notation by denoting the set $\left\{p+v |v\in V^*_p\right\}$ also by $V^*_p$.
Let $p\in\omega^*$ and take $\gamma\in\Gamma$ close enough to $p$ such that $\gamma\cap\omega^*=\left\{p+\Delta q=p+\left(\Delta y,\Delta b\right)\right\}$ and $\gamma\cap V^*_p=\left\{p+\Delta q'=p+\left(\Delta y',\Delta b'\right)\right\}$. We define by $$A=\sup_{p}\limsup_{\gamma\to p}\frac{|\Delta y-\Delta y'|}{|\Delta q|}.$$
Observe that $V^*$ describes the tangent bundle of $\omega^*$ if and only if $A=0$.
\begin{lem}
For $\delta>0$ small enough, $A\leq 1$.
\end{lem}
\begin{proof}
Let $p=(y,b)$ and denote by $\Delta=\Delta b'-\Delta b$. From the definition of admissible planes we get 
\begin{equation}\label{fin1}
|\Delta y'|\leq\chi y\left(\|\Delta b'\|_0\right),
\end{equation}
and by the definition of almost horizontal lines,
\begin{equation}\label{fin2}
\frac{1}{2}\theta y^{\kappa}\|\Delta \|_0\leq |\Delta y|+ |\Delta y'|,
\end{equation}
when $\gamma$ is close enough to $p$. Consider the straight line $L$ going to $p$ and $p+\Delta q$. $L$ intersect $\omega^*$ in two points. Hence, $L$ is not an almost horizontal curve and $\Delta q\notin C_p$. As consequence, when $\gamma$ is close enough to $p$.
\begin{equation}\label{fin3}
|\Delta y|\leq 2\theta y^{\kappa}\|\Delta b \|_0.
\end{equation}
By (\ref{fin1}), (\ref{fin2}) and (\ref{fin3}) we get 
$$\|\Delta \|_0\leq \frac{2\theta+\chi y^{1-\kappa}}{\frac{1}{2}\theta-\chi y^{1-\kappa}}\|\Delta b\|_0,$$ and 
\begin{equation}\label{fin4}
\|\Delta  b'\|_0\leq 6\|\Delta b\|_0,
\end{equation}
when $\delta>0$ is small enough. By (\ref{fin1}), (\ref{fin3}) and (\ref{fin4}) we get 
$$\frac{|\Delta y-\Delta y'|}{|\Delta q|}\leq\frac{|\Delta y|+|\Delta y'|}{\|\Delta b \|_0+|\Delta y|}\leq 12\chi y+4\theta y^{\kappa}\leq 1,$$ for $\delta>0$ is small enough.
\end{proof}
Lemma \ref{omegaW} and the following proposition conclude the proof of Theorem \ref{InvMan}. Namely $W=\omega^*$ is a $\Cuno$ manifold.
\begin{figure}[h]\label{FigP}
\centering
\includegraphics[width=0.8\textwidth]{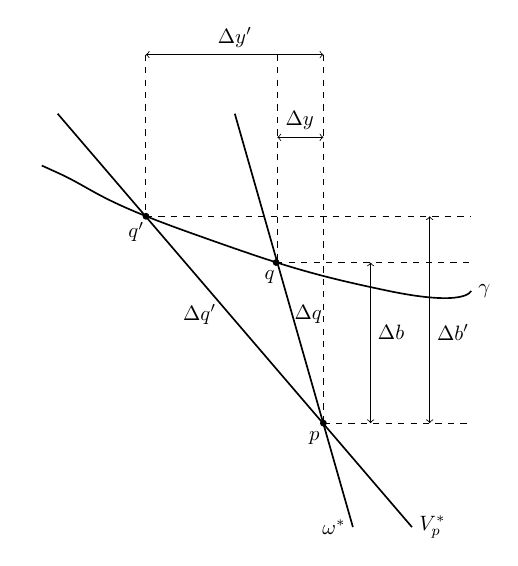}
\caption{Notation for the proof of Proposition \ref{tgV}}
\end{figure}
\begin{figure}[h]\label{FigI}
\centering
\includegraphics[width=0.8\textwidth]{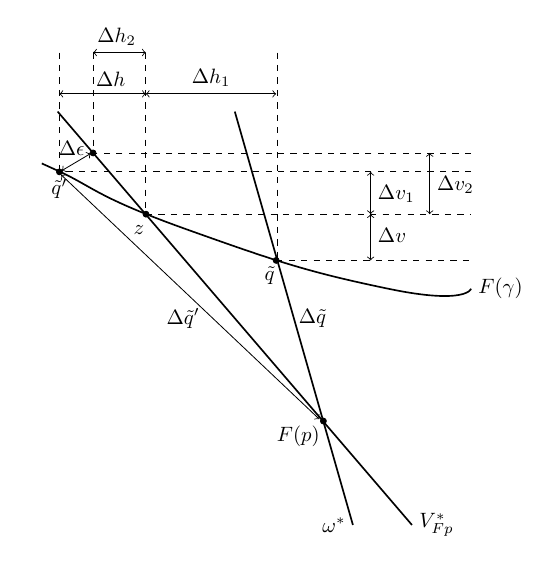}
\caption{Notation for the proof of Proposition \ref{tgV}}
\end{figure}
\begin{prop}\label{tgV}
For $\delta>0$ and $\eta>0$ small enough, each point $p\in \omega^*$ has a tangent plane $T_p\omega^*=V^*_p$. 
\end{prop}
\begin{proof}
One has to show that $A=0$. We use the notation from the definition of $A$ and we introduce the following.%see Figure \ref{FigP} and Figure \ref{FigI}.
\begin{eqnarray*}
F(p)&=&(\tilde y, \tilde b),\\
F\left(\gamma\right)\cap\omega^*&=&\tilde q=F(p)+\Delta\tilde q=F(p)+\left(\Delta\tilde y,\Delta\tilde b\right),\\ F\left(\gamma\right)\cap V^*_{F(p)}&=&z=F(p)+\Delta z,\\
F\left(q'\right)&=&\tilde q'=F(p)+\Delta\tilde q'=F(p)+\left(\Delta\tilde y',\Delta\tilde b'\right),\\
z-\tilde q&=&\left(\Delta h_1,\Delta v\right),\\
\tilde q'-z&=&\left(\Delta h,\Delta v_1\right),\\
\Delta\tilde q'&=&DF_p\left(\Delta q'\right)+\Delta\epsilon,\\
DF_p\left(\Delta q'\right)-\Delta z&=&\left(\Delta h_2,\Delta v_2\right).
\end{eqnarray*} 
For curves $\gamma\in\Gamma$ close enough to $p$ and by Lemma \ref{lemma1}, the differentiability of $F$, Lemma \ref{yexpansion}, $V^*_{F(p)}$ is admissible, $\gamma\in\Gamma$ and $F(\gamma)\in\Gamma$, we get
\begin{eqnarray}\label{L1}
|\Delta\tilde q|&=&O\left(\frac{y}{\tilde y^{\kappa}}\right)|\Delta q|,\\\label{L2}
|\Delta\epsilon|&=&o\left(|\Delta q'|\right),\\ 
\label{L3}
|\Delta h_1|+|\Delta h|&\geq&\frac{1}{3y}|\Delta y-\Delta y'|,\\
\label{L7}
|\Delta h_2|&\leq& 2\chi\tilde y\|\Delta v_2\|_0,\\
\label{L4}
\frac{1}{2}\theta y^{\kappa}\|\Delta b'-\Delta b\|_0 &\leq& |\Delta y-\Delta y'|,\\
\label{L5}
\frac{1}{2}\theta \tilde y^{\kappa}\|\Delta v_1\|_0 &\leq& |\Delta h|,\\
\label{L6}
\|\Delta v_2-\Delta v_1\|_0 &\leq& |\Delta\epsilon |.
\end{eqnarray} 
By (\ref{L7}), (\ref{L6}) and (\ref{L5}) we have 
$$|\Delta h|\leq |\Delta\epsilon|+2\chi \tilde y\left(\frac{2}{\theta\tilde y^{\kappa}}|\Delta h|+|\Delta\epsilon|\right).$$ Hence,  for $\delta>0$ small enough,
\begin{equation}\label{L8}
|\Delta h|\leq 2|\Delta\epsilon|.
\end{equation}
By (\ref{L2}) and  (\ref{L4}), we have 
\begin{equation*}
|\Delta\epsilon|=o\left(|\Delta q|+4\left(1+\frac{1}{\theta y^{\kappa}}\right)|\Delta y-\Delta y'|\right)=o\left(|\Delta q|\left(1+4\left(1+\frac{1}{\theta y^{\kappa}}\right) A\right)\right).
\end{equation*}
Hence 
\begin{equation}\label{L9}
|\Delta\epsilon|=o\left(|\Delta q|\right).
\end{equation}
By (\ref{L3}), (\ref{L1}), (\ref{L8}) and (\ref{L9}) we get
$$
\frac{|\Delta y-\Delta y'|}{|\Delta q|}\leq O\left(3\frac{y^2}{\tilde y^{\kappa}}\right)\frac{|\Delta h_1|}{|\Delta\tilde q|}+6y\frac{|\Delta\epsilon|}{|\Delta q|}=\leq O\left(3\frac{y^2}{\tilde y^{\kappa}}\right)\frac{|\Delta h_1|}{|\Delta\tilde q|}+o(1).
$$
As consequence, 
$$
\limsup_{\gamma\to p}\frac{|\Delta y-\Delta y'|}{|\Delta q|}\leq O\left(3\frac{y^2}{\tilde y^{\kappa}}\right) A.
$$
Finally, from (\ref{cond4}) we get $A=0$ for $\eta$ small enough.
\end{proof}
\begin{lem}\label{pullbacklemma}
Let  $W\subset\R\times B_1$ with the following properties:
\begin{itemize}
\item[-] $W$ is a $\Cuno$ codimension one manifold,
\item[-] for all $p\in W$, the tangent space $T_pW$ extends to a plane in $\R\times B_0$ also denoted by $T_pW$,
\item[-] the dependence of $p$ to $T_pW$ is continuous.
\end{itemize}
Let $V$ be a normed vector space and $H:V\to\R\times B_1$ be such that
\begin{itemize}
\item[-] $H:V\to\R\times B_0$ is continuously differentiable,
\item[-] if $H(v)\in W$, then $DH_v\pitchfork T_{H(v)}W$.
\end{itemize}
Then $H^{-1}(W)$ is a $\Cuno$ codimension one manifold.
\end{lem}
\begin{proof}
Let $v\in H^{-1}(W)$ and $p=H(v)$. Observe that $DH^{-1}_v(T_{p}W)$ is a codimension one subspace which depends continuously on $v$. We denote it by $T_v H^{-1}(W)$. By a similar argument as in the proof of Proposition \ref{tgV}, it follows that $T_v H^{-1}(W)$ is the tangent space at $v$ to $H^{-1}(W)$. As consequence $H^{-1}(W)$ is a $\Cuno$ codimension one manifold.
\end{proof}

\begin{rem} Observe that Lemma \ref{pullbacklemma} is not the usual pull back lemma. Namely, $H:V\to\R\times B_0$ is continuously differentiable but $W\subset R\times B_0$ has infinite codimension. It has codimension one  in $\R\times B_1$.
\end{rem}

\section{Renormalization of Cherry dynamics}
\label{section:renorm}
This section presents the dynamical systems of interest, namely the circle maps with a flat interval and the action of the renormalization operator.  Because of their close connection with Cherry flows, we call circle maps with a flat interval, Cherry maps.
\subsection{The class of functions}
We fix $1<\ell<2$ and we denote by $\Sigma^{(X)}$ the simplex
\begin{equation*}
\Sigma^{(X)}=\{(x_1,x_2,x_3,x_4,s)\in\mathbb R^5 | x_1<0<x_3<x_4<1, 0<x_2<1\text{ and } 0<s<1 \},
\end{equation*}
by $\text{ Diff }^r([0,1])$ the space of $\Cr$, $r\ge 2$, orientation preserving diffeomorphisms of $[0,1]$.
The space of $\Cr$ circle maps with a flat interval is denoted by 
$$\L^{(X,r)}=\Sigma^{(X)}\times \text{Diff }^r([0,1])\times \text{Diff }^r([0,1])\times \text{Diff }^r([0,1]).$$ 
The space $\L^{(X,r)}$ is equipped by a distance which is the sum of the usual distances: the euclidian distance on $\Sigma^{(X)}$ 
and the sum of the $\Cr$ distance on $\text{Diff }^r([0,1])$. Usually we will suppress the index indicating the smoothness of the maps and simply use the notation $\L^{(X)}$. A point $f=(x_1,x_2,x_3,x_4, s,\varphi,\varphi^{l},\varphi^{r})\in \L^{(X)}$ represents the following interval map $f:[x_1,1]\to[x_1,1]$:

\begin{center}
\begin{equation}\label{eqfun}
f(x):=\left\{ 
    \begin{aligned}   
   &   (1-x_2)q_s\circ\varphi\left(1-\frac{x}{x_1}\right)+x_2 & \text{ if } x\in[x_1,0) \\
   &   x_1\left(\varphi^{l}\left(\frac{x_3-x}{x_3}\right)\right)^{\ell} & \text{ if } x\in[0,x_3] \\
   &    0 & \text{ if } x\in(x_3,x_4)\\
&x_2\left(\varphi^{r}\left(\frac{x-x_4}{1-x_4}\right)\right)^{\ell} & \text{ if } x\in[x_4,1] \\       
    \end{aligned}
\right.
\end{equation}
\end{center}
and $q_s:[0,1]\to [0,1]$ is a diffeomorphic part of $x^{\ell}$ parametrized by $s\in (0,1)$, namely $$q_s(x)=\frac{\left[(1-s)x+s\right]^{\ell}-s^{\ell}}{1-s^{\ell}}.$$
The meaning of the parts of $f$ is illustrated in figures \ref{Fig1}, \ref{Fig1.1}, \ref{Fig1.2} and \ref{Fig1.3}. The role of $q_s$ will become clear in the study of the asymptotical behavior of the renormalization operator, see Section \ref{asymrenormalization}.
\begin{figure}[h]
\centering
\includegraphics[width=0.5\textwidth]{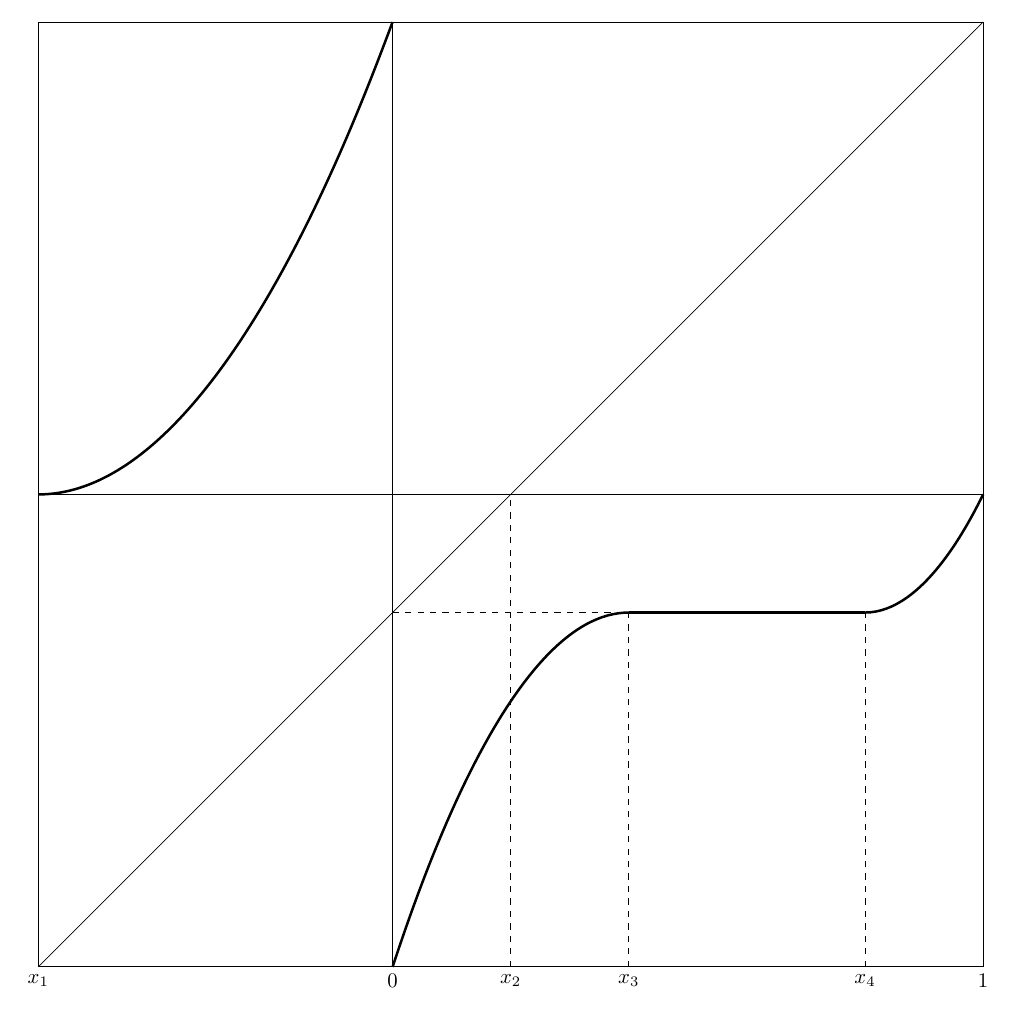}
\caption{A function in $\L^{(X)}$}
\label{Fig1}
\end{figure}

\begin{figure}[h]
\centering
\includegraphics[width=0.3\textwidth]{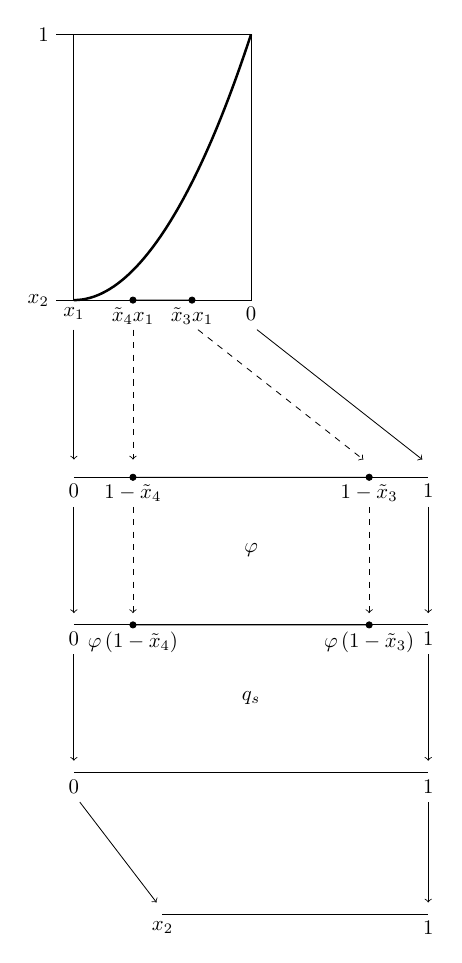}
\caption{The left branch of a function in $\L^{(X)}$}
\label{Fig1.1}
\end{figure}
\begin{figure}[h]
\centering
\includegraphics[width=0.3\textwidth]{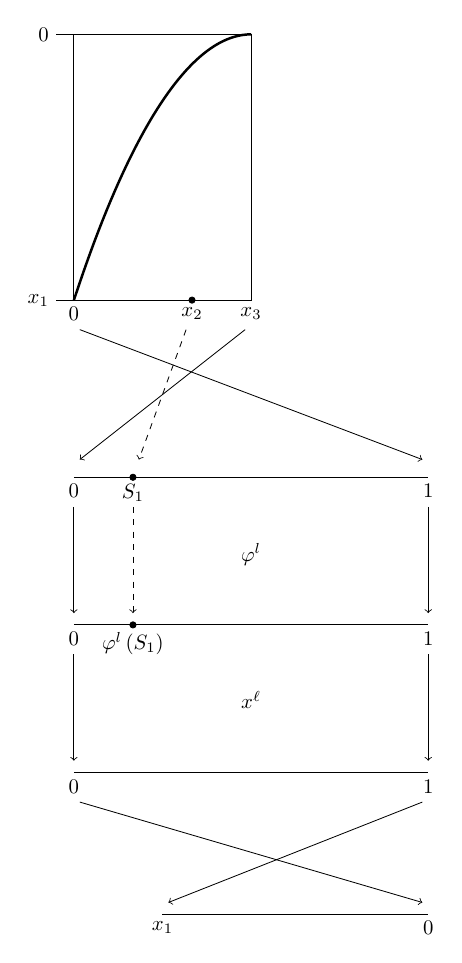}
\caption{The central branch of a function in $\L^{(X)}$}
\label{Fig1.2}
\end{figure}
\begin{figure}[h]
\centering
\includegraphics[width=0.3\textwidth]{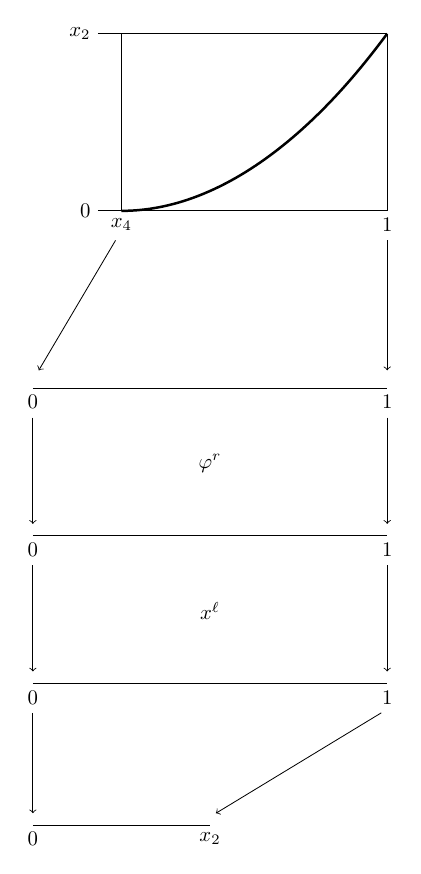}
\caption{The right branch of a function in $\L^{(X)}$}
\label{Fig1.3}
\end{figure}

Depending on the situation, we use different coordinate systems. Given a system $f=(x_1,x_2,x_3,x_4, s, \varphi,\varphi^{l},\varphi^{r})\in \L^{(X)}$ we represent it in $S-$coordinates as follows: $f=(S_1,S_2,S_3,S_4, S_5, \varphi,\varphi^{l},\varphi^{r})$ where
$$
  \begin{aligned}
 &S_1=\frac{x_3-x_2}{x_3}, & S_2=\frac{1-x_4}{1-x_2}, & &S_3=\frac{x_3}{1-x_4}, && S_4=\frac{x_2}{-x_1}, && S_5=s^{\ell-1}.
\end{aligned}
$$
As a consequence, we define 
\begin{equation*}
\Sigma^{(S)}=\{(S_1,S_2,S_3,S_4,S_5)\in\mathbb R^5 | 0<S_2, 0<S_3, 0<S_4\text{ and } 0<S_5<1 \},
\end{equation*}
and 
$$\L^{(S)}=\Sigma^{(S)}\times \text{Diff }^r([0,1])\times \text{Diff }^r([0,1])\times \text{Diff }^r([0,1]).$$ 
Similarly, given a system $f=(S_1,S_2,S_3,S_4, S_5, \varphi,\varphi^{l},\varphi^{r})\in \L^{(S)}$ we represent it in $Y-$coordinates as follows: $f=(y_1,y_2,y_3,y_4, y_5, \varphi,\varphi^{l},\varphi^{r})$ where
$$
  \begin{aligned}
& y_1=S_1, & &y_2=\log S_2, && y_3=\log S_3, && y_4=\log S_4, & y_5=\log S_5.
\end{aligned}
$$
As a consequence, we define 
\begin{equation*}
\Sigma^{(Y)}=\{(y_1,y_2,y_3,y_4,y_5)\in\mathbb R^5 |  y_5<0 \},
\end{equation*}
and 
$$\L^{(Y)}=\Sigma^{(Y)}\times \text{Diff }^r([0,1])\times \text{Diff }^r([0,1])\times \text{Diff }^r([0,1]).$$ 
Observe that these coordinates changes induce diffeomorphisms between $\L^{(X)}$, $\L^{(S)}$ and $\L^{(Y)}$. In particular by explicit calculations the following lemma holds. 
\begin{lem}\label{xtos}
The inverse of $(x_1, x_2, x_3, x_4)\to (S_1, S_2, S_3, S_4)$ is given by
\begin{enumerate}
 \item $x_1=-\frac{S_3(1-S_1)S_2}{(1+S_3(1-S_1)S_2)S_4}$,
 \item $x_2=\frac{S_3(1-S_1)S_2}{1+S_3(1-S_1)S_2}$,
 \item $x_3=\frac{S_3S_2}{1+S_3(1-S_1)S_2}$,
  \item $x_4=1-\frac{S_2}{1+S_3(1-S_1)S_2}$.
\end{enumerate}
\end{lem}

From the context and the notation it will be clear which parametrization of our space we are using. The space will then be simply denoted by $\L$ instead of $\L^{(X)}$, $\L^{(S)}$ or $\L^{(Y)}$. If we want to specify that our maps are $\mathcal C^r$ smooth, then we use the notation $\L^{r}$. When needed we use the notation $x_3(f)$ to denote the $x_3$ coordinate of $f$. Similarly for all others.
\subsection{Renormalization}
In this section we define the renormalization operator. The renormalization scheme that we use is adapted to study circle maps with Fibonacci rotation number. For basic concepts concerning circle maps, see \cite{H79}.
\begin{defin}
A map $f\in\L$ is renormalizable if $0<x_2<x_3$. The space of renormalizable maps is denoted by $\L_0$.
\end{defin}
Let $f\in\L_0$ and let $\text{pre}R(f)$ be the first return map of $f$ to the interval $[x_1,x_2]$. Let us consider the function $h:[x_1,x_2]\to[0,1]$ 
defined as $h(x)={x}/{x_1}$ for all $x\in [x_1,x_2]$. Then the function 
\begin{equation*}
Rf:=h\circ \text{pre}R(f)\circ h^{-1}
\end{equation*}
is again a map in $\L$. Notice that $Rf$ is nothing else than the first return map of $f$ to the interval $[x_1,x_2]$ rescaled and flipped. This define the renormalization operator 
$$R:\L_0\to \L.$$ 
\begin{defin}
A map $f\in\L$ is $\infty$-renormalizable if for every $n\geq 0$, $R^n f\in\L_0$. The set of $\infty$-renormalizable functions will be denoted by $\mathscr W\subset\L$. The maps in $\W$ are called Fibonacci maps. 
\end{defin}
\begin{rem}
Observe that, if $f\in\W$, by identifying $x_1$ with $1$ we obtain a map of the circle having Fibonacci rotation number.
\end{rem}
In Defintion \ref{zoomop} we introduce the concept of the zoom operator needed later to describe the action of the renormalization operator on the space of diffeomorphisms.
\begin{defin}\label{zoomop}
Let $I=[a,b]\subset[0,1]$. The \emph{zoom} operator $Z_{I}:\text{ Diff }^0([0,1])\to \text {Diff }^0([0,1])$ is defined as
$$
Z_{I}\varphi(x)=\frac{\varphi((b-a)x+a)-\tilde a}{\tilde b-\tilde a}
$$
where $\varphi\in \text{ Diff }^0([0,1])$, $x\in[0,1]$, $\tilde a=\varphi(a)$ and $\tilde b=\varphi(b)$.
\end{defin} 
 The following two lemmas are a direct consequence of the definition of the renormalization operator.
\begin{lem}\label{changexs}
Let $f=(x_1,x_2,x_3,x_4, s, \varphi,\varphi^{l},\varphi^{r})\in \L_0$ and let\\
$Rf=(\tilde x_1,\tilde x_2,\tilde x_3, \tilde x_4, \tilde s, \tilde \varphi,\tilde\varphi^{l},\tilde\varphi^{r})$. Then
$$
\begin{aligned}
 1.&\text{  }\tilde x_{1}=\frac{x_2}{x_1},\\
 2.&\text{  }\tilde x_{2}=\left(\varphi^{l}\left(\frac{x_3-x_2}{x_3}\right)\right)^{\ell},\\
 3.&\text{  }\tilde x_3=1-\left(\varphi^{-1}\circ q_s^{-1}\right)\left(\frac{x_4-x_2}{1-x_2}\right),\\
  4.&\text{  }\tilde x_4=1-\left(\varphi^{-1}\circ q_s^{-1}\right)\left(\frac{x_3-x_2}{1-x_2}\right),\\
  5.&\text{  }\tilde s=\varphi^{l}\left(\frac{x_3-x_2}{x_3}\right),\\
 6.&\text{  }\tilde \varphi =Z_{\left[\frac{x_3-x_2}{x_3},1\right]}\varphi^{l},\\
   7.&\text{  } \tilde\varphi^{l}=\varphi^{r}\circ Z_{\left[1-\tilde x_3,1\right]}\left( q_s\circ \varphi\right),\\
   8.&\text{  }\tilde \varphi^{r}=Z_{\left[0, \frac{x_3-x_{2}}{x_3}\right]}\varphi^{l}\circ Z_{\left[0,1-\tilde x_4\right]} \left(q_s\circ \varphi\right).
 \end{aligned}
$$
\end{lem}
\begin{lem}\label{ss}
Let $f=(S_1,S_2,S_3,S_4, S_5, \varphi,\varphi^{l},\varphi^{r})\in \L_0$ and \\
$Rf=(\tilde S_1, \tilde S_2, \tilde S_3, \tilde S_4, \tilde S_5, \tilde \varphi,\tilde\varphi^{l},\tilde\varphi^{r})$. Then
\begin{eqnarray*}
1.&\tilde S_{1}&=1-\left(\frac{\ell S_{1}^{\ell}}{S_{2}}\right)\cdot\left[\frac{S_2}{1-\left(\varphi^{-1}\circ q_s^{-1}\right)\left(1- S_2\right)}\cdot\frac{\left(\varphi^{l}\left( S_1\right)\right)^{\ell}}{\ell S_{1}^{\ell}}\right],\\
2.&  \tilde S_{2}&=\frac{ S_{1}S_{2}S_{3}}{\ell S_{5}}\cdot\left[\frac{\ell S_5 \left(\varphi^{-1}\circ q_s^{-1}\right)\left(S_{1}S_{2}S_{3}\right)}{S_{1}S_{2}S_{3}}\cdot\frac{1}{1-\left(\varphi^{l}\left( S_1\right)\right)^{\ell}}\right],\\
3.&  \tilde S_{3}&=\frac{S_{5}}{ S_{1}S_{3}}\cdot\left[\frac{S_{1}S_{3} \left(1-\left(\varphi^{-1}\circ q_s^{-1}\right)\left(1-S_{2}\right)\right)}{S_5 \left(\varphi^{-1}\circ q_s^{-1}\right)\left(S_{1}S_{2}S_{3}\right)}\right],\\
 4.& \tilde S_{4}&=\frac{S_{1}^{\ell}}{S_{4}}\cdot\left[\left(\frac{\varphi^{l}\left( S_1\right)}{S_{1}}\right)^{\ell}\right],\\
5.&  \tilde S_{5}&=S_{1}^{\ell-1}\cdot\left[\left(\frac{\varphi^{l}\left( S_1\right)}{S_{1}}\right)^{\ell -1}\right],\\
6.&\tilde \varphi &=Z_{\left[S_1,1\right]}\varphi^{l},\\
7.& \tilde\varphi^{l}&=\varphi^{r}\circ Z_{\left[q_s^{-1}\left(1-S_{2}\right),1\right]}\left( q_s\right)\circ Z_{\left[\varphi^{-1}\circ q_s^{-1}\left(1-S_{2}\right),1\right]}\left( \varphi\right),\\
%7.& \tilde\varphi^{l}&=\varphi^{r}\circ Z_{\left[\varphi^{-1}\circ q_s^{-1}\left(1-S_{2}\right),1\right]}\left( q_s\circ \varphi\right)\\
8.& \tilde\varphi^{r}&=Z_{\left[0, S_1\right]}\left(\varphi^{l}\right)\circ Z_{\left[0,q_s^{-1}\left(S_1S_{2}S_3\right)\right]} \left(q_s\right)\circ Z_{\left[0,\varphi^{-1}\circ q_s^{-1}\left(S_1S_{2}S_3\right)\right]} \left(\varphi\right).
%8.& \tilde\varphi^{r}&=Z_{\left[0, S_1\right]}\varphi^{l}\circ Z_{\left[0,\varphi^{-1}\circ q_s^{-1}\left(S_1S_{2}S_3\right)\right]} \left(q_s\circ \varphi\right)
\end{eqnarray*}
\end{lem}
\subsection{Fibonacci rotation number}\label{Fib}
In the sequel we fix the critical exponent $1<\ell<2$ and we use the following notation. If $f\in\W$ the coordinates of $R^n(f)$ are indicated with a lower index $n$, for example $$x_{3,n}=x_{3}(R^n f),$$ and $S_{2,n}=S_{2}(R^n f)$. Similarly for the other coordinates.

Moreover, let $U_f=[x_3, x_4]$ be the flat interval of $f$. Observe that $R^nf_{|[0,1]}$ is a rescaled version of $f^{q_n}$ where the sequence $\left(q_n\right)_{n\in\N}$ is the Fibonacci sequence satisfying that: $q_1=1$, $q_2=2$ and for all $n\geq 3$, $q_n=q_{n-1}+q_{n-2}$. Observe that if $R^nf=(x_{1,n},x_{2,n},x_{3,n},x_{4,n}, s_n,\varphi_{n},\varphi_{n}^l,\varphi_{n}^r)$ then the points $x_{1,n},x_{2,n},x_{3,n},x_{4,n}$ correspond to dynamical points of the original function $f$. Namely, 
\begin{itemize}
\item[-] $\hat x_{1,n}=f^{q_n+1}(x_3)=f^{q_n+1}(x_4)=f^{q_n}(0)$,
\item[-]  $\hat x_{2,n}=f^{q_{n+1}+1}(x_3)=f^{q_{n+1}+1}(x_4)=f^{q_{n+1}}(0)$,
\item[-]  $\hat x_{3,n}=f^{-q_{n}+1}(x_3)$,
\item[-]  $\hat x_{4,n}=f^{-q_{n}+1}(x_4)$.
\end{itemize}

\section{The asymptotics of renormalization}\label{asymrenormalization}
This section explores the asymptotic behaviour of the renormalization operator. 
Let $f\in \W\subset\L^{(Y)}$. For all $n\in\N$ we define 
$$w_n(f)=\left(\begin{matrix}
y_2(R^nf)\\y_3(R^nf)\\y_4(R^nf)\\y_5(R^nf)
\end{matrix}
\right),$$
$\varphi_{n}(f)=\varphi(R^nf)$, $\varphi^l_{n}(f)=\varphi^l(R^nf)$ and $\varphi^r_{n}(f)=\varphi^r(R^nf)$.
\begin{defin}
Let $\varphi: N\to N$ be a $\Cuno$ map where $N$ is an interval. If $T\subset N$ is an interval such that $D\varphi(x)\neq 0$ for every $x\in T$, we define the \emph{distortion} of $\varphi$ in $T$ as:
$$\text{dist}(\varphi,T)=\sup_{x,y\in T}\log\frac{|D\varphi(x)|}{|D\varphi(y)|}.$$
Here $|D\varphi(x)|$ denoted the norm or absolute value of the derivative of $\varphi$ in $x$.
\end{defin}
\begin{prop}\label{superformula}
Let $1<\ell<2$. Then there exist $\lambda_u>1$, $|\lambda_s|<1$, $E_u,E_s, E_{-1},w_{fix}\in\R^4$ such that the following holds. Given $f\in \W\subset\L^{(Y)}$ with critical exponent $\ell$ then there exist $C_u(f)<0$, $C_s(f)$ and $C_{-}(f)$ such that for all $n\in\N$ 
$$w_n(f)=C_u(f)\lambda_u^nE_u+C_s(f)\lambda_s^nE_s+C_-(f)(-1)^nE_-+O\left(e^{\frac{C_u(f)\lambda_u^{n-4}}{\ell}}\right)+w_{fix},$$ 
and 
$$\begin{matrix}
  \text{dist}(\varphi_{n}(f))&=&O\left(e^{\frac{C_u(f)\lambda_u^{n-3}}{\ell}}\right),\\  \text{dist}(\varphi_{n}^{l}(f))&=&O\left(e^{\frac{C_u(f)\lambda_u^{n-2}}{\ell}}\right), \\ 
   \text{dist}(\varphi_{n}^{r}(f))&=&O\left(e^{\frac{C_u(f)\lambda_u^{n-1}}{\ell}}\right).
  \end{matrix}
$$
\end{prop}
The rest of the section is devoted to prove this proposition. In particular we show that 
\begin{enumerate}
\item $\lambda_s=\frac{\frac{1}{\ell}-\sqrt{\left(\frac{1}{\ell}\right)^2+\frac{4}{\ell}}}{2}\in(-1,0),$
\item $\lambda_u=\frac{\frac{1}{\ell}+\sqrt{\left(\frac{1}{\ell}\right)^2+\frac{4}{\ell}}}{2}>1,$
\item $E_u=\left(1,\frac{-\lambda_u+\ell-1}{\ell\lambda_u(1+\lambda_u)},\frac{1}{1+\lambda_u},\frac{\ell-1}{\ell\lambda_u}\right),$
\item $E_s=\left(1,\frac{-\lambda_s+\ell-1}{\ell\lambda_s(1+\lambda_s)},\frac{1}{1+\lambda_s},\frac{\ell-1}{\ell\lambda_s}\right),$ 
\item $E_{-}=(0,0,1,0).$
\end{enumerate}

\subsection{Asymptotics of the scaling ratio}
We define the following sequence of scaling ratio which plays a main role in the rest of the paper. Let $f\in\W$, the scaling ratio $\alpha_n$ is
$$\alpha_n:=\frac{\left|\left[0, x_{3,n}\right]\right|}{\left|\left[0, x_{4,n}\right]\right|}=\frac{x_{3,n} }{x_{4,n}}=\frac{S_{2,n}S_{3,n}}{1-S_{2,n}+\left(1-S_{1,n}\right)S_{2,n}S_{3,n}}.$$

\begin{theo}\label{alpha}
For every $f\in\W$, there exists $\alpha_0<1$ such that $$\alpha_n=O\left(\alpha_0^{\left(\frac{2}{\ell}\right)^{\left(\frac{n}{2}\right)}}\right).$$
\end{theo}
\begin{proof}
Let $f\in\W\subset\L^r$, $r\geq 2$ and $\pi:\R\to\S$ the natural projection with period $1-x_1(f)$. Consider the lift of $f$, $F:\R\to\R$ such that $F(0)=x_1(f)$. Then, there exists $H:\R\to\R$ which is a lift of a circle homeomorphism and satisfies the following properties:
\begin{itemize}
\item[-] $H(x+1-x_1)=H(x)+1-x_1$,
\item[-] $H_{|[0,1]}=id$,
\item[-] $H:[x_1,0]\to [x_1,0]$ is a polynomial diffeomorphism,
\item[-] $G=H\circ F\circ H^{-1}$ is a $\Cr$ map except in $\pi^{-1}(x_3(f))$ and $\pi^{-1}(x_4(f))$.
\end{itemize} 
Observe that close to $\pi^{-1}(x_3(f))$ and $\pi^{-1}(x_4(f))$, $G$ has the form of $x^{\ell}$ up to a $\Cr$ coordinate change. Let $g:\S\to\S$ be the projection of $G$. Then $g$ belongs to the class of circle maps studied in \cite{5aut, P1, P2} and that the theorem is valid for $g$ (see Appendix in \cite{P1}). Because $H_{|(0,1)\cup (x_1,0)}$ is a diffeomorpism and $H$ is a conjugacy between $F$ and $G$ we get ${\alpha_n(f)}/{\alpha_n(g)}\to 1$ when $n$ goes to infinity. The theorem follows.  
\end{proof}
\begin{rem}\label{prop2}
Observe that, by Proposition 2 in \cite{5aut}, for every $f\in\W$ and for $n\in\N$ large enough, 
$$
\frac{x_{3,n}}{x_{4,n}-x_{3,n}}\leq 2\alpha_n.
$$
\end{rem}
\subsection{Asymptotics of the diffeomorphisms}
In this section we show that the distortion of the diffeomorphic parts of the renormalizations behave as the sequence $\alpha_n$. As a consequence, the diffeomorphisms tend to identity double exponentially fast.
\begin{prop}\label{affdiffeo}
Let $f\in \W$, then for all $n\in\N$,
$$\begin{matrix}
   \text{dist}(\varphi_{n})=O\left(\alpha_{n-2}^{\frac{1}{\ell}}\right), & \text{dist}(\varphi_{n}^{l})=O\left(\alpha_{n-1}^{\frac{1}{\ell}}\right),& 
   \text{dist}(\varphi_{n}^{r})=O\left(\alpha_{n}^{\frac{1}{\ell}}\right).
  \end{matrix}
$$
\end{prop}
The following Proposition is a preparation for proving Proposition \ref{affdiffeo}.
\begin{prop}\label{Koebe}
Let $f\in\W$. There exists a constant $K>0$ such that for every $\alpha>0$ the following holds. Let $T$ and $M\subset T$ be two intervals and let $S, D$ be the left and the right component of $T\setminus M$ and $n\in\N$. Suppose that
\begin{enumerate}
\item for every $0\leq i\leq n-1$ the intervals $f^i(T)$ are pairwise disjoint,
\item $f^n:T\to f^n(T)$ is a diffeomorphism,
\item ${\left|f^n(M)\right|}/{\left|f^n(S)\right|},{\left|f^n(M)\right|}/{\left|f^n(D)\right|}<\alpha$.
\end{enumerate}
Then 
$$
\text{dist}\left(f^n(M)\right)\leq K\alpha.
$$
\end{prop}
The proof of the previous proposition can be found in \cite{dmvs}.
We are now ready to prove Proposition \ref{affdiffeo}.
\begin{proof} In this proof we use the notation introduced in Subsection \ref{Fib}.
We start by proving the statement for $\varphi_n^l$. We define 
\begin{itemize}
\item[-] $T=\left[\hat x_{4,n+1},\hat x_{4,n}\right]$,
\item[-] $M=\left[f(x_{3}),\hat x_{3,n}\right]=\left[0,\hat x_{3,n}\right]$,
\item[-] $S=\left[\hat x_{4,n+1}, f(x_{3})\right]=\left[\hat x_{4,n+1}, 0\right]$,
\item[-] $D=\left[\hat x_{3,n},\hat x_{4,n}\right]$.
\end{itemize}
Observe that $$\varphi_{n}^l=Z_{M}f^{q_n-1}.$$ We claim that:
\begin{enumerate}
\item for every $0\leq i\leq q_n-2$ the intervals $f^i(T)$ are pairwise disjoint,
\item $f^{q_n-1}:T\to f^{q_n-1}(T)$ is a diffeomorphism,
\item ${\left|f^{q_n-1}(M)\right|}/{\left|f^{q_n-1}(S)\right|},{\left|f^{q_n-1}(M)\right|}/{\left|f^{q_n-1}(D)\right|}=O\left(\alpha_{n-1}^{\frac{1}{\ell}}\right)$.
\end{enumerate}
Points $1$ and $2$ comes from general properties of circle maps. For point $3$, observe that 
\begin{eqnarray*}
\frac{\left|f^{q_n-1}(M)\right|}{\left|f^{q_n-1}(S)\right|}=\frac{\left|\left[f^{q_n}(x_3),x_3\right]\right|}{\left|\left[f^{-q_{n-1}}(x_4),f^{q_n}(x_3)\right]\right|}.
\end{eqnarray*}
As consequence, under the image of $f$,
\begin{eqnarray*}
\frac{\left|f^{q_n}(M)\right|}{\left|f^{q_n}(S)\right|}=\frac{\left|\left[\hat x_{2,n-1}, 0\right]\right|}{\left|\left[\hat x_{4,n-1},\hat x_{2,n-1}\right]\right|}=O\left(\frac{\left|\left[\hat x_{3,n-1}, 0\right]\right|}{\left|\left[\hat x_{4,n-1},\hat x_{3,n-1}\right]\right|}\right)=O\left(\alpha_{n-1}\right),
\end{eqnarray*}
where we also used Remark \ref{prop2}. Hence, because of the form of the map near to the boundary points of the flat interval 
\begin{eqnarray*}
\frac{\left|f^{q_n-1}(M)\right|}{\left|f^{q_n-1}(S)\right|}=O\left(\alpha_{n-1}^{\frac{1}{\ell}}\right).
\end{eqnarray*}
Observe that, for $n$ large enough, 
\begin{eqnarray*}
\frac{\left|f^{q_n-1}(M)\right|}{\left|f^{q_n-1}(D)\right|}=\frac{\left|\left[f^{q_n}(x_3),x_3\right]\right|}{\left|\left[x_3, x_4\right]\right|}\leq \frac{\left|f^{q_n-1}(M)\right|}{\left|f^{q_n-1}(S)\right|}=O\left(\alpha_{n-1}^{\frac{1}{\ell}}\right).
\end{eqnarray*}
By Proposition \ref{Koebe} we get the desired distortion estimate for $\varphi_n^l$.
\\
For the distortion estimate of $\varphi_n$, it is enough to repeat the previous argument for 
\begin{itemize}
\item[-] $T=\left[\hat x_{4,n-1},\hat x_{4,n}\right]$,
\item[-] $M=\left[\hat x_{1,n}, f(x_{3})\right]=\left[\hat x_{1,n}, 0\right]$,
\item[-] $S=\left[\hat x_{4,n-1}, \hat x_{1,n}\right]$,
\item[-] $D=\left[f(x_{3}),\hat x_{4,n}\right]=\left[0, \hat x_{4,n}\right]$,
\end{itemize}
and to notice that $\varphi_n=Z_{M}f^{q_{n-1}-1}$.
\\
For the distortion estimate of $\varphi_n^r$, take
\begin{itemize}
\item[-] $T=\left[\hat x_{3,n},\hat x_{3,n-2}\right]$,
\item[-] $M=\left[\hat x_{4,n}, \hat x_{2,n-2}\right]$,
\item[-] $S=\left[\hat x_{3,n}, \hat x_{4,n}\right]$,
\item[-] $D=\left[\hat x_{2,n-2},\hat x_{3,n-2}\right]$,
\end{itemize}
and notice that $\varphi_n^r=Z_{M}f^{q_{n}-1}$. We claim that:
\begin{enumerate}
\item for every $0\leq i\leq q_n-2$ the intervals $f^i(T)$ are pairwise disjoint,
\item $f^{q_n-1}:T\to f^{q_n-1}(T)$ is a diffeomorphism,
\item ${\left|f^{q_n-1}(M)\right|}/{\left|f^{q_n-1}(S)\right|},{\left|f^{q_n-1}(M)\right|}/{\left|f^{q_n-1}(D)\right|}=O\left(\alpha_{n}^{\frac{1}{\ell}}\right)$.
\end{enumerate}
Points $1$ and $2$ comes from general properties of circle maps. For point $3$, observe that 
\begin{eqnarray*}
\frac{\left| f^{q_n-1}(M)\right|}{\left| f^{q_n-1}(D)\right|}=\frac{\left|\left[x_4, f^{q_{n+1}}(x_3)\right]\right|}{\left|\left[f^{q_{n+1}}(x_3), f^{q_{n-1}}(x_3)\right]\right|}\leq \frac{\left|\left[x_4, f^{q_{n+1}}(x_3)\right]\right|}{\left|\left[f^{-q_{n}}(x_3), f^{-q_{n}}(x_4))\right]\right|}.
\end{eqnarray*}
As consequence, under the image of $f$,
\begin{eqnarray*}
\frac{\left|f^{q_n}(M)\right|}{\left|f^{q_n}(D)\right|}\leq\frac{\left|\left[0,\hat x_{3,n}\right]\right|}{\left|\left[\hat x_{3,n},\hat x_{4,n}\right]\right|}=O\left(\alpha_{n}\right),
\end{eqnarray*}
where we also used Remark \ref{prop2}. Hence 
\begin{eqnarray*}
\frac{\left|f^{q_n-1}(M)\right|}{\left|f^{q_n-1}(D)\right|}=O\left(\alpha_{n}^{\frac{1}{\ell}}\right).
\end{eqnarray*}
Observe that, for $n$ large enough, 
\begin{eqnarray*}
\frac{\left|f^{q_n-1}(M)\right|}{\left|f^{q_n-1}(S)\right|}=\frac{\left|\left[x_4, f^{q_{n+1}}(x_3)\right]\right|}{\left|\left[x_3, x_4\right]\right|}\leq \frac{\left|f^{q_n-1}(M)\right|}{\left|f^{q_n-1}(D)\right|}=O\left(\alpha_{n}^{\frac{1}{\ell}}\right).
\end{eqnarray*}
By Proposition \ref{Koebe} we get the desired distortion estimate for $\varphi_n^r$.
\end{proof}

%We are now ready to prove Proposition \ref{affdiffeo}.We define 
%\begin{eqnarray*}
%\eta_n=|\varphi_n| & \eta^l_n=|\varphi^l_n| & \eta^r_n=|\varphi^r_n| 
%\end{eqnarray*}
%From Lemma \ref{nonlinearitylemma} and from points $6$, $7$, $8$ of Lemma \ref{ss} and Figures $2$, $3$, $4$ we get the following: 
%
%\begin{equation*}
%\left\{
%\begin{matrix}
%\eta_{n+1}&\leq & \eta_n^l\\
%\eta^l_{n+1}&\leq &\eta_n^r+\left| Z_{\left[\varphi_n(1-x_{3,n+1}), 1\right]}q_{s_n}\right|+\left| Z_{\left[1-x_{3,n+1}, 1\right]}\varphi_{n}\right|\\
%\eta^r_{n+1}&\leq &\left| Z_{\left[0, S_{1,n}\right]}\varphi^{\ell}\right|+\left| Z_{\left[0, \varphi_n(1-x_{4,n+1})\right]}q_{s_n}\right|+\left| Z_{\left[0,1-x_{3,n+1}\right]}\varphi_{n}\right| 
%\end{matrix}
%\right\}
%\end{equation*}
%\TDD{remove last parentesis in the previous and in next sets of equations}
%
%From Lemma \ref{prev} we get
%$$
%\left\{
%\begin{matrix}
%\eta_{n+1}&\leq & \eta_n^l\\
%\eta^l_{n+1}&\leq &O\left(\alpha_{n+1}\right)\eta_n+\eta_n^r+O\left(\alpha_{n+1}\right)\\
%\eta^r_{n+1}&\leq &O\left(\alpha_{n+1}\right)\eta_n+O\left(\alpha_{n+1}\right)\eta^l_n+O\left(\alpha_{n+1}^{\frac{1}{\ell}}\right) \end{matrix}
%\right\}
%$$
%By iterating the previous equations three times and by denoting $m_n=\max\left(\eta_n,\eta_n^l,\eta_n^r\right)$ we get
%$$
%m_{n+3}=O\left(\alpha_{n+1}^{\frac{1}{\ell}}\right)m_n+O\left(\alpha_{n+1}^{\frac{1}{\ell}}\right).
%$$
%This implies that the sequence $\left\{m_n\right\}_{n\in\N}$ is bounded. It follows that 
%$$m_n=O\left(\alpha_{n-2}^{\frac{1}{\ell}}\right).$$
%This concludes the proof of Proposition \ref{affdiffeo}.

\subsection{Asymptotic linear behaviour of renormalization}
We are now ready to prove Proposition \ref{superformula}. The proof will be presented in a series of lemmas.
\begin{lem}\label{prev}
Let $f\in \W$, then for all $n\in\N$,
$$
\begin{aligned}
% 1.& \left|Z_{\left[\varphi_n(1-x_{3,n+1}), 1\right]}q_{s_n}\right|&=O\left(\alpha_{n+1}\right)\\
%2.& \left| Z_{\left[0,\varphi_n(1-x_{4,n+1})\right]}q_{s_n}\right| &=O\left(\alpha_{n+1}^{\frac{1}{\ell}}\right)\\
&1.&S_{1,n}&=O\left(\alpha_{n+1}^{\frac{1}{\ell}}\right),\\
&2.&  S_{2,n}&=O\left(\alpha_{n-1}^{\frac{1}{\ell}}\right),\\
&3.&  s_{n}&=O\left(\alpha_{n}^{\frac{1}{\ell}}\right),\\
&4.&S_{1,n}S_{2,n}S_{3,n}&=O\left(\alpha_{n}\right),\\
&5.&S_{2,n}S_{3,n}&=\frac{S_{2,n-1}}{\ell}\left(1+O\left(\alpha_{n-3}^{\frac{1}{\ell}}\right)\right).
%3.& x_{3,n}&\leq \alpha_{n}\\
% 4.& 1-x_{4,n}&=O\left(\alpha_{n-1}^{\frac{1}{\ell}}\right)
\end{aligned}
$$
\end{lem}
\begin{proof}
%The point $3$ is a direct consequence of the definition. We prove point $3$. 
Because $\varphi^l_n$ has bounded distortion (see Proposition \ref{affdiffeo}), we have that $S_{1,n}=O\left(\varphi^l_n\left(S_{1,n}\right)\right)$. Observe that
$$
\left(\varphi^l_n\left(S_{1,n}\right)\right)^{\ell}=x_{2,n+1}.
$$
Hence
$$
S_{1,n}=O\left(x_{2,n+1}^{\frac{1}{\ell}}\right)=O\left(x_{3,n+1}^{\frac{1}{\ell}}\right)=O\left(\alpha_{n+1}^{\frac{1}{\ell}}\right).
$$
Point $1$ follow. In order to s $2$ observe that, by Proposition $2$ in \cite{5aut} and Proposition \ref{affdiffeo}, there exist two constants $K_1$ and $K_2$, such that 
$$
S_{2,n+2}\leq K_1\frac{\left|\left[x_{4,n+2},x_{3,n}\right]\right|}{\left|\left[x_{3,n+2},x_{3,n}\right]\right|}\leq K_1K_2\frac{\left|\left[f^{-q_{n+1}}(x_4),x_{3}\right]\right|}{\left|\left[f^{-q_{n+1}}(x_3),x_{3}\right]\right|}.
$$
Moreover,
$$
\left(\frac{\left|\left[f^{-q_{n+1}}(x_4),x_{3}\right]\right|}{\left|\left[f^{-q_{n+1}}(x_3),x_{3}\right]\right|}\right)^{\ell}=O\left(\frac{x_{3,n+1}}{x_{4,n+1}}\right).
$$
Combining the two previous inequalities, we find

$$
S_{2,n+2}^{\ell}=O\left(\frac{x_{3,n+1}}{x_{4,n+1}}\right)=O\left(\alpha_{n+1}\right).
$$
Point $2$ follows. 
%For proving point $2$ observe that 
%$$
%S_{2,n+2}\leq\frac{x_{3,n}-x_{4,n+2}x_{2,n}}{x_{3,n}-x_{3,n+2}x_{2,n}}.
%$$
%As before, we have
%\begin{eqnarray*}
%S_{2,n+2}^{\ell}&=&O\left(\left[\frac{\varphi^l_n\left(1-\frac{x_{3,n}-x_{4,n+2}x_{2,n}}{x_{3,n}}\right)}{\varphi^l_n\left(1-\frac{x_{3,n}-x_{3,n+2}x_{2,n}}{x_{3,n}}\right)}\right]^{\ell}\right)=O\left(\frac{x_{3,n+1}}{x_{4,n+1}}\right)=O\left(\alpha_{n+1}\right).
%\end{eqnarray*}
%It follows point $2$. %For proving $6$ it is enough to observe that $1-x_{4,n}\leq S_{2,n}$. 
For proving point $3$, we use point $1$ of this lemma, point $5$ of Lemma \ref{ss} and Proposition \ref{affdiffeo}. Namely, 

$$
s_{n}^{\ell-1}=S_{5,n}=S_{1,n-1}^{\ell-1}\cdot\left(\frac{\varphi_n^l\left(S_{1,n-1}\right)}{S_{1,n-1}}\right)^{\ell-1}=O\left(\alpha_{n}^{\frac{\ell-1}{\ell}}\right)\cdot\left(1+O\left(\alpha_{n-2}^{\frac{1}{\ell}}\right)\right)=O\left(\alpha_{n}^{\frac{\ell-1}{\ell}}\right).
$$
Point $3$ follows. 
The proof of point $4$ is a consequence of Remark \ref{prop2}. Namely,
$$
S_{1,n}S_{2,n}S_{3,n}=\frac{x_{3,n}-x_{2,n}}{1-x_{2,n}}\leq\frac{x_{3,n}}{x_{4,n}-x_{3,n}}=O\left(\alpha_n\right).
$$
For point $5$, by points $2$ and $3$ of Lemma \ref{ss}
\begin{eqnarray*}
S_{2,n}S_{3,n}=\frac{S_{2,n-1}}{\ell}\left[\frac{\ell}{S_{2,n-1}}\cdot\frac{1-\left(\varphi_{n-1}^{-1}\circ q_{s_{n-1}}^{-1}\right)\left(1-S_{2,n-1}\right)}{1-\left(\varphi_{n-1}^{l}\left( S_{1,n-1}\right)\right)^{\ell}}\right].
\end{eqnarray*}
Observe that by Proposition \ref{affdiffeo} we get
\begin{equation}\label{1-varphiq}
1-\left(\varphi_{n-1}^{-1}\circ q_{s_{n-1}}^{-1}\right)\left(1-S_{2,n-1}\right)=\frac{S_{2,n-1}}{Dq_{s_{n-1}}\left(\theta_{n-1}\right)\left(1+O\left(\alpha_{n-3}^{\frac{1}{\ell}}\right)\right)},
\end{equation}
where $\left|\theta_n-1\right|\leq S_{2,n}$. 
By points $2$ and $3$ we find
\begin{equation}\label{Dq}
Dq_{s_{n-1}}\left(\theta_{n-1}\right)=\ell\cdot \left(1+O\left(\alpha_{n-2}^{\frac{1}{\ell}}\right)\right).
\end{equation}
By point $1$ and Proposition \ref{affdiffeo} we get
\begin{equation}\label{1-varphil}
\frac{1}{1-\left(\varphi_{n-1}^l\left(S_{1,n-1}\right)\right)^{\ell}}=\left(1+O\left(\alpha_{n}\right)\right).
\end{equation}
Point $5$ follows.
\end{proof}
\begin{prop}\label{ssn}
Let $f\in \W$ and
$R^nf=(S_{1,n}, S_{2,n}, S_{3,n}, S_{4,n}, S_{5,n}, \varphi_n,\varphi^{l}_n,\varphi^{r}_n)$. Then
$$
\begin{aligned}
&1.& S_{1,n+1}&= 1-\left(\frac{\ell S_{1,n}^{\ell}}{S_{2,n}}\right)\left(1+O\left(\alpha_{n-2}^{\frac{1}{\ell}}\right)\right),\\
&2.&  S_{2,n+1}&=\frac{ S_{1,n}S_{2,n}S_{3,n}}{\ell S_{5,n}}\left(1+O\left(\alpha_{n-2}^{\frac{1}{\ell}}\right)\right),\\
&3.&  S_{3,n+1}&=\frac{S_{5,n}}{S_{1,n}S_{3,n}}\left(1+O\left(\alpha_{n-2}^{\frac{1}{\ell}}\right)\right),\\
& 4.& S_{4,n+1}&=\frac{S_{1,n}^{\ell}}{S_{4,n}}\left(1+O\left(\alpha_{n-1}^{\frac{1}{\ell}}\right)\right),\\
&5.&  S_{5,n+1}&= S_{1,n}^{\ell-1}\left(1+O\left(\alpha_{n-1}^{\frac{1}{\ell}}\right)\right).
\end{aligned}
$$
\end{prop}
\begin{proof}
Let us prove point $1$. By (\ref{1-varphiq}) and (\ref{Dq})
\begin{equation}\label{eq:varphis1overs1}
1-\left(\varphi_n^{-1}\circ q_{s_n}^{-1}\right)\left(1-S_{2,n}\right)=\frac{S_{2,n}}{\ell}\left(1+O(\alpha_{n-2}^{\frac{1}{\ell}})\right).
\end{equation}
Finally, by point $1$ of Lemma \ref{ss}, point $1$ of Lemma \ref{prev}, \eqref{eq:varphis1overs1} and Proposition \ref{affdiffeo} we get
\begin{eqnarray*}
\left[\frac{S_{2,n}}{1-\left(\varphi_n^{-1}\circ q_{s_n}^{-1}\right)\left(1- S_{2,n}\right)}\cdot\frac{\left(\varphi_n^{l}\left( S_{1,n}\right)\right)^{\ell}}{\ell S_{1,n}^{\ell}}\right]&=&\\
\ell\left(1+O\left(\alpha_{n-1}^{\frac{1}{\ell}}\right)\right)\left(1+O\left(\alpha_{n-2}^{\frac{1}{\ell}}\right)\right)\cdot\frac{1}{\ell}\left(1+O\left(\alpha_{n-1}^{\frac{1}{\ell}}\right)\right)&=&
1+O\left(\alpha_{n-2}^{\frac{1}{\ell}}\right).
\end{eqnarray*}
Notice that the previous estimate 
$$
\frac{\left(\varphi_n^{l}\left( S_{1,n}\right)\right)^{\ell}}{ S_{1,n}^{\ell}}=\left(1+O\left(\alpha_{n-1}^{\frac{1}{\ell}}\right)\right),
$$
proves point $4$ and $5$ by using Lemma \ref{ss}.
Let us prove point $2$. By Proposition \ref{affdiffeo} we get
$$
\frac{\ell S_{5,n}\left(\varphi_n^{-1}\circ q_{s_n}^{-1}\right)\left(S_{1,n}S_{2,n}S_{3,n}\right)}{S_{1,n}S_{2,n}S_{3,n}}\leq\frac{\ell S_{5,n}}{Dq_{s_n}\left(0\right)\left(1+O\left(\alpha_{n-2}^{\frac{1}{\ell}}\right)\right)}.
$$
By calculations and points $3$ of Lemma \ref{prev} we find
$$
Dq_{s_n}\left(0\right)=\ell \left(1+O\left(\alpha_{n}^{\frac{1}{\ell}}\right)\right)S_{5,n}.
$$
Finally, by point $2$ of Lemma \ref{ss}, by (\ref{1-varphil}) and by the previous two estimates we have
\begin{eqnarray*}
\left[\frac{\ell S_{5,n}\left(\varphi_n^{-1}\circ q_{s_n}^{-1}\right)\left(S_{1,n}S_{2,n}S_{3,n}\right)}{S_{1,n}S_{2,n}S_{3,n}}\cdot\frac{1}{1-\left(\varphi_n^l\left(S_{1,n}\right)\right)^{\ell}}\right]&=&\\
\left(1+O\left(\alpha_{n-2}^{\frac{1}{\ell}}\right)\right)\left(1+O\left(\alpha_{n}^{\frac{1}{\ell}}\right)\right)\cdot\left(1+O\left(\alpha_{n+1}\right)\right)&=&
1+O\left(\alpha_{n-2}^{\frac{1}{\ell}}\right).
\end{eqnarray*}
Notice that point $1$ and $2$ proves that 
\begin{eqnarray}\label{S1S2}
 \frac{\ell S_{1,n}^{\ell}}{S_{2,n}}=\left(1+O\left(\alpha_{n-2}^{\frac{1}{\ell}}\right)\right).
\end{eqnarray}
By point $5$ of Lemma \ref{prev}, (\ref{S1S2}) and point $5$ of this proposition we get 
\begin{equation}\label{S1S2S3}
\frac{S_{1,n}S_{2,n}S_{3,n}}{s_n^{\ell}}=S_{1,n}\left(1+O\left(\alpha_{n-3}^{\frac{1}{\ell}}\right)\right).
\end{equation}
We are now ready to prove point $3$. Notice that 
$$
\left(\varphi_{n}^{-1}\circ q_{s_{n}}^{-1}\right)\left(S_{1,n}S_{2,n}S_{3,n}\right)=\frac{S_{1,n}S_{2,n}S_{3,n}}{Dq_{s_{n}}\left(\zeta_{n}\right)\left(1+O\left(\alpha_{n-2}^{\frac{1}{\ell}}\right)\right)},
$$
where $q_{s_{n}}\left(\zeta_n\right)\leq S_{1,n}S_{2,n}S_{3,n}\left(1+O\left(\alpha_{n-2}^{\frac{1}{\ell}}\right)\right)$. In particular, because $q_{s_{n}}\left(x\right)\geq x^{\ell}$, we get 
\begin{equation}\label{zeta}
0<\zeta_n\leq\left(S_{1,n}S_{2,n}S_{3,n}\right)^{\frac{1}{\ell}}\left(1+O\left(\alpha_{n-2}^{\frac{1}{\ell}}\right)\right).
\end{equation}
By Lemma \ref{ss}, (\ref{1-varphiq}) and the previous estimate we get
$$
S_{3,n+1}=\frac{S_{5,n}}{S_{1,n}S_{3,n}}\left[\frac{1}{S_{5,n}}\cdot\frac{\left(\left(1-s_n\right)\zeta_n+s_n\right)^{\ell-1}}{\left(\left(1-s_n\right)\theta_n+s_n\right)^{\ell-1}}\left(1+O\left(\alpha_{n-2}^{\frac{1}{\ell}}\right)\right)\right].
$$
Now, by point $2$ and $3$ of Lemma\ref{prev} and by the definition of $S_{5,n}$
\begin{eqnarray*}
S_{3,n+1}&=&\frac{S_{5,n}}{S_{1,n}S_{3,n}}\left[\frac{\left(\left(1-s_n\right)\zeta_n+s_n\right)^{\ell-1}}{S_{5,n}}\left(1+O\left(\alpha_{n-2}^{\frac{1}{\ell}}\right)\right)\right]\\&=&\frac{S_{5,n}}{S_{1,n}S_{3,n}}\left[\left(\left(1-s_n\right)\frac{\zeta_n}{s_n}+1\right)^{\ell-1}\left(1+O\left(\alpha_{n-2}^{\frac{1}{\ell}}\right)\right)\right].
\end{eqnarray*}
Finally, by (\ref{zeta}), (\ref{S1S2S3}) and point $1$ of Lemma \ref{prev}, we find
\begin{eqnarray*}
S_{3,n+1}&=&\frac{S_{5,n}}{S_{1,n}S_{3,n}}\left[\left(1+O\left(S_{1,n}^{\frac{1}{\ell}}\right)\right)\left(1+O\left(\alpha_{n-2}^{\frac{1}{\ell}}\right)\right)\right]\\&=&\frac{S_{5,n}}{S_{1,n}S_{3,n}}\left(1+O\left(\alpha_{n-2}^{\frac{1}{\ell}}\right)\right).
\end{eqnarray*}
Point $3$ follows.
\end{proof}
\begin{cor}\label{s1}
Let $f\in \W$ and
$R^nf=(S_{1,n}, S_{2,n}, S_{3,n}, S_{4,n}, S_{5,n}, \varphi_n,\varphi^{l}_n,\varphi^{r}_n)$. Then
$$
\begin{aligned}
&1.& \frac{\ell S_{1,n}^{\ell}}{S_{2,n}}&=\left(1+O\left(\alpha_{n-2}^{\frac{1}{\ell}}\right)\right),\\
&2.& \frac{S_{2,n}S_{3,n}}{s_n^{\ell}}&=\left(1+O\left(\alpha_{n-3}^{\frac{1}{\ell}}\right)\right).
\end{aligned}
$$
\end{cor}
Let $f\in\W$. Recall that for all $n\in\N$ we defined
$$w_n=\left(\begin{matrix}
\log S_{2,n}\\\log S_{3,n}\\\log S_{4,n}\\\log S_{5,n}
\end{matrix}
\right)=\left(\begin{matrix}
y_{2,n}\\ y_{3,n}\\ y_{4,n}\\ y_{5,n}
\end{matrix}
\right).$$
We define now
$$
 M=\left(\begin{matrix} 1+\frac{1}{\ell} & 1& 0& -1\\
 -\frac{1}{\ell} & -1& 0& 1\\
 1 & 0& -1& 0\\
 1-\frac{1}{\ell} & 0& 0& 0
 \end{matrix}\right),
$$
and 
$$w^*=\left(\begin{matrix}
-\left(1+\frac{1}{\ell}\right)\log\ell\\\frac{1}{\ell}\log\ell\\-\log\ell\\ -\left(1-\frac{1}{\ell}\right)\log\ell
\end{matrix}
\right).$$
Point $1$ of Corollary \ref{s1} allows to eliminate $S_{1,n}$ which asymptotically is determined by $S_{2,n}$. With the notations introduced above, 
the new estimates of Proposition \ref{ssn} obtained by the substitution of $S_{1,n}$ takes the following linear form.
\begin{prop}\label{wn}
 Let $f\in \W$, then 
 $$w_{n+1}=Mw_n+O\left(\alpha_{n-2}^{\frac{1}{\ell}}\right)+w^*.$$
\end{prop}

\begin{lem}\label{eighvector}
The eigenvalues of $M$ are
$$
\begin{matrix}
 -1,& 0, &\lambda_s=\frac{\frac{1}{\ell}-\sqrt{\left(\frac{1}{\ell}\right)^2+\frac{4}{\ell}}}{2}\in(-1,0),&
\lambda_u=\frac{\frac{1}{\ell}+\sqrt{\left(\frac{1}{\ell}\right)^2+\frac{4}{\ell}}}{2}>1,
\end{matrix}$$
and the corresponding eigenvectors 
$$
\begin{aligned}
 E_{-}&=&(e^-_i)_{i=2,3,4,5}&=(0,0,1,0),\\
 E_0&=&(e^0_i)_{i=2,3,4,5}&=(0,1,0,1),\\ 
   E_s&=&(e^s_i)_{i=2,3,4,5}&=\left(1,\frac{-\lambda_s+\ell-1}{\ell\lambda_s(1+\lambda_s)},\frac{1}{1+\lambda_s},\frac{\ell-1}{\ell\lambda_s}\right),\\
E_u&=&(e^u_i)_{i=2,3,4,5}&=\left(1,\frac{-\lambda_u+\ell-1}{\ell\lambda_u(1+\lambda_u)},\frac{1}{1+\lambda_u},\frac{\ell-1}{\ell\lambda_u}\right).
\end{aligned}
$$
\end{lem}
Let $w_{fix}$ be the fixed point of the equation $Lw_{fix}+w^*=w_{fix}.$ 

\subsection{Asymptotics in $Y$-coordinates}
\begin{lem}\label{suplemma}
For every $f\in \W$ there exist $C_u(f)$, $C_s(f)$ and $C_{-}(f)$ such that, for all $n\in\N$ ,
\begin{equation}\label{eq:wninalphan}
w_n=C_u(f)\lambda_u^nE_u+C_s(f)\lambda_s^nE_s+C_-(f)(-1)^nE_-+O\left(\left(\alpha_0^{(\frac{2}{\ell})^{\frac{n-3}{2}}}\right)^{\frac{1}{\ell}}\right)+w_{fix}.
\end{equation}
\end{lem}
\begin{proof}
For all $n\in\N$ let $v_n=w_n-w_{fix}$. Then by Proposition \ref{wn}, $$v_{n+1}=Mv_n+\epsilon_n,$$ where 
$\epsilon_n=O\left(\alpha_{n-2}^{\frac{1}{\ell}}\right)=O\left(\left(\alpha_0^{(\frac{2}{\ell})^{\frac{n-2}{2}}}\right)^{\frac{1}{\ell}}\right)$, see Theorem \ref{alpha}. 
By iterating this formula we get
\begin{equation}\label{forvn}
 v_n=M^nv_0+\sum_{k=0}^{n-1}L^{n-k-1}\epsilon_k.
\end{equation}
By expressing $v_0$ and $\epsilon_n$ in the eigenbasis we find the following equalities:
$$v_0=C_u^0E_u+C_s^0E_s+C_-^0E_-+C_0^0E_0,$$
$$\epsilon_n=\epsilon_{u,n}E_u+\epsilon_{s,n}E_s+\epsilon_{-,n}E_-+\epsilon_{0,n}E_0.$$ 
We consider now the following quantities depending on $f$:
\begin{eqnarray}\label{cu}
 C_u(f)&=&C_u^0+\sum_{k=0}^{\infty}\frac{\epsilon_{u,k}}{\lambda_u^{k+1}},\\
 \label{cs}
 C_s(f)&=&C_s^0+\sum_{k=0}^{\infty}\frac{\epsilon_{s,k}}{\lambda_s^{k+1}},\\
 \label{c-}
 C_-(f)&=&C_-^0+\sum_{k=0}^{\infty}\frac{\epsilon_{-,k}}{(-1)^{k+1}}.
\end{eqnarray}
Moreover, equation (\ref{forvn}) in the coordinates becomes
\begin{eqnarray*}
v_n&=&\left(C_u^0+\sum_{k=0}^{n-1}\frac{\epsilon_{u,k}}{\lambda_u^{k+1}}\right)\lambda_u^nE_u+\left(C_s^0+\sum_{k=0}^{n-1}\frac{\epsilon_{s,k}}{\lambda_s^{k+1}}\right)
\lambda_s^nE_s+\\
&&\left(C_-^0+\sum_{k=0}^{n-1}\frac{\epsilon_{-,k}}{(-1)^{k+1}}\right)(-1)^nE_-+
\epsilon_{0, n-1}E_0\\
&=&C_u(f)\lambda_u^nE_u+C_s(f)\lambda_s^nE_s+C_-(f)(-1)^nE_-+C_0^0E_0+
\left(\sum_{k=n}^{\infty}\frac{\epsilon_{u,k}}{\lambda_u^{k+1}}\right)\lambda_u^nE_u
+\\ 
&&\left(\sum_{k=n}^{\infty}\frac{\epsilon_{s,k}}{\lambda_s^{k+1}}\right)\lambda_s^nE_s+
\left(\sum_{k=n}^{\infty}\frac{\epsilon_{-,k}}{(-1)^{k+1}}\right)(-1)^nE_-+
\epsilon_{0, n-1}E_0\\
&=&C_u(f)\lambda_u^nE_u+C_s(f)\lambda_s^nE_s+C_-(f)(-1)^nE_-+O\left(\left(\alpha_0^{(\frac{2}{\ell})^{\frac{n-3}{2}}}\right)^{\frac{1}{\ell}}\right).
\end{eqnarray*}
Notice that the three tail terms were estimated in the following way. Let us start with the tail term corresponding to $E_s$. The others are treated in a similar way.
Notice that for $k$ large enough we have
$$\frac{\left(\alpha_0^{(\frac{2}{\ell})^{\frac{k}{2}}}\right)^{\frac{1}{\ell}}}{\left(\alpha_0^{(\frac{2}{\ell})^{\frac{k-1}{2}}}\right)^{\frac{1}{\ell}}\lambda_s}\leq\frac{1}{{2}}.$$
As a consequence 
\begin{eqnarray*}
\left|\sum_{k=n}^{\infty}\frac{\epsilon_{s,k}}{\lambda_s^{k+1}}\right|\lambda_s^n&=&O\left(
\sum_{k=n}^{\infty}\frac{\left(\alpha_0^{(\frac{2}{\ell})^{\frac{k-2}{2}}}\right)^{\frac{1}{\ell}}}{\lambda_s^{k}}\lambda_s^{n-1}\right)\\
&=&O\left(\frac{\left(\alpha_0^{(\frac{2}{\ell})^{\frac{n-2}{2}}}\right)^{\frac{1}{\ell}}}{\lambda_s^n}\lambda_s^{n-1}\sum_{k=0}^{\infty}\left(\frac{1}{{2}}\right)^k\right)\\
&=&O\left(\left(\alpha_0^{(\frac{2}{\ell})^{\frac{n-2}{2}}}\right)^{\frac{1}{\ell}}\right).
\end{eqnarray*}
Finally, observe that the largest estimation comes from the term $$\epsilon_{0,n-1}=
O\left(\left(\alpha_0^{(\frac{2}{\ell})^{\frac{n-3}{2}}}\right)^{\frac{1}{\ell}}\right).$$
\end{proof}

\begin{lem}\label{asymalpha}
Let $f\in \W$ then $$\alpha_n=O\left(e^{C_u(f)\lambda_u^{n-1}}\right),$$ and $C_u(f)<0$.
\end{lem}
\begin{proof}
By the definition of $\alpha_n$ 
$$\alpha_n=\frac{S_{2,n}S_{3,n}}{1-S_{2,n}+\left(1-S_{1,n}\right)S_{2,n}S_{3,n}},$$
and by applying point $2$ and $5$ of Lemma \ref{prev}, we have
$$\frac{\alpha_n}{S_{2,n-1}}=1+O\left(S_{2,n-1}\right)=1+O\left(\alpha_{n-2}^{\frac{1}{\ell}}\right).$$
From the previous estimate and from \eqref{eq:wninalphan}, we get
$$\log\alpha_n=\log S_{2,n-1}+O\left(\alpha_{n-2}^{\frac{1}{\ell}}\right)=C_u(f)e^u_2\lambda_u^{n-1}+C_s(f)e^s_2\lambda_s^{n-1}+O\left(\left(\alpha_0^{(\frac{2}{\ell})^{\frac{n-3}{2}}}\right)^{\frac{1}{\ell}}\right).$$
 From Lemma \ref{eighvector} and Lemma \ref{alpha} it follows that 
 $$\alpha_n=O\left(e^{C_u(f)\lambda_u^{n-1}}\right),$$
 and $C_u(f)<0$.
\end{proof}
Lemma \ref{suplemma} and Lemma \ref{asymalpha} proves Proposition \ref{superformula} which was the aim of this section.
\subsection{Asymptotics in $X$-coordinates}
In the previous section we described the asymptotics of the renormalizations in the $S$ and $Y$-coordinates. Here we deduce the asymptotics of the renormalizations in the $X$-coordinates. This is needed in Section \ref{rigidity}. 
\begin{lem}\label{xs}
Let $f\in\W^{(X)}$, then 
\begin{eqnarray*}
\log |x_{1,n}|&=&C_u(f)\lambda_u^n(e_2^u+e_3^u-e_4^u)+C_s(f)\lambda_s^n(e_2^s+e_3^s-e_4^s)-C_-(f)(-1)^n+\delta_{1,n},
\\
 \log x_{2,n}&=&C_u(f)\lambda_u^n(e_2^u+e_3^u)+C_s(f)\lambda_s^n(e_2^s+e_3^s)+
 \delta_{2,n},\\
 \log x_{3,n}&=&C_u(f)\lambda_u^n(e_2^u+e_3^u)+C_s(f)\lambda_s^n(e_2^s+e_3^s)+
 \delta_{3,n},\\
\log |1-x_{4,n}|&=&C_u(f)\lambda_u^n(e_2^u)+C_s(f)\lambda_s^n(e_2^s)+
 \delta_{4,n},
\end{eqnarray*}
where $$\delta_{1,n},\delta_{2,n},\delta_{3,n},\delta_{4,n}=O\left(e^{\frac{C_u(f)\lambda_u^{n-4}}{\ell}}\right).$$
\end{lem}
\begin{proof}
By Lemma \ref{xtos},
$$
x_{1,n}=-\frac{S_{3,n}(1-S_{1,n})S_{2,n}}{(1+S_{3,n}(1-S_{1,n})S_{2,n})S_{4,n}}.
 $$
 As consequence
 $$
\log\left| x_{1,n}\right|=y_{2,n}+y_{3,n}-y_{4,n}+O\left(S_{1,n}\right)+O\left(S_{2,n}S_{3,n}\right),
 $$
 and by points $1$, $2$ and $5$ of Lemma \ref{prev}
 $$
\log\left| x_{1,n}\right|=y_{2,n}+y_{3,n}-y_{4,n}+O\left(\alpha_{n+1}^{\frac{1}{\ell}}\right)+O\left(\alpha_{n-2}^{\frac{1}{\ell}}\right).
 $$
By Proposition \ref{superformula} and by Lemma \ref{asymalpha} the esimates for $\log\left| x_{1,n}\right|$ follows. Similarly one obtains the other estimates. 
\end{proof}

\section{Renormalization of decomposed maps}\label{deco}

The renormalization operator in the space $\L$ is not differentiable. 
We cannot use the standard cone field method for the construction of invariant manifolds. 
The reason for which our renormalization operator is not differentiable is that the composition 
of $\Ct$ diffeomorphisms is not differentiable. 
In this section we introduce a space $L$ and a corresponding renormalization $R:L_0\to L$ in such a way that $R$ does not involve composition\footnote{We will abuse the notation denoting  both renormalization operators by $R$}.  Moreover, there is a natural projection $O:L\to\L$ such that $R:L_0\to L$ is a lift of $R:\L_0\to \L$. The space $L$ is called 
the space of decompositions. The corresponding renormalization $R$ on $L$ is still not differentiable, but it is jump-out-differentiable. This will allow us to apply Theorem \ref{InvMan} to obtain the invariant manifold formed by Fibonacci Cherry maps in the space $L$ which we will pull in $\L$ by using the natural projection $O$.

\subsection{The space of decompositions}
To discuss the differentiability of $R$ we need the space $\text{Diff }^{3+\epsilon}([0,1])$, with $\epsilon>0$, to be a normed vector space. 
This is achieved by using the concept of non-linearity.
\begin{defin}
Let $r\geq 2$. The non linearity $\eta:\text{Diff }^r([0,1])\to\mathcal{C}^{r-2}([0,1])$ is defined as
$$\varphi\to\eta_{\varphi}=D\log D\varphi.$$
\end{defin}
\begin{lem}
$\eta$ is a bijection.
\end{lem}
\begin{proof}
There is an explicit inverse of $\eta$. Namely,
$$\varphi(x)=\frac{\int_0^x e^{\int_0^s\eta} ds}{\int_0^1 e^{\int_0^s\eta} ds}.$$
\end{proof}
We can now identify $\text{Diff }^r([0,1])$ with $\mathcal{C}^{r-2}([0,1])$ and use the Banach space structure of 
$\mathcal{C}^{r-2}([0,1])$ on $\text{Diff }^r([0,1])$. The norm of a diffeomorphism $\varphi\in \text{Diff }^r([0,1])$ is defined as 
 $$|\varphi|_r=|\eta_{\varphi}|_{\mathcal{C}^{r-2}}.$$
 \begin{lem} Let $r\geq 2$. The metric defined on $\text{Diff }^r([0,1])$ by the norm $|\cdot|_r$, is equivalent to the $\mathcal{C}^{r}$ distance. Namely,
  $$|\varphi_n-\varphi|_r\to 0\iff 
  |\varphi_n-\varphi|_ {\mathcal{C}^{r}}\to 0.$$
\end{lem}
We are now ready for the definition of decomposition. Let $$T=\left\{\frac{k}{2^n}| n\geq 0, 0<k< 2^n\right\}\subset (0,1)$$ be the dyadic rationals 
with his natural order. Let $\theta:T\to T$ the doubling map $$\theta:\tau\to 2\tau \mod 1.$$ 
\\
Let $X_r$, $r\ge 2$, be the space of decomposed diffeomorphisms
$$X_r=\left\{\underline\varphi=(\varphi_\tau)_{\tau\in T}| 
\varphi_\tau\in\text{Diff }^r([0,1]), \sum |\varphi_\tau|_r<\infty \right\}.$$ 
We define the norm of $\underline\varphi\in X_r$ as
$$|\underline\varphi|_r=\sum |\varphi_{\tau}|_r.$$
\begin{rem}
One should think of $\underline\varphi$ as a chain of diffeomorphisms labeled by $T$. Their use allows to avoid composition. When $r=3+\epsilon$ we suppress the lower index and we simply use the notation $X=X_{3+\epsilon}$ and
$|\underline\varphi|=|\underline\varphi|_{3+\epsilon}$.
\end{rem}
Let $O_n:X\to\text{Diff }^2([0,1])$ be the partial composition defined as
$$O_n\underline\varphi=\varphi_1\circ\cdots\circ\varphi_{\frac{k}{2^n}}\circ\cdots\circ
\varphi_{\frac{2}{2^n}}\circ \varphi_{\frac{1}{2^n}}.$$
\begin{prop}\label{comp}
 The limit $$O\underline\varphi:=\lim O_n\underline\varphi$$ exists and it 
 is called the composition of $\underline\varphi$. 
 The composition $$O:X\to\text{Diff }^2([0,1])$$  is Lipschitz on bounded sets.
%  and 
% $$O:X\to\text{Diff }^0([0,1])$$  is $\Cuno$ .
\end{prop}
\begin{proof}
 The proof is a minor variation of the proof of the Sandwich Lemma from \cite{M}.
\end{proof}
We are now ready to define the space of decompositions $L$.
Recall that functions in $$\L^{(Y, 3+\epsilon)}=\Sigma^{(Y)}\times \text{Diff }^{ 3+\epsilon}([0,1])\times \text{Diff }^{ 3+\epsilon}([0,1])\times\text{Diff }^{ 3+\epsilon}([0,1])$$  are represented in $Y-$coordinates as follows: $f=(y_1,y_2,y_3,y_4, y_5, \varphi,\varphi^{l},\varphi^{r})$ where
$$
  \begin{matrix}
 y_1=S_1, & y_2=\log S_2, & y_3=\log S_3, & y_4=\log S_4, & y_5=\log S_5,
\end{matrix}
$$
and
\begin{equation*}
\Sigma^{(Y)}=\{(y_1,y_2,y_3,y_4,y_5)\in\mathbb R^5 | y_5<0 \}.
\end{equation*}
The space of decompositions is similar to $\L^{(Y, 3+\epsilon)}$ except that the diffeomorphisms are replaced by decomposed diffeomorphisms. Namely, $$L=\Sigma^{(Y)}\times X\times X\times X,$$
  and similarly as before, a point $\underline f\in L$ is represented by  $\underline f=(y_1,y_2,y_3,y_4,y_5, \underline\varphi,\underline\varphi^{l}, \underline\varphi^{r}).$ 
Observe that $L\subset\R^5\times X\times X\times X$ which carries a norm defined by the euclidian norm of $\R^5$ and the non-linearity norm of $X$. Namely, if $\underline f=\left(y_1,y_2,y_3,y_4,y_5, \underline\varphi,\underline\varphi^{l}, \underline\varphi^{r}\right)\in L$ then $$\left|\underline f\right|=\sum\left| y_i\right| +\left|\underline\varphi\right| +\left|\underline\varphi^{l}\right| +\left|\underline\varphi^{r}\right|.$$
\\
Let $O: L\to\L^{(Y,2)}$ be the composition defined as 
$$O:\underline f=\left(y_1,y_2,y_3,y_4,y_5, \underline\varphi,\underline\varphi^{l}, \underline\varphi^{r}\right)\to 
f=(y_1,y_2,y_3,y_4,y_5, O\underline\varphi,O\underline\varphi^l, O\underline\varphi^r)
$$
Observe that $\L^{(Y, 3+\epsilon)}$ is an open subset of $\R^5\times \text{Diff}^{ 3+\epsilon}([0,1])\times \text{Diff}^{ 3+\epsilon}([0,1])\times \text{Diff}^{ 3+\epsilon}([0,1])$ which carries a norm defined by the euclidian norm of $\R^5$ and the non-linearity norm of $\text{Diff}^{ 3+\epsilon}([0,1])$. There is a natural embedding of $\L^{(Y, 3+\epsilon)}$ into $L$.

\subsection{Renormalization on the space of decompositions}\label{RenonL}
Recall the definition of zoom, $Z_I$, from Definition \ref{zoomop}. The following lemmas hold. 
\begin{lem}\label{normzoom}
Let $r\ge 2$ and let $I$ be an interval in $[0,1]$. The norm of the linear operator $Z_{I}:\text{Diff }^r([0,1])\to \text{Diff }^r([0,1])$ satisfies
$$|Z_{I}|_r=|I|.$$
\end{lem}
\begin{lem}
Let $I$ be an interval of $[0,1]$. Then $Z_{I}\varphi$ is 
the diffeomorphism obtained by rescaling the restriction of $\varphi$ to $I$.
\end{lem}
We are now ready to define the renormalization operator on $L$.
We consider the subset of renormalizable maps $L_0\subset L$, 
$$L_0=\left\{\underline f\in L| 0<y_1<1\right\}.$$ 
 Let $\tau\in T$ 
and let $\pi^{\tau}: X\to X$ be defined as
$$(\pi^{\tau}\underline{\varphi})_{\tau'}=\left\{
    \begin{aligned}
    &0  & \tau'>\tau\\
      &\varphi_{\tau'}  & \tau'\leq\tau 
      \end{aligned}\right. $$
 Let $\tau\in T$. To describe the orbit of a point $x\in[0,1]$ in a decomposition we define the function $\gamma_{\tau}: X\times [0,1]\to [0,1]$ as 
 $$\gamma_{\tau}(\underline\varphi,x)=O\circ\pi^{\tau}(\underline\varphi)(x).$$ 
% The following lemma holds.
% \begin{lem} 
%  For every $\tau\in T$, the function $\gamma_{\tau}:X\times [0,1]\to [0,1]$ is $\Cuno$. 
%  Moreover $\gamma_{\tau}$ are uniformly bounded in the $\Cuno$ norm on bounded set of $X\times [0,1]$. 
% \end{lem}
%\TDD{Write a proof and see where it is used} 

Let $\underline f=\left({y}_1,{y}_2,{y}_3,{y}_4,{y}_5,{\underline\varphi},
{\underline\varphi}^{l},{\underline\varphi}^{r}\right)\in L_0$. The renormalization of $\underline f$ is
$$R\underline f=\left(\tilde{y}_1,\tilde{y}_2,\tilde{y}_3,\tilde{y}_4,\tilde{y}_5,\tilde{\underline\varphi},
\tilde{\underline\varphi}^{l},\tilde{\underline\varphi}^{r}\right)$$ which is defined in the following way. Let $f=O\left(\underline f\right)=\left({y}_1,{y}_2,{y}_3,{y}_4,{y}_5,{O\underline\varphi},
{O\underline\varphi}^{l},{O\underline\varphi}^{r}\right)$,
then 
$$Rf=\left(\tilde{y}_1,\tilde{y}_2,\tilde{y}_3,\tilde{y}_4,\tilde{y}_5,O\tilde{\underline\varphi}, 
O\tilde{\underline\varphi}^{l},O\tilde{\underline\varphi}^{r}\right).$$ This defines the $\tilde{y}_k$ of $R\underline f$. It is left to define $\tilde{\underline\varphi}$, $\tilde{\underline\varphi}^l$ and $\tilde{\underline\varphi}^r.$ The definition is the same as the definition of $\tilde{\varphi}$, $\tilde{\varphi}^l$ and $\tilde{\varphi}^r$ except that the diffeomorphisms are now decomposed diffeomorphisms, see Lemma \ref{ss}.
\\
Let $\tau\in T$, then $$y_1(\tau)=\gamma_{\tau}\left(\underline\varphi^{l},y_1\right),$$ and  
$$\tilde{\varphi}_\tau= Z_{\left[y_1(\tau), 1\right]}\varphi^{l}_\tau.$$ 
\\
Let $\tau\in T$, define $$\varphi^{-1}\circ q_s^{-1}\left(1-e^{y_2}\right)(\tau)=\gamma_{\tau}\left(\underline{\varphi},\left(O\underline\varphi\right)^{-1}\circ q_s^{-1}\left(1-e^{y_2}\right)\right),$$ and  
$$\tilde{\varphi}^{l}_{\tau}=\left\{
    \begin{aligned}
    &  Z_{\left[\varphi^{-1}\circ q_s^{-1}\left(1-e^{y_2}\right)(\theta\tau),1\right]}\varphi_{\theta\tau}  & \tau<\frac{1}{2}\\
      &   Z_{\left[q_s^{-1}\left(1-e^{y_2}\right),1\right]}q_s& \tau=\frac{1}{2}\\
    & \varphi^r_{\theta\tau} & \tau>\frac{1}{2}
       \end{aligned} \right.$$
where $s=e^{{y_5}/{\ell-1}}.$
\\
Let $\tau\in T$, define $$\varphi^{-1}\circ q_s^{-1}\left(y_1e^{y_2}e^{y_3}\right)(\tau)=\gamma_{\tau}\left(\underline{\varphi},\left(O\underline\varphi\right)^{-1}\circ q_s^{-1}\left(y_1e^{y_2}e^{y_3}\right)\right),$$ and  
$$\tilde{\varphi}^{r}_{\tau}=\left\{
    \begin{aligned}    
   & Z_{\left[0,\varphi^{-1}\circ q_s^{-1}\left(y_1e^{y_2}e^{y_3}\right)(\theta\tau)\right]}\varphi_{\theta\tau} & \tau<\frac{1}{2}\\      
  &  Z_{\left[0,q_s^{-1}\left(y_1e^{y_2}e^{y_3}\right)\right]}q_s& \tau=\frac{1}{2}\\
    &  Z_{\left[0, y_1(\theta\tau)\right]}\varphi^l_{\theta\tau} & \tau>\frac{1}{2}      
      \end{aligned} \right.$$

\bigskip

The renormalization operator $R:L_0\to L$  is now defined. 

\begin{defin}
A map $\underline{f}\in L$ is $\infty$-renormalizable if for every $n\geq 0$, $R^n \underline{f}\in L_0$. The set of $\infty$-renormalizable functions is denoted by $W\subset L$.
\end{defin}

Renormalization on $L_0$ is naturally defined to be the lift of renormalization on $\L_0$. Namely,
\begin{lem}
The renormalization operator on $L_0$ commutes with the renormalization operator on $\L_0$ under the composition $O$, i.e $$O\circ R=R\circ O.$$ Moreover 
$$W=O^{-1}(\W).$$
\end{lem}

\section{The manifold structure of the Fibonacci class}\label{manifold1}
In this section we prove that the class of Fibonacci maps is a $\Cuno$
codimension one manifold. 
\begin{theo}\label{manifold}
 $\W$ is a $\Cuno$ codimension one  manifold in $\L^{4+\epsilon}$.
\end{theo}
In order to apply Theorem \ref{InvMan} we use the following notation. 
Let $B_1=\R^4\times X\times X\times X$ and $B_0=\R^4\times X_2\times X_2\times X_2$. Observe that,
$$
R: L_0\subset 
\R\times B_1\to \R \times  B_1\subset \R\times B_0.
$$ 
The renormalization operator $R$ does not globally satisfies the hypothesis of Theorem \ref{InvMan}. However these hypothesis are satisfied on a domain $D$ introduced in the following. Let $1\geq\epsilon^*>0$, $C^*\geq 1$ and $C_u^*\geq 1$. A map $\underline f=\left({y}_1,{y}_2,{y}_3,{y}_4,{y}_5,{\underline\varphi},
{\underline\varphi}^{l},{\underline\varphi}^{r}\right)\in D_{\epsilon^*,C^*, C_u^*}$ if the following holds.
\\
Express the vector $w=(y_2,y_3,y_4,y_5)$ in the eigenvectors of the matrix $M$ as defined in Lemma \ref{eighvector},
 $$
 w=C_uE_u+C_sE_s+C_-E_-+C_0E_0.
 $$
Define $U=U_{\epsilon^*,C^*, C_u^*}\subset B_1$ by 
$$\left\{\begin{aligned}
 &C_u<  -C^*_u,\\
 &|C_0|, |C_s|<C^*,\\
& 16|\underline\varphi |+8
|{\underline\varphi}^{l}|+3|{\underline\varphi}^{r}|<\epsilon^*,
\end{aligned}\right.
$$
and  $\partial_\pm: U\to  \R$ by
$$\left\{\begin{aligned}
&\partial_+(b)=\left(\frac{3}{2\ell}e^{y_2}\right)^\frac{1}{\ell},\\
&\partial_-(b)=\left(\frac{1}{4\ell}e^{y_2}\right)^\frac{1}{\ell},
\end{aligned}\right.
$$
where $b=\left({y}_2,{y}_3,{y}_4,{y}_5,{\underline\varphi},
{\underline\varphi}^{l},{\underline\varphi}^{r}\right)$.
The map $\underline f=\left(y_1, b\right)$ is in $D_{\epsilon^*,C^*, C_u^*}$ if 
$$\left\{\begin{aligned}
&b\in U,\\
&\partial_-(b)\leq y_1\leq \partial_+(b).
\end{aligned}\right.
$$
In the sequel, we supress the indices and use the notation $D=D_{\epsilon^*,C^*, C_u^*}$.
 Observe that $D\subset L_0$ when $C^*_u$ is large enough.

\subsection{Jump-out-differentiability and structure of the derivatives}\label{jumpoutdiff} 
\begin{prop} \label{jumpoutdiffprop}
  Let $D=D_{\epsilon^*,C^*, C_u^*}$ be a domain or a closed bounded set, then the map $R:D\to L\subset\R\times B_1\subset \R\times B_0$ is jump-out-differentiable.
\end{prop}
\begin{proof}
By Propositions \ref{Aderivative}, \ref{Bsderivative}, \ref{Blderivative}, \ref{Csderivative}, \ref{Clderivative}, \ref{Crderivative}, \ref{Dslderivative}, \ref{Dlsderivative}, \ref{Dlrderivative}, \ref{Drsderivative}, \ref{Drlderivative} we get that 
$$
DR_{\underline{f}}: \R^5\times X_2\times X_2\times X_2\to \R^5\times X_2\times X_2\times X_2
$$
as described by (\ref{matrix}), is bounded and it depends continuously on $\underline f$. The same propositions imply that 
$$
\lim_{\left| \Delta\underline f\right|\to 0}\frac{\left|R\left(\underline f+\Delta\underline f\right)-\left[R\underline f+DR_{\underline f}\left(\Delta\underline f\right)\right]\right|_2}{\left|\Delta\underline f\right|}=0.
$$
This shows that $R:D\to L$ is jump-out differentiable.
\end{proof}

\begin{prop} There exists $E>0$ and $0\leq\kappa<1$ such that every domain $D_{\epsilon^*,C^*, C_u^*}$ has the following property. Let $\underline f=\left({y}_1, b\right)\in D_{\epsilon^*,C^*, C_u^*}$ with $R\underline f=\left(\tilde y_1, \tilde b\right)$ then $\left(\Delta\tilde y,\Delta\tilde b\right)=DR_{\underline f}\left(\Delta y,\Delta b\right)$ satisfies,
\begin{equation*}
\left\{\begin{aligned}
&\Delta\tilde y=\frac{E_{\underline f}}{y_1}\Delta y+O\left(\Delta b\right),\\
&&\\
&\|\Delta\tilde b \|_0=O\left(\frac{1}{\tilde y_1^{\kappa}}|\Delta y|+\|\Delta b\|_0\right),
\end{aligned}\right.
\end{equation*}
with ${1}/{E}<|E_{\underline f}|< E$. 
\end{prop}
\begin{proof}
Consider $D=D_{1,1,1}$. From Propositions \ref{Aderivative}, \ref{Bsderivative}, \ref{Blderivative}, \ref{Csderivative}, \ref{Clderivative}, \ref{Crderivative}, \ref{Dslderivative}, \ref{Dlsderivative}, \ref{Dlrderivative}, \ref{Drsderivative}, \ref{Drlderivative} we get that 
\begin{eqnarray*}
\frac{\partial\tilde y_1}{\partial b}&=&O(1),\\
\frac{\partial\tilde b}{\partial b}&=&O(1),
\end{eqnarray*}
and from Lemma \ref{partial y1 partial yj} we get that there exists $E>0$ such that $$\frac{1}{E}\frac{1}{y_1}\leq\frac{\partial\tilde y_1}{\partial y_1}\leq E\frac{1}{y_1}.$$
By Propositions \ref{Csderivative}, \ref{Clderivative}, Lemma \ref{partial y2 partial yj}, \ref{partial y3 partial yj}, \ref{partial y4 partial yj}, \ref{partial y5 partial yj}, \ref{partial phi r min 12 partial y}, \ref{partial phi r12 partial y}, \ref{partial phi r max 12 partial y} we get 
$$\frac{\partial\tilde b}{\partial y_1}=O\left(\frac{1}{y_1}+S_2^{\frac{1}{\ell^2}}S_3^{\frac{1}{\ell}}S_5^{-\frac{2}{\ell-1}}\right).$$
Observe that, by point $2$ of Lemma \ref{ss} and Lemma \ref{D111} we get 
$$
\tilde S_1S_2^{\frac{1}{\ell^2}}S_3^{\frac{1}{\ell}}S_5^{-\frac{2}{\ell-1}}=O\left(S_2^{\frac{\ell^2+\ell+1}{\ell^2}}S_3^{\frac{\ell+1}{\ell}}S_5^{\frac{-\ell^2-2\ell+1}{\ell(\ell-1)}}\right).
$$
Expressing the $S$-coordinates in terms of the coordinates in the eigenbasis of the matrix $M$, see Lemma \ref{matrix}, we obtain
$$ 
\log\left[\tilde S_1S_2^{\frac{1}{\ell^2}}S_3^{\frac{1}{\ell}}S_5^{-\frac{2}{\ell-1}}\right]= C_u\left[\frac{\lambda_u^2\left(\ell^2+\ell+1\right)+\lambda_u\left(-2\ell+1\right)-2\ell}{\ell^2\lambda_u\left(1+\lambda_u\right)}\right]+O\left(C\right).
$$
A calculation shows that the coefficient of $C_u$ is positive. A similar expression holds for $
\log\left[\tilde S_1^{\kappa}S_2^{\frac{1}{\ell^2}}S_3^{\frac{1}{\ell}}S_5^{-\frac{2}{\ell-1}}\right]$ and when $0<\kappa<1$ is close enough to $1$,  the coefficient of $C_u$, which depends continuously on $\kappa$, is again positive. Because $C_u\leq -1$ we get
\begin{equation}\label{kappass}
\tilde S_1^{\kappa}S_2^{\frac{1}{\ell^2}}S_3^{\frac{1}{\ell}}S_5^{-\frac{2}{\ell-1}}=O(1).
\end{equation}
Using again point $2$ of Lemma \ref{ss} and Lemma \ref{D111}  we get 
$$
\frac{\tilde y_1}{y_1}=O\left(S_2^{\frac{1}{\ell}} S_3 S_5^{-1}\right).
$$
Expressing the $S$-coordinates in terms of the coordinates in the eigenbasis of the matrix $M$, see Lemma \ref{matrix}, and because $C_u\leq -1$ we obtain
$$
\log\left[\frac{\tilde y_1}{y_1}\right]\leq -\left[\frac{\lambda_u+1-\ell}{\ell\left(1+\lambda_u\right)}\right]+O\left(C\right).
$$
As before, for $0<\kappa<1$ close enough to $1$ we also have that 
\begin{equation}\label{kappay}
\frac{1}{y_1}=O\left(\frac{1}{\tilde y_1^{\kappa}}\right).
\end{equation}
Finally, by (\ref{kappass}) and (\ref{kappay}) we get
$$\frac{\partial\tilde b}{\partial y_1}=O\left(\frac{1}{\tilde y_1^{\kappa}}\right).$$
The uniform bounds follow from the fact that $D_{\epsilon^*,C^*, C_u^*}\subset D_{1,1,1}$.
\end{proof}

\begin{prop}\label{xiexpansion} 
For every $\epsilon^*<1$, $C^*>1$ and $\xi>0$, if $\underline f=\left({y}_1, b\right)\in D_{\epsilon^*,C^*, C_u^*}$ with $R\underline f=\left(\tilde y_1, \tilde b\right)$, then $R$ has vertical $\xi$-expansion for $C_u^*>1$ large enough, i.e.
\begin{eqnarray*}
\frac{y_1}{\tilde y_1^{\kappa}}&\geq &\xi.
\end{eqnarray*}

 \end{prop}
 \begin{proof}
 Using point $2$ of Lemma \ref{ss} and Lemma \ref{D111}  we get 
$$
\frac{\tilde y_1}{y_1}=O\left(S_2^{\frac{1}{\ell}} S_3 S_5^{-1}\right).
$$
Expressing the $S$-coordinates in terms of the coordinates in the eigenbasis of the matrix $M$, see Lemma \ref{matrix}, we obtain
$$
\log\left[\frac{\tilde y_1}{y_1}\right]\leq -C_u\left[\frac{\lambda_u+1-\ell}{\ell\left(1+\lambda_u\right)}\right]+O\left(C\right).
$$
For $0<\kappa<1$ close enough to $1$ and $C_u^*>1$ large enough, we also have that 
$$
\frac{\tilde y_1^{\kappa}}{y_1}\leq\frac{1}{\xi}.
$$
 \end{proof}
 \begin{prop}\label{etadominationg} 
For every $\epsilon^*<1$, $C^*>1$ and $\eta>0$, if $\underline f=\left({y}_1, b\right)\in D_{\epsilon^*,C^*, C_u^*}$ with $R\underline f=\left(\tilde y_1, \tilde b\right)$, then $R$ has $\eta$-dominating horizontal expansion for $C_u^*>1$ large enough, i.e.
\begin{eqnarray*}
\frac{y_1^2}{\tilde y_1^{\kappa}}&\leq &\eta.
\end{eqnarray*}
 \end{prop}
 \begin{proof}
 Using point $2$ of Lemma \ref{ss} and Lemma \ref{D111}  we get 
$$
\frac{y_1^2}{\tilde y_1}=O\left(S_2^{\frac{\ell-1}{\ell}} S_3^{-1} S_5\right).
$$
Expressing the $S$-coordinates in terms of the coordinates in the eigenbasis of the matrix $M$, see Lemma \ref{matrix}, we obtain
$$
\log\left[\frac{y_1^2}{\tilde y_1}\right]\leq -C_u\left[\frac{\lambda_u(\ell-1)+2\ell-1}{\ell\left(1+\lambda_u\right)}\right]+O\left(C\right).
$$
For $0<\kappa<1$ close enough to $1$ and $C_u^*>1$ large enough, we also have that 
$$
\frac{y_1^2}{\tilde y_1^{\kappa}}\leq\eta.
$$
 \end{proof}

%\begin{rem}
%With the same techniques we can prove that $\W_c$ is in fact a $\Cd$ manifold. However, this would involve more technical details without being relevant for this paper. 
%%In general, our techniques can be generalized for proving that in the space of 
%%%$\mathcal C^r$ functions, $r\geq 3$, it is possible to construct a $\mathcal C^{r-1}$ invariant manifold.
%\end{rem} 

 \subsection{Topological hyperbolicity}\label{tophyp}
In this section we prove that, for an appropriate choice of $\epsilon^*, C^*, C_u^*$,  $R:D_{\epsilon^*,C^*, C_u^*}\to L$ is topologically hyperbolic. The choice needs some preparation.
 \\
The largest eigenvalue of 
$$Q=\left(\begin{matrix}
    
    0  & 1 & 0\\
     \frac{1}{16} & 0 & 1\\
    \frac{1}{16} &  \frac{1}{16} & 0\\
      \end{matrix}\right)
$$
  is $\frac12$ with eigenvector $\underline E=(16,8,3)$.    Let $\underline f=\left({y}_1,{y}_2,{y}_3,{y}_4,{y}_5,{\underline\varphi},
{\underline\varphi}^{l},{\underline\varphi}^{r}\right)\in L_0$ and $Rf=\left(\tilde{y}_1,\tilde{y}_2,\tilde{y}_3,\tilde{y}_4,\tilde{y}_5,O\tilde{\underline\varphi}, 
O\tilde{\underline\varphi}^{l},O\tilde{\underline\varphi}^{r}\right)$. Define
  $$
  m=m(\underline{f})=(|{\underline\varphi}|,
|{\underline\varphi}^{l}|,|{\underline\varphi}^{r}|),
  $$ 
  and $\tilde{m}=m(R\underline f)$.

\begin{lem}\label{distcontraction} For every $\epsilon^*<1$ and $C^*\geq 1$ we have
$$
\tilde{m}\le Qm+ \frac{\epsilon^*}{33}\left(\begin{matrix}
0\\
1\\
1\\
\end{matrix}\right),
$$
when $C_u^*$ is large enough.
In particular,
$$
(E,\tilde{m})\le \frac12 (E,m) +\frac{1}{3}\epsilon^*.
$$
\end{lem}

 \begin{proof} We estimate the norms $\tilde{w}$, $\tilde{w}^l$ and $\tilde{w}^r$. Observe,
 $$
\left|\left[y_1(\tau),1\right]\right|\le 1. 
 $$
 This implies, using Lemma \ref{normzoom} and the definition of renormalization,
 \begin{equation}\label{w}
 \tilde{w}\le w^l.
 \end{equation}
 Moreover, for $C_u^*>1$ large enough we have   
 \begin{equation}\label{ZS2m}
|[ \varphi^{-1}\circ q_s^{-1}\left(1-e^{y_2}\right)(\tau),1]|=O\left(S_2(1+\epsilon^*)\right)\le \frac{1}{16}.
 \end{equation}
 The previous inequality, Lemma \ref{normzoom} and Lemma \ref{C2} imply,
 \begin{equation}\label{wl}
 \tilde{w}^l\le \frac{1}{16} w+w^r+\frac{\epsilon^*}{33}.
 \end{equation}
For $C_u^*>1$ large enough and Lemma \ref{s1235} we get
 \begin{eqnarray}\label{ZSSm}\nonumber
 \varphi^{-1}\circ q_s^{-1}\left(y_1e^{y_2}e^{y_3}\right)(\theta\tau)&=&
 O\left(\frac{S_1S_2S_3}{S_5}\right)=O\left(S_2^{1+\frac{1}{\ell}}S_3 S_5^{-1}\right)\\&=&
O\left(S_2^{1+\frac{1}{\ell}}S_3 S_5^{-\frac{\ell}{\ell-1}}\right)\le \frac{1}{16},
 \end{eqnarray}
 where we used ${\ell S_1^{\ell}}/{S_2}\leq {3}/{2}$ and $\log\left({S_1S_2S_3}/{S_5^{\frac{\ell}{\ell-1}}}\right)\leq -C_u^*\left({\lambda_u^2(\ell+1)-1}/{\ell\lambda_u(1+\lambda_u)}\right)+O(C^*)$, see Lemma \ref{eighvector}. Furthermore, for $C_u^*>1$ large enough,
  \begin{equation}\label{Zy}
 |[0,y_1(\theta \tau)]|=O(y_1)=O\left(S_2^{\frac{1}{\ell}}\right)\le \frac{1}{16},
 \end{equation}
and by Lemma \ref{C9} we obtain
 \begin{equation}\label{wr}
 \tilde{w}^r\le \frac{1}{16} w+\frac{1}{16}w^l+\frac{\epsilon^*}{33}.
 \end{equation}
 Estimates (\ref{w}), (\ref{wl}), and (\ref{wr}) conclude the proof of the lemma.
 \end{proof}

Let $\underline f_\pm=(\partial_\pm(b),b)\in \partial D$ with $R\underline f_\pm=(\tilde{y}_\pm, \tilde{b}_\pm)$. The following holds.
 \begin{lem}\label{overlap} For every  $C^*\geq 1$  
  we have
 $$
\tilde{y}_+\le -\frac18 \text{ and } \tilde{y}_-\ge \frac{7}{16},
 $$
 when  $C_u^*>1$ is large enough and $\epsilon^*<1$ is small enough.
\end{lem}

 \begin{proof} Observe that
 $$
\frac{S_2}{1-\left(\varphi^{-1}\circ q_s^{-1}\right)\left(1- S_2\right)}=\ell(1+O(\epsilon^*)),
 $$
 and 
 $$
 \frac{\left(\varphi^{l}\left( S_1\right)\right)^{\ell}}{\ell S_{1}^{\ell}}=1+O(\epsilon^*).
 $$
 Then for $\epsilon^*$ small enough we get, by using point 1 of Lemma \ref{ss},
 $$
 \tilde{y}_+\le 1-\frac32\left(1-\frac14\right)=-\frac18,
 $$
 and
 $$
 \tilde{y}_-\ge 1-\frac14\left(2+\frac14\right)\ge \frac{7}{16}.
 $$
 \end{proof}

\begin{prop} \label{proptophyp} For $C^*>1$ large enough $R: D_{\epsilon^*, C^*,C_u^* }\to \R\times B_1$ is topologically hyperbolic when $C_u^*>1$ is large enough and $\epsilon^*<1$ is small enough.
\end{prop} 

\begin{proof} Apply Lemma \ref{ss} and Lemma \ref{squarefactors} to obtain
 \begin{eqnarray*}
 \tilde{C}_u&\leq &\lambda_u C_u+\log 2,\\
 \tilde{C}_s&\leq&\lambda_s C_s+\log 2,\\
 \tilde{C}_0&\leq&\log 2.
 \end{eqnarray*}
 This implies, when $C^*>1$ and $C_u^*>1$ are large enough, that
 $$
 \tilde{C}_u\le -C^*_u \text{ and } |\tilde{C}_0|, |\tilde{C}_s|\le C^*.
 $$
 From Lemma \ref{distcontraction} we get
 $$
 (E, \tilde{w})\le \frac12 \epsilon^*+\frac{1}{3}\epsilon^*\le \epsilon^*.
 $$
Finally when $C_u^*>1$ is large enough, for every $b\in U_{\epsilon^*, C^*,C_u^* }$ 
 $$
 0\le \partial_-(b)< \partial_+(b)\le \frac14.
 $$   
Lemma \ref{overlap} concludes the proof that $R:D\to \R\times B_1$ is topologically hyperbolic. 
 \end{proof}

 \subsection{The Fibonacci class}
 In this section we prove Theorem \ref{manifold} which states that the class of Fibonacci maps in $\L^{4+\epsilon}$ is a $\Cuno$ codimension one manifold.
 
 Consider $B_1^{4+\epsilon}=\R^4\times X_{4+\epsilon}\times X_{4+\epsilon}\times X_{4+\epsilon}$ and observe that $\left(B_1^{4+\epsilon}, |\cdot |_{3+\epsilon}\right)\subset B_1$ as a normed vector space. Similarly, $\left(B_1, |\cdot |_{2}\right)$ is a normed vector space in $B_0$. Define $$L_{3+\epsilon}^{4+\epsilon}=\Sigma^{(Y)}\times X_{4+\epsilon}\times X_{4+\epsilon}\times X_{4+\epsilon}\subset L\subset\R\times B_1,$$ and $$L_2^{3+\epsilon}=\Sigma^{(Y)}\times X_{3+\epsilon}\times X_{3+\epsilon}\times X_{3+\epsilon}\subset \R\times B_0.$$ By Proposition \ref{jumpoutdiffprop}, $$R:L_{3+\epsilon}^{4+\epsilon}\to L_2^{3+\epsilon}$$ is continuously differentiable.
Apply Theorem \ref{InvMan} to get $D\subset [0,1]\times U$ with $U\subset B_1$ maximal and a $\Cuno$ function $\hat\omega^*:U\to\R$ such that the invariant set\footnote{See Lemma \ref{omegaW}}, $$\omega^*=\left\{p\in D | \forall n\in\N\text{ }R^n(p)\in D\right\}$$ is the graph of $\hat\omega^*$. 
Consider $$W=\left\{\underline f\in L | \exists n\in\N \text{ s.t. } R^n\underline f\in\omega^*\right\},$$ and notice that 
\begin{itemize}
\item[-]$\omega^*$ is a $\Cuno$ graph in $W$,
\item[-]for all $n\in\N$, $R^n\left(W\right)\subset W$.
\end{itemize}

\begin{lem}\label{transversaldeformation} Given $f\in \W\subset\L^{4+\epsilon}$ there exists a family $f_t\in\L^{4+\epsilon}$, smooth in $t\in (-1,1)$, with $f_0=f$ and $b_0>0$ such that 
$$
\{f'\in \L^{4+\epsilon} | \text{ } \exists t\ne 0 \text{ with } |f'-f_t|_{\C0}\leq b_0\cdot t\}\cap \W =\emptyset.
$$
\end{lem}

\begin{proof} Let $f\in\W$ with $x_1=x_1(f)$ and $\pi:\R\to [x_1,1)$ be a piece wise smooth projection with period $1-x_1$ such that the lift $F:\R\to \R$ of $f$ is smooth. The proof of Lemma \ref{alpha} assures that such a projection $\pi$ exists. Moreover, we may assume that $\pi: [0,1)\to [0,1)$ is identity and $\pi:[x_1,0)\to [x_1,0)$ is a diffeomorphism. Let
$F_t:\R\to \R$ be given by an added rotation
$$
F_t(x)=F(x)+t.
$$
Then for every lift $G$ of a map $g\in \L^{4+\epsilon}$ with period $1-x_1$ and 
\begin{equation}\label{GFt}
|G(x)-F_t(x)|\le \frac12 t,
\end{equation} 
 the circle map $g$ is not a Fibonacci map. 
 \\
 Using an appropriate projection of period $1-x_1$, each map $F_t$ is the lift of a map $f_t\in\L^{4+\epsilon}$. This needs some preparation. Observe that $J=[x_3(f),x_4(f)]$ is the flat interval of any $F_t$. Let 
 $$
 o(t)=F_t(J)=t, \textbf{ } x_1(t)=F_t(o(t)) \textbf{ } 1(t)=x_1(t)+1-x_1.
 $$
The interval $[x_1(t), 1(t))$ is a fundamental domain for $F_t$. Let $\pi_t:\R\to \R$ be the $1-x_1$ periodic piece wise smooth map given by
$$
\pi_t(x)=
\left\{\begin{aligned}
&\pi\left(\frac{x-o(t)}{1(t)-o(t)}\right)&\text{ when } x\in [o(t), 1(t))\\
&\pi\left(x_1(f) \frac{x-o(t)}{x_1(t)-o(t)}\right) &\text{ when } x\in [x_1(t), o(t))\\
\end{aligned}\right.
$$
Consider the smooth family $f_t\in\L^{4+\epsilon}$ with $f_t:[x_1,1)\to [x_1,1)$ defined by 
$$
f_t=\pi_t\circ F_t\circ \pi_t^{-1},
$$
with $t$ in a small interval centered around $0$. For $b_0>0$ small enough we have that every $f'\in \L^{4+\epsilon}$ with $|f'-f_t|_{\C0}\le b_0 t$ there exists a  $1-x_1$ periodic lift $F'$ with 
$|F'(x)-F_t(x)|\le (1/2) t$. According to \ref{GFt}, $f'$ is not a Fibonacci map. Hence $f'\notin\W$.
\end{proof}

Define the embedding $i:\L^{4+\epsilon}\to L$ as follows. If $f=\left(y_1, y_2, y_3, y_4, y_5, \varphi, \varphi^l, \varphi^r\right)\in\L^{4+\epsilon}$, then $$i(f)=\left(y_1, y_2, y_3, y_4, y_5, \underline\varphi, \underline\varphi^l, \underline\varphi^r\right),$$ with $$\varphi_{\frac{1}{2}}=\varphi,\text{ }\varphi^l_{\frac{1}{2}}=\varphi^l, \text{ }\varphi^r_{\frac{1}{2}}=\varphi^r,$$ and $\varphi_{\tau}=\varphi^l_{\tau}=\varphi^r_{\tau}=0$ for $\tau\neq{1}/{2}$. This is a $\Cuno$ map with $O\circ i=id$.

\begin{lem}\label{transversality}
Let $\underline f\in W\cap L_{3+\epsilon}^{4+\epsilon}$ such that $W$ is locally a $\Cuno$ codimension 1 manifold in $L$ containing $R\underline f$, then $$DR_{\underline f}\pitchfork T_{R\underline f}W.$$ 
\end{lem}
\begin{proof} Let $f=O(\underline f)$. Consider the family $f_t$ through $f$ from Lemma \ref{transversaldeformation}. The Lipschitz continuity of $O$, see Proposition\ref{comp}, implies that for $b>0$ small enough
$$
\{\underline{f}'\in L | \text{ } \exists t\ne 0 \text{ with } |\underline{f}'-i(f_t)|_{\C0}\leq b\cdot t\}\cap W =\emptyset.
$$
Moreover, for every $\underline{f}'\in L$ with $ |\underline{f}'-i(f_t)|_{\C0}\leq b\cdot t$ for some $t\ne 0$ we have 
\begin{equation}\label{RfnotW}
R\underline{f}'\notin W.
\end{equation}
Let $\Delta v=\frac{\partial( i( f_t))}{\partial t}$. Then 
$DR_{\underline f}(\Delta v)\in \R^5\times X_{3+\epsilon}\times X_{3+\epsilon}\times X_{3+\epsilon}$ and because (\ref{RfnotW})
$$
DR_{\underline f}(\Delta v)  \pitchfork T_{R\underline f}W. 
$$
\end{proof}

\begin{cor}\label{transi} Let $f\in \W\cap \L_{3+\epsilon}^{4+\epsilon}$ such that $W$ is locally a $\Cuno$ codimension 1 manifold in $L$ containing   $i(f)$, then $$Di_{f}\pitchfork T_{i(f)} W.$$ 
\end{cor}

\begin{prop}\label{WC1} Let $\underline f\in W\cap L_{3+\epsilon}^{4+\epsilon}$ then $W$ is locally a $\Cuno$ codimension 1 manifold in $L$ containing $\underline f$.
\end{prop}
\begin{proof}
By contradiction, suppose that there exists $\underline f\in W\cap L_{3+\epsilon}^{4+\epsilon}$ such that $W$ is not locally a $\Cuno$ codimension 1 manifold in $L$ containing $\underline f$. Then we may assume that $W$ is locally a $\Cuno$ codimension 1 manifold containing $R\underline f$. Such an $\underline f$ exists because $\omega^*$ is a $\Cuno$ graph in $W$.
Let $n\ge 1$ such that $R^n\underline{f}\in \omega^*$. The derivative of $R^n$ extends to a bounded operator $DR^n_{\underline{f}}: \R^5\times X_{2}\times X_{2}\times X_{2}\to \R^5\times X_{2}\times X_{2}\times X_{2}$ because $R$ is jump-out-differentiable. 
Hence the tangent space at $R\underline f$ to $W$, $T_{R\underline f}W$, extends to a plane in $\R^5\times X_{2}\times X_{2}\times X_{2}$. Namely, 
$T_{R\underline{f}}W=(DR^n_{\underline{f}})^{-1}(T_{R^n\underline{f}} \omega^*)$. The contradiction comes straight from Lemma \ref{transversality} and Lemma \ref{pullbacklemma}.
\end{proof}

\begin{lem}\label{DegcondinL}
If $f\in\W$, then $i(f)\in W$.
\end{lem}
\begin{proof}
Fix $\epsilon^*$, $C^*$ and $C_u^*$ as in Proposition \ref{proptophyp} and let $f\in\W$. The aim is to prove that there exists $n\in\N$ such that $R^n i(f)\in\omega^*$. Proposition \ref{superformula} ensures that, for $n$ large enough, $C_u\left(R^ni(f)\right)=C_u\left(R^nf\right)<-C_u^*$ and $\left|C_s\left(R^ni(f)\right)\right|=\left|C_s\left(R^nf\right)\right|<C^*$.
\\
 It remains to prove that, for $n$ large enough, $\left(E, w(R^n i(f))\right)\leq\epsilon^*$. For all $n\in\N$, denote by $ w_n=w\left(R^ni(f)\right)$, $ w^l_n=w^l\left(R^ni(f)\right)$ and $ w^r_n=w^r\left(R^ni(f)\right)$. By the definition of renormalization on $L$, see Section \ref{RenonL}, and Lemma \ref{normzoom} we get 
$$|w_{n+1}|+|w^l_{n+1}|+|w^r_{n+1}|\leq |w_{n}|+|w^l_{n}|+\left|w^r_{n}\right|+\left| Z_{\left[q_s^{-1}\left(1-S_2\right),1\right]}q_s\right|+\left|Z_{\left[0,q_s^{-1}\left(S_1S_2S_3\right)\right]}q_s\right|.$$ 
Applying (\ref{ZS2}) and (\ref{ZSS}) we get 
$$|w_{n+1}|+|w^l_{n+1}|+|w^r_{n+1}|\leq |w_{n}|+|w^l_{n}|+\left|w^r_{n}\right|+O\left(S_2\right)+O\left(S_1 S_2 S_3 S_5^{-\frac{\ell}{\ell-1}}\right).$$ 
Writing the affine terms in the eigenvector base, see Lemma \ref{eighvector}, we notice that they tend to zero double exponentially fast. As consequence,
$$|w_{n}|+|w^l_{n}|+\left|w^r_{n}\right|\leq O(1).$$ By the definition of renormalization on $L$, see Section \ref{RenonL}, Lemma \ref{normzoom}, (\ref{ZS2}), (\ref{ZSS}), (\ref{ZS2m}), (\ref{ZSSm}) and (\ref{Zy}) we get
$$
\left\{\begin{aligned}
|w_{n+1}|\leq& |w^l_{n}|,\\
|w^l_{n+1}|\leq& O\left(S_2\right)|w_{n}|+ |w^r_{n}|+O\left(S_2\right),\\
|w^r_{n+1}|\leq& O\left(S_1 S_2 S_3 S_5^{-\frac{\ell}{\ell-1}}\right)|w_{n}|+O\left(S_2^{\frac{1}{\ell}}\right) |w^l_{n}|+O\left(S_1 S_2 S_3 S_5^{-\frac{\ell}{\ell-1}}\right).\\
\end{aligned}\right.
$$
As before, the coefficients in the orders decay double exponentially fast. Iterating the above estimates three times, we obtain 
$$|w_{n+1}|+|w^l_{n+1}|+|w^r_{n+1}|\leq\lambda^{n-2}\left(|w_{n-2}|+|w^l_{n-2}|+\left|w^r_{n-2}\right|\right)+\lambda^{n-2},$$ where $\lambda<1$. The lemma follows.
\end{proof}
By Corollary \ref{transi}, Proposition \ref{WC1},  Lemma \ref{DegcondinL} and Lemma \ref{pullbacklemma}, $i^{-1}\left( W\right)=\W$ is a $\Cuno$ codimension one manifold in $\L$. This concludes the proof of Theorem \ref{manifold}.

\section{Rigidity}\label{rigidity}
For any function $f\in\W$, the non-wandering set\footnote{The non wandering set of a map $f$ is the set of the points $x$ such that for any open neighborhood $V\ni x$ there exists an integer $n>0$ such that the intersection of $V$ and $f^n(V)$ is non-empty.} of $f$ is $K_f=\S\setminus\cup_{i\geq 0}f^{-i}(U_f)$, where  $U_f=[x_3(f),x_4(f)]$ is the flat interval of $f$, see \cite{5aut}, \cite{M-vS-dM-M}. 
\begin{defin}
Let $f\in\L$. The set $A_f$ is the attractor of $f$ if $A_f$ is the limit set of every point.
\end{defin}
The following lemma holds. For its proof, the reader can refer to \cite{5aut} and \cite{M-vS-dM-M}.
\begin{lem}
Let $f\in\W$. The non-wandering set $K_f$ is a Cantor set and it is the attractor of the system.
\end{lem}
Let $f,g\in \W$ and let $h$ be an homeomorphism which conjugates $f$ and $g$. 
Observe that $h$ is such that $h(U_f)=U_g$ but $h$ is not uniquely defined inside $U_f$. However, $h_{|K_f}$ is uniquely defined. 
 Being interested in the geometry of the Cantor set $K_f$, from now on we
  only discuss the quality of $h_{|K_f}$ which is simply denote by $h$.
\begin{defin}
Let $K\subset\R$ be a Cantor set. A function $h:K\to\R$ is differentiable if there exists a function $Dh:K\to\R$ such that, for every $x_0\in K$,
$$h\left(x\right)=h\left(x_0\right)+Dh\left(x_0\right)\left(x-x_0\right)+o\left(\left|x-x_0\right|\right)$$
for every $x\in K$.
Moreover $h$ is $\mathcal C^{1}$ differentiable if $Dh$ is continuous and $h$ is $\mathcal C^{1+\beta}$ differentiable if $Dh$ is H\"older with exponent $\beta>0$.
\end{defin}
We are now ready to state the main theorem of this section.
\begin{theo}\label{LM}
Let $1<\ell<2$ and let $\beta=\frac{\lambda_u-1}{\lambda_u^4\ell}>0$. If $f,g\in\W$ with critical exponent $\ell$ and if $h$ is the topological conjugacy
between $f$ and $g$, then $h$ is a H\"older homeomorphism. Moreover the following holds.
\begin{eqnarray*}
 h \text{ is a Lipschitz homeomorphism} &\iff& C_u(f)=C_u(g),\\
 h \text{ is a } \mathcal C^1 \text{ diffeomorphism }&\iff& C_u(f)=C_u(g), C_-(f)=C_-(g),\\
 h \text{ is a } \mathcal C^{1+\beta} \text{ diffeomorphism}&\iff& C_u(f)=C_u(g), C_-(f)=C_-(g), C_s(f)=C_s(g).\\
\end{eqnarray*}
\end{theo}

%\begin{prop}
% Let $\ell<2$, $f,g\in \left[\infty-R\right]$ with critical exponent $\ell$ and $C_u(f)\neq C_u(g)$ then $h$ is not Holder.
%\end{prop}
%\TDD{define somewhere the meaning of $\sim$}
\begin{prop}\label{nothol}
 Let $f,g\in\W$ with different critical exponents $\ell_f\neq\ell_g$,  $1<\ell_f,\ell_g<2$. Then $h$ is not H\"older.
\end{prop}
\begin{proof}
Because $\ell_f\neq\ell_g$, then $\lambda_u(f)\neq\lambda_u(g)$. Without loose of generality we may assume $\lambda_u(g)<\lambda_u(f).$
 Observe that $$\frac{\hat x_{2,n}}{-\hat x_{1,n}}=\frac{f^{q_{n+1}}(0)}{-f^{q_{n}}(0)}=\frac{x_{2,n}}{-x_{1,n}}=S_{4,n}.$$ As a consequence, by Proposition \ref{superformula},
 \begin{equation}\label{asymqn}
\hat x_{2,n}= f^{q_{n+1}}(0)=\prod_{k\leq n}S_{4,k}|\hat x_{1,0}|\sim e^{C_u(f)\lambda_{u}(f)^ne^u_4(f)}.
  \end{equation}
  In the same way 
 $$g^{q_{n+1}}(0)\sim e^{C_u(g)\lambda_{u}(g)^ne^u_4(g)}.$$
Notice that, if $h$ conjugates $f$ and $g$ then $h(f(U_f))=g(U_g)$. Suppose now that $h$ is H\"older continuous with exponent
 $\beta>0$ then $$\lim_{n\to\infty}\frac{g^{q_{n+1}}(0)}{\left( f^{q_{n+1}}(0)\right)^{\beta}}
 \sim \lim_{n\to\infty}e^{C_u(g)\lambda_{u}(g)^ne^u_4(g)-\beta C_u(f)\lambda_{u}(f)^ne^u_4(f)}=\infty$$ where we used that $C_u(f)<0$ 
 (see Lemma \ref{asymalpha}) and that $\lambda_u(g)<\lambda_u(f)$. Finally, $h$ cannot be H\"older. 
\end{proof}
We would like to compare Proposition \ref{nothol} with the main Theorem in \cite{P3}. The author proves that, if $f,g\in\W$ correspond to smooth circle maps with different critical exponents 
$\ell_f\neq\ell_g$ both belonging to $(2,+\infty]$, then 
$h$ is a quasi-symmetric homeomorphism, in particular it is H\"older. The reason for the regularity of $h$ under such a weak condition is related with the fact that, for functions with critical 
exponent $\ell>2$, the sequence $\alpha_n$ is bounded away from zero (see \cite{P1}). In the setting of our paper, of functions with critical exponent $\ell<2$,
the sequence $\alpha_n$ goes
 to zero double exponentially fast and this causes the loss of regularity of the conjugacy. 
\subsection{Proof of Theorem \ref{LM}}
\begin{prop}\label{lipc1}
 Let $1<\ell<2$. If $f,g\in \W$ with critical exponent $\ell$ and if $h$ is the topological conjugacy 
between $f$ and $g$ then, the following holds.
\begin{eqnarray*}
 h \text{ is a Lipschitz homeomorphism} &\implies& C_u(f)=C_u(g),\\
 h \text{ is a } \mathcal C^1 \text{ diffeomorphism }&\implies& C_u(f)=C_u(g), C_-(f)=C_-(g).
\end{eqnarray*}
\end{prop}
\begin{proof}
Without loss of generality we may assume that $C_u(g)\geq C_u(f)$.
 From (\ref{asymqn}) and Proposition \ref{superformula} we get 
 $$D=\limsup_{n\to\infty}\frac{g^{q_{n+1}}(0)}{f^{q_{n+1}}(0)}=\limsup_{n\to\infty}e^{(C_u(g)-C_u(f))\lambda_{u}^ne^u_4+(C_-(g)-C_-(f))(-1)^n}.$$
 Observe that $h(f^{q_{n+1}}(0))=(g^{q_{n+1}}(0))$. If $h$ is a Lipschitz homeomorphism then $D$ is bounded by a positive constant. Hence, $C_u(g)=C_u(f)$. If $h$ is a $\mathcal C^1$ diffeomorphism then $D=\lim_{n\to\infty}{g^{q_{n+1}}(0)}/{f^{q_{n+1}}(0)}$. Hence, $C_u(g)=C_u(f)$ and 
 $C_-(g)=C_-(f)$.
\end{proof}
For all $n\in\N$ we define the following intervals:
\begin{itemize}
 \item[-] $A_n(f)=\left[\hat x_{1,n},0 \right]$,
 \item[-] $B_n(f)=\left[0, \hat x_{3,n}\right]$,
 \item[-] $C_n(f)=\left[\hat x_{3,n}, \hat x_{4,n}\right]$,
 \item[-] $D_n(f)=\left[\hat x_{4,n}, \hat x_{1,n-1}\right]$,
\end{itemize}
and their iterates
\begin{itemize}
 \item[-] $A_n^i(f)=f^i(A_n(f))$ for $0\leq i<q_{n-1}$,
 \item[-] $B_n^i(f)=f^i(B_n(f))$ for $0\leq i<q_{n}$,
 \item[-] $C_n^i(f)=f^i(C_n(f))$ for $0\leq i<q_{n}$,
 \item[-] $D_n^i(f)=f^i(D_n(f))$ for $0\leq i<q_{n}$.
\end{itemize}
Observe that, for all $n\in\N$, $$\mathscr P_n(f)=\left\{A_n^i(f), B_n^j(f), C_n^j(f), D_n^j(f)|0\leq i<q_{n-1}, 0\leq j<q_{n}\right\}$$ is a partition of the domain of $f$ and that 
\begin{itemize}
 \item[-] $h(A_n^i(f))=A_n^i(g)$ for $0\leq i<q_{n-1}$,
 \item[-] $h(B_n^i(f))=B_n^i(g)$ for $0\leq i<q_{n}$,
 \item[-] $h(C_n^i(f))=C_n^i(g)$ for $0\leq i<q_{n}$,
 \item[-] $h(D_n^i(f))=D_n^i(g)$ for $0\leq i<q_{n}$.
\end{itemize}
For all $n\in\N$ we define the mesurable function $Dh_n:[0,1]\vers\R^+$ as $$Dh_n(x)=\frac{|h(I)|}{|I|},$$ with $I\in\mathscr P_n$ and $x\in\text{Interior}{(I)}$. 
Observe that, if $T=\cup I_i$ with $I_i\in\mathscr P_n$, then for all $m\geq n$,
\begin{equation}\label{ht}
 |h(T)|=\int_TDh_m.
\end{equation}
\begin{lem}\label{length}
Let $1<\ell<2$. There exists $C>0$ such that, for every interval $I\in \mathscr P_n$, $I\neq C_n^i$, then
$$
|I|\geq \frac{1}{C} e^{C_u(f)\lambda_u^{n}\frac{\lambda_u}{\lambda_u-1}}.
$$
%\TDD{I am not completely sure about this formula why $\frac{\lambda_u}{\lambda_u-1}$ appears? Check again!!!}
\end{lem}
\begin{proof}
 Let $J\in\mathscr P_{n-1}$ with $I\subset J$. By the construction of $\mathscr P_n$ from $\mathscr P_{n-1}$ we see that
 $$\frac{|I|}{|J|}\geq\left(1+O\left(\alpha_{n-3}^{\frac{1}{\ell}}\right)\right)\min\left\{1-x_{4,n},x_{3,n}, S_{1,n-1}\right\},$$ where we used Lemma \ref{affdiffeo}.
 From Lemma \ref{asymalpha}, Lemma \ref{xs}, Corollary \ref{s1} and Proposition \ref{superformula}, there exists a positive constant $K>0$ such that 
 \begin{enumerate}
  \item $1-x_{4,n}\geq e^{C_u\lambda_u^{n}e_2^u-K\lambda_s^n},$
  \item $x_{3,n}\geq e^{C_u\lambda_u^{n}(e_2^u+e^u_3)-K\lambda_s^n},$
  \item $S_{1,n-1}\geq e^{C_u\lambda_u^{n-1}(\frac{e_2^u}{\ell})-K\lambda_s^n-\frac{1}{\ell}\log\ell}$.
 \end{enumerate}
From Proposition \ref{superformula}, Lemma \ref{eighvector} and the value of $\lambda_u$, we get that 
$$\frac{|I|}{|J|}\geq e^{C_u\lambda_u^{n}-K\lambda_s^n}.$$
 Finally, because $J\neq C_{n-1}^j$, we can repeat this estimates and we get
 $$|I|\geq e^{\sum_{k=0}^{n} C_u\lambda_u^{k}-K\lambda_s^k}.$$ The lemma follows.
\end{proof}

\begin{lem}\label{lengthup}
Let $1<\ell<2$. There exists $C>0$ such that, for every interval $I\in \mathscr P_n$, $I\neq C_n^i$, then
$$
|I|\leq C e^{\frac{C_u(f)\lambda_u^{n-3}}{\ell}}.
$$
\end{lem}
\begin{proof} By Lemma \ref{affdiffeo},
\begin{itemize}
  \item[-] $\max |A_{n}^i|\leq \max |B_{n}^i|\left(1+O\left(\alpha_{n-3}^{\frac{1}{\ell}}\right)\right)$,
  \item[-] $\max |B_{n}^i|\leq\max\left[ x_{3,n}\max |A_{n-1}^i|,  \max |D_{n-1}^i|\right]\left(1+O\left(\alpha_{n-3}^{\frac{1}{\ell}}\right)\right)$,
  \item[-] $\max |D_{n}^i|\leq\max\left[\left(1-x_{4,n}\right)\max |A_{n-1}^i|,  S_{1,n-1}\max |B_{n-1}^i|\right]\left(1+O\left(\alpha_{n-3}^{\frac{1}{\ell}}\right)\right)$,
 \end{itemize}
 and by recursion 
 \begin{itemize}
 \item[-] $\max |D_{n}^i|\leq O\left(\max\left[ 1-x_{4,n}, S_{1,n-1}\right]\right)$,
  \item[-] $\max |B_{n}^i|\leq O\left(\max\left[ x_{3,n}, 1-x_{4,n-1}, S_{1,n-2}\right]\right)$,
  \item[-] $\max |A_{n}^i|\leq O\left(\max\left[ x_{3,n-1}, 1-x_{4,n-2}, S_{1,n-3}\right]\right)$.
 \end{itemize}
 From Lemma \ref{xs}, Corollary \ref{s1} and Proposition \ref{superformula} 
 \begin{itemize}
  \item[-] $1-x_{4,n-2}=O\left(e^{C_u\lambda_u^{n-2}e_2^u}\right)$,
  \item[-] $x_{3,n-1}=O\left(e^{C_u\lambda_u^{n-1}\left(e_2^u+e^u_3\right)}\right)$,
  \item[-] $S_{1,n-3}=O\left(e^{C_u\lambda_u^{n-3}\left(\frac{e_2^u}{\ell}\right)}\right)$.
 \end{itemize}
By Lemma \ref{eighvector}, we have that ${1}/{\ell}\leq\min\left[\lambda_u^2\left(1+e_3^u\right), \lambda_u \right]$ and by the fact that $C_u<0$, the lemma follows.
\end{proof}
\begin{lem}\label{Dh}
Let $1<\ell<2$. If $f,g\in\W$ with critical exponent $\ell$ and $C_u=C_u(f)=C_u(g)$ then 
 $$\log\frac{Dh_{n+1}}{Dh_n}=O\left(\alpha_{n-2}(f)^{\frac{1}{\ell}}+\left|\left(C_s(g)-C_s(f)\right)\lambda_s^n\right|\right).$$  
\end{lem}
\begin{proof}
The hypothesis $C_u=C_u(f)=C_u(g)$ implies that $\log\left({\alpha_n(g)}/{\alpha_n(f)}\right)=O\left(\lambda_s^{n-1}\right)$, see Lemma \ref{asymalpha}.
By Lemma \ref{affdiffeo} and point $1$ of Lemma \ref{prev} we get 
\begin{eqnarray*}  
\frac{Dh_{n+1|A^i_{n+1}(f)}}{Dh_{n|A^i_{n+1}(f)}}&=&\frac{{|A^i_{n+1}(g)|}/{|A^i_{n+1}(f)|}} {{|B^i_{n}(g)|}/{|B^i_{n}(f)|}}=
\frac{{|A^i_{n+1}(g)|}/{|B^i_{n}(g)|}}{{|A^i_{n+1}(f)|}/{|B^i_{n}(f)|}}\\&=&\frac{{|A_{n+1}(g)|}/{|B_{n}(g)|}}
{{|A_{n+1}(f)|}/{|B_{n}(f)|}}\left(1+O\left({\alpha_{n-2}}(f)^{\frac{1}{\ell}}\right)\right)\\&=&\frac{{x_{2,n}(g)}/{x_{3,n}(g)}}{{x_{2,n}(f)}/{x_{3,n}(f)}}\left(1+O\left({\alpha_{n-2}}(f)^{\frac{1}{\ell}}\right)\right)\\&=&
\frac{1-S_{1,n}(g)}{1-S_{1,n}(f)}\left(1+O\left({\alpha_{n-2}}(f)^{\frac{1}{\ell}}\right)\right)\\&=&1+O\left({\alpha_{n-2}}(f)^{\frac{1}{\ell}}\right).
\end{eqnarray*}
Observe that, for $i\geq q_{n-1}$, $B^i_{n+1}(f)=D^{i-q_{n-1}}_{n}(f)$ and in particular $$\frac{Dh_{n+1|B^i_{n+1}(f)}}{Dh_{n|B^i_{n+1}(f)}}=1.$$
Let $i< q_{n-1}$. With a similar calculation as before, we get 
\begin{eqnarray*}
\frac{Dh_{n+1|B^i_{n+1}(f)}}{Dh_{n|B^i_{n+1}(f)}}&=&\frac{{|B^i_{n+1}(g)|}/{|B^i_{n+1}(f)|}}{{|A^i_{n}(g)|}/{|A^i_{n}(f)|}}
=\frac{{|B_{n+1}(g)|}/{|A_{n}(g)|}}
{{|B_{n+1}(f)|}/{|A_{n}(f)|}}\left(1+O\left({\alpha_{n-2}}(f)^{\frac{1}{\ell}}\right)\right)\\&=&\frac{x_{3,n+1}(g)}{x_{3,n+1}(f)}\left(1+O\left({\alpha_{n-2}}(f)^{\frac{1}{\ell}}\right)\right)\\&=& 
e^{\left(C_s(g)-C_s(f)\right)\lambda_s^{n+1}\left(e_2^s+e_3^s\right)+O\left(e^{\frac{C_u\lambda_u^{n-3}}{\ell}}\right)}\left(1+O\left({\alpha_{n-2}}(f)^{\frac{1}{\ell}}\right)\right)\\&=&1+O\left(\alpha_{n-2}(f)^{\frac{1}{\ell}}+\left|\left(C_s(g)-C_s(f)\right)\lambda_s^n\right|\right),
\end{eqnarray*}
where we used Lemma \ref{xs}.
\\
Observe that, for $i\geq q_{n-1}$, $C^i_{n+1}(f)=C^{i-q_{n-1}}_{n}(f)$ and in particular $$\frac{Dh_{n+1|C^i_{n+1}(f)}}{Dh_{n|C^i_{n+1}(f)}}=1.$$
Let $i< q_{n-1}$. With a similar calculation as before, we get 
\begin{eqnarray*}
\frac{Dh_{n+1|C^i_{n+1}(f)}}{Dh_{n|C^i_{n+1}(f)}}&=&\frac{{|C_{n+1}(g)|}/{|C_{n+1}(f)|}}{{|A_{n}(g)|}/{|A_{n}(f)|}}\left(1+O\left({\alpha_{n-2}}(f)^{\frac{1}{\ell}}\right)\right)\\
&=&\frac{x_{4,n+1}(g)-x_{3,n+1}(g)}{x_{4,n+1}(f)-x_{3,n+1}(f)}\left(1+O\left({\alpha_{n-2}}(f)^{\frac{1}{\ell}}\right)\right)\\&=&
\frac{1-[x_{3,n+1}(g)+(1-x_{4,n+1}(g))]}{1-[x_{3,n+1}(f)+(1-x_{4,n+1}(f))]}\left(1+O\left({\alpha_{n-2}}(f)^{\frac{1}{\ell}}\right)\right)\\&=&\frac{1-O\left(\alpha_{n+2}(g)\right)}{1-O\left(\alpha_{n+2}(f)\right)}  
\left(1+O\left({\alpha_{n-2}}(f)^{\frac{1}{\ell}}\right)\right)=1+O\left({\alpha_{n-2}}(f)^{\frac{1}{\ell}}\right),
\end{eqnarray*}
where we used Lemma \ref{xs} and Lemma \ref{asymalpha}.
\\
We denote by $\Delta_n(f)=\left(f^{q_{n+1}}(0),f^{-q_{n}}(0)\right)=\left(\hat x_2,\hat x_3\right)$ and its iterates by $\Delta^i_n(f)=f^i\left(\Delta_n(f)\right)$ for $0\leq i<q_n$. 
Observe that, for $i\geq q_{n-1}$, $D^i_{n+1}(f)=\Delta^{i-q_{n-1}}_{n}(f)$ and in particular 
\begin{eqnarray*}
\frac{Dh_{n+1|D^i_{n+1}(f)}}{Dh_{n|D^i_{n+1}(f)}}&=&\frac{{|\Delta^{i-q_{n-1}}_{n}(g)|}/{|\Delta^{i-q_{n-1}}_{n}(f)|}}
{{|B^{i-q_{n-1}}_{n}(g)|}/{|B^{i-q_{n-1}}_{n}(f)|}}\\&=&\frac{{|\Delta_{n}(g)|}/{|\Delta_{n}(f)|}}
{{|B_{n}(g)|}/{|B_{n}(f)|}}\left(1+O\left({\alpha_{n-2}}(f)^{\frac{1}{\ell}}\right)\right)\\&=&\frac{{|\Delta_{n}(g)|}/{|B_{n}(g)|}}
{{|\Delta_{n}(f)|}/{|B_{n}(f)|}}\left(1+O\left({\alpha_{n-2}}(f)^{\frac{1}{\ell}}\right)\right)\\
&=&\frac{S_{1,n}(g)}{S_{1,n}(f)}\left(1+O\left({\alpha_{n-2}}(f)^{\frac{1}{\ell}}\right)\right)=1+O\left({\alpha_{n-2}}(f)^{\frac{1}{\ell}}\right),
\end{eqnarray*}
where we used point $1$ of Lemma \ref{prev}.
Let $i< q_{n-1}$, then 
\begin{eqnarray*}
\frac{Dh_{n+1|D^i_{n+1}(f)}}{Dh_{n|D^i_{n+1}(f)}}&=&\frac{{|D_{n+1}(g)|}/{|D_{n+1}(f)|}}{{|A_{n}(g)|}/{|A_{n}(f)|}}\left(1+O\left({\alpha_{n-2}}(f)^{\frac{1}{\ell}}\right)\right)\\
&=&\frac{{|D_{n+1}(g)|}/{|A_{n}(g)|}}{{|D_{n+1}(f)|}/{|A_{n}(f)|}}\left(1+O\left({\alpha_{n-2}}(f)^{\frac{1}{\ell}}\right)\right)\\ 
&=&\frac{1-x_{4,n+1}(g)}{1-x_{4,n+1}(f)}\left(1+O\left({\alpha_{n-2}}(f)^{\frac{1}{\ell}}\right)\right)\\&=&
e^{\left(C_s(g)-C_s(f)\right)\lambda_s^{n+1}+O\left(e^{\frac{C_u\lambda_u^{n-3}}{\ell}}\right)}\left(1+O\left({\alpha_{n-2}}(f)^{\frac{1}{\ell}}\right)\right)\\&=&1+O\left({\alpha_{n-2}}(f)^{\frac{1}{\ell}}+\left|\left(C_s(g)-C_s(f)\right)\lambda_s^n\right|\right),
\end{eqnarray*}
where we used Lemma \ref{xs} and Lemma \ref{asymalpha}.
\end{proof}
Observe that every boundary point of an interval in $\mathscr P_{n}$ is in the orbit of the critical point $\text{Orbit}(0)$. As a consequence, if $x\in\text{Orbit}(0)$, then $Dh_{{n}}(x)$ is well 
defined and Lemma \ref{Dh} implies that the following limit $$D(x)=\lim_{n\to\infty} Dh_{{n}}(x),$$ exists. Observe that, 
\begin{equation}\label{boundD}
0<\inf_{x}D(x)<\sup_xD(x)<\infty.
\end{equation}
\begin{lem}\label{Dcont}
Let $1<\ell<2$. If $f,g\in\W$ with critical exponent $\ell$ and $C_u=C_u(f)=C_u(g)$ then, the function $D:[0,1]\setminus\text{Orbit}(0)\to(0,\infty)$ is continuous. 
 In particular for all $n\in\N$, $$\left|\log\frac{D(x)}{Dh_{{{n}}}(x)}\right|=O\left(\alpha_{n-2}(f)^{\frac{1}{\ell}}+\left|\left(C_s(g)-C_s(f)\right)\lambda_s^{n}\right|\right).$$
 Moreover, if $x,y\in I\in \mathscr P_{{n}}$,
 $$\left|\log\frac{D(x)}{D(y)}\right|=O\left(\alpha_{n-2}(f)^{\frac{1}{\ell}}+\left|\left(C_s(g)-C_s(f)\right)\lambda_s^{n}\right|\right).$$
\end{lem}
\begin{proof}
Fix $n_0\in\N$, then by Lemma \ref{Dh}
\begin{eqnarray*}
\left|\log\frac{D(x)}{Dh_{{{n_0}}}(x)}\right|&\leq&\sum_{k\geq 0}\left|\log\frac{Dh_{{{n_0+k+1}}}(x)}{Dh_{{{n_0+k}}}(x)}\right|\\
&\leq&\sum_{k\geq 0}O\left(\alpha_{k-2}(f)^{\frac{1}{\ell}}+\left|\left(C_s(g)-C_s(f)\right)\lambda_s^{k}\right|\right)\\&=&
O\left(\alpha_{n_0-2}(f)^{\frac{1}{\ell}}+\left|\left(C_s(g)-C_s(f)\right)\lambda_s^{n_0}\right|\right).
\end{eqnarray*}
As a consequence, if $x,y\in I\in \mathscr P_{{n_0}}$, then 
$$\left|\log\frac{D(x)}{D(y)}\right|\leq \left|\log\frac{D(x)}{Dh_{{{n_0}}}(x)}\right|+\left|\log\frac{D(y)}
{Dh_{{{n_0}}}(y)}\right|=O\left(\alpha_{n_0-2}(f)^{\frac{1}{\ell}}+\left|\left(C_s(g)-C_s(f)\right)\lambda_s^{n_0}\right|\right)$$ which means that $D$ is continuous.
\end{proof}

Let $T$ be an interval in $\text{Domain}(f)$, then by (\ref{ht}) %\TDD{unclear!!! check and Corollary \ref{dpos}} 
we have
\begin{equation}\label{hD}
 |h(T)|=\lim_{n\to\infty}|h_{n}(T)|=\lim_{n\to\infty}\int_{T}Dh_{n}\leq\int_{T}D<\sup_x D(x)|T|.
\end{equation}
\begin{prop}\label{lipc1}
 Let $1<\ell<2$ and let $f,g\in\W$ with critical exponent $\ell$. If $h$ is the topological conjugacy 
between $f$ and $g$, then the following holds.
\begin{eqnarray*}
 h \text{ is a Lipschitz homeomorphism} &\Leftarrow& C_u(f)=C_u(g),\\
 h \text{ is a } \mathcal C^1 \text{ diffeomorphism }&\Leftarrow& C_u(f)=C_u(g), C_-(f)=C_-(g).
\end{eqnarray*}
\end{prop}
\begin{proof}
Observe that, by (\ref{hD}) and by Lemma \ref{Dcont} it follows immediately that if $C_u(f)=C_u(g)$ then $h$ is a Lipschitz homeomorphism.
\\
Assume now that $C_u(f)=C_u(g)$ and $C_-(f)=C_-(g)$. We prove that, under these conditions, $D$ extends to a continuous function of the $\text{Domain}(f)$. This extension together with  (\ref{hD}) implies that $h$ is differentiable with $Dh(x)=D(x)$ for all $x\in\text{Domain}(f)$ and by (\ref{boundD}), $h$ is a diffeomorphism.
The only problematic points to extend $D$ could be $c_k=f^k(0)$ with $k\geq 0$. 
We prove that $$\lim_{x\uparrow c_k}D(x)=\lim_{x\downarrow c_k}D(x),$$ and in particular $D$ extends to a continuous function $D:\text{Domain}(f)\to\R$.
\\
Fix $k\geq 0$ and let $n\in\N$ large enough, such that $q_{n-1}>k$. Define $$D_+(c_k)=\lim_{x\downarrow c_k} D(x)=
\lim_{n\to\infty}\frac{|B_{2n}^k(g)|}{|B_{2n}^k(f)|},$$
and 
$$D_-(c_k)=\lim_{x\uparrow c_k} D(x)=
\lim_{n\to\infty}\frac{|A_{2n}^k(g)|}{|A_{2n}^k(f)|}.$$
By Lemma \ref{affdiffeo} and Lemma \ref{xs} it follows that
\begin{eqnarray*}
 \frac{D_+(c_k)}{D_-(c_k)}&=&\lim_{n\to\infty}\frac{{|B_{2n}^k(g)|}/{|A_{2n}^k(g)|}}{{|B_{2n}^k(f)|}/{|A_{2n}^k(f)|}}
 =\lim_{n\to\infty}\frac{{|B_{2n}(g)|}/{|A_{2n}(g)|}}{{|B_{2n}(f)|}/{|A_{2n}(f)|}}\\
 &=&\lim_{n\to\infty}\frac{{x_{3,2n}(g)}/{-x_{1,2n}(g)}}{{x_{3,2n}(f)}/{-x_{1,2n}(f)}}
 =\lim_{n\to\infty} e^{O\left(\lambda_s^{2n}\right)}=1. 
\end{eqnarray*}
\end{proof}
\begin{prop}
Let $1<\ell<2$, $\beta=\frac{\lambda_u-1}{\lambda_u^4\ell}>0$ and $f,g\in\W$ with critical exponent $\ell$. If $h$ is the topological conjugacy 
between $f$ and $g$, then the following holds.
\begin{eqnarray*}
h \text{ is a } \mathcal C^{1+\beta} \text{ diffeomorphism }&\iff& C_u(f)=C_u(g), C_-(f)=C_-(g), C_s(f)=C_s(g)
\end{eqnarray*} 
\end{prop}
\begin{proof}
Assume that $h$ is $\mathcal C^{1+\beta}$, then by Proposition \ref{lipc1}, $C_u(f)=C_u(g)$ and $C_-(f)=C_-(g)$. Because $C_u=C_u(f)=C_u(g)$ we have that $\log\left({\alpha_n(g)}/{\alpha_n(f)}\right)=O\left(\lambda_s^{n-1}\right)$, see Lemma \ref{asymalpha}. Observe that  
\begin{eqnarray*}
 \frac{Dh(f^{q_{n+1}}(0))}{Dh(f^{q_{n}}(0))}&=&\frac{Dg^{q_{n-1}}(g^{q_{n}}(0))}{Df^{q_{n-1}}(f^{q_{n}}(0))}\\
 &=&\frac{Dq_{s_{n}(g)}(0)\left({1-x_{2,n}(g)}/{x_{1,n}(g)}\right)}{Dq_{s_{n}(f)}(0)\left({1-x_{2,n}(f)}/{x_{1,n}(f)}\right)}\left(1+O\left(\alpha_{n-2}(f)^{\frac{1}{\ell}}\right)\right)\\
 &=&\frac{S_{5,n}(g)\left(1+O\left(s_n(g)\right)\right)\left({1-x_{2,n}(g)}/{x_{1,n}(g)}\right)}{S_{5,n}(f)\left(1+O\left(s_n(f)\right)\right)\left({1-x_{2,n}(f)}{x_{1,n}(f)}\right)}\left(1+O\left(\alpha_{n-2}(f)^{\frac{1}{\ell}}\right)\right)\\
 &=&\frac{{S_{5,n}(g)}/{x_{1,n}(g)}}{{S_{5,n}(f)}/{x_{1,n}(f)}}\left(1+O\left(x_{2,n}(f)\right)+O\left(s_{n}(f)\right)+O\left(\alpha_{n-2}(f)^{\frac{1}{\ell}}\right)\right)\\
 &=&e^{-(C_s(g)-C_s(f))\lambda_s^n(e_2^s+e_3^s-e_4^s-e_5^s)}\left(1+O\left(e^{\frac{C_u(e^u_2\lambda_u^{n-4}}{\ell}}\right)\right),
% &=&e^{\left(C_s(g)-C_s(f)\right)\lambda_s^ne_5^s+O\left(e^\frac{C_u(f)\lambda_u^{n-4}}{\ell}\right)}}\left(1+O\left(\alpha_{n-2}(f)^{\frac{1}{\ell}}\right)\right)
 \end{eqnarray*}
 % &&e^{-(C_s(g)-C_s(f))\lambda_s^n(e_2^s+e_3^s-e_4^s)+O(\sqrt{e^{C_u(f)(e^u_2+e^u_3)\lambda_u^{n-3}}})}\\
where we used Proposition \ref{superformula}, Lemma \ref{xs}, point $3$ of Lemma \ref{prev} and Lemma \ref{asymalpha}. From Lemma \ref{eighvector} it follows that  
$$(e_2^s+e_3^s-e_4^s-e_5^s)\neq 0.$$ Moreover, by the hypothesis that $h$ is $\mathcal C^{1+\beta}$ and by (\ref{boundD}) we get  
\begin{eqnarray*}
{Dh(f^{q_{n+1}+1}(U))}{Dh(f^{q_{n}+1}(U))}&=&O\left(|f^{q_{n}}(0),f^{q_{n+1}}(0)|^{\beta}\right)=O\left(|f^{q_{n}}(0),0|^{\beta}\right)\\&=&O\left(e^{\beta C_u(f)\lambda_u^n\left(e_2^u+e_3^u-e_4^u\right)}\right),
\end{eqnarray*}
where, by Lemma \ref{eighvector},
$$(e_2^u+e_3^u-e_4^u)\neq 0.$$
The two above estimates imply that $C_s(f)=C_s(g)$. This proves one side of the implication.
 \\ 
Assume now that $C_u(f)=C_u(g), C_-(f)=C_-(g)$ and $C_s(f)=C_s(g)$. Let $x,y\in K_f$ and choose $n$ maximal such
that, there exists $I\in\mathscr P_n$ containing $x$ and $y$. Because of the fact that $B_{n+1}^{i+q_{n-1}}=D_n^i$ and by the maximality of $n$, we have $I\neq D_n^i$. Moreover, because $x,y\in K_f$, then $I\neq C_n^i$. It remains to study two cases: 
either $I=A_n^i$ or $I=B_n^i$. 
\\
Suppose that $I=A_n^i$, then by the maximality of $n$, $x\in D_{n+1}^i$ and $y\in B_{n+1}^i$. From this, the fact that 
$1-x_{4,n+1},x_{3,n+1}<<1$, Lemma \ref{affdiffeo}  and Lemma \ref{length} it follows that 
\begin{equation}\label{hol1}
|x-y|\geq\frac{1}{2}|A_n^i|\geq\frac{1}{2C}e^{C_u(f)\lambda_u^{n}\frac{\lambda_u}{\lambda_u-1}}.
\end{equation}
Moreover, by Lemma \ref{Dcont} and by Lemma \ref{asymalpha} we have
\begin{equation}\label{hol2}
 \log\frac{Dh(x)}{Dh(y)}=O\left(e^{\frac{C_u\lambda_u^{n-3}}{\ell}}\right).
\end{equation}
By (\ref{hol1}) and (\ref{hol2}) we find
$$\frac{|Dh(x)-Dh(y)|}{|x-y|^\beta}=O\left(e^{\frac{C_u\lambda_u^{n-3}}{\ell}-\beta C_u\lambda_u^n\frac{\lambda_u}{\lambda_u-1}}\right)=O(1),$$ 
for $0<\beta\leq \frac{\lambda_u-1}{\lambda_u^4\ell}$, see Lemma \ref{eighvector}.
\\
Suppose now that $I=B_n^i$, then $x\in A_{n+1}^i$ and $y\in D_{n+1}^{i+q_{n-1}}$. Moreover, $x\leq f^{q_{n+1}+i}(0)\leq y$. Let $m\geq n+1$ maximal such that 
$x\in A_{m}^{i+q_{n+1}}$ and $y\in B_{m}^{i+q_{n+1}}$. As a consequence, using Lemma \ref{length} we have
\begin{equation}\label{hol3}
|x-y|\geq\frac{1}{2}\min\left(|A_{m}^{i+q_{n+1}}|,|B_{m}^{i+q_{n+1}}|\right)\geq\frac{1}{2C}e^{C_u(f)\lambda_u^{m}\frac{\lambda_u}{\lambda_u-1}}.
\end{equation}
Observe now that, by Lemma \ref{Dcont} and Lemma \ref{asymalpha} we have
\begin{equation}\label{hol4} 
 \left|\log\frac{Dh(x)}{Dh(y)}\right|=\left|\log\frac{Dh(y)}{Dh(f^{q_{n+1}+i}(0))}\right|+\left|\log\frac{Dh(x)}{Dh(f^{q_{n+1}+i}(0))}\right|=
 O\left(e^{\frac{C_u\lambda_u^{m-3}}{\ell}}\right)
\end{equation}
By (\ref{hol3}) and (\ref{hol4}) we find
$$\frac{|Dh(x)-Dh(y)|}{|x-y|^\beta}=O\left(e^{\frac{C_u\lambda_u^{m-3}}{\ell}-\beta C_u\lambda_u^m\frac{\lambda_u}{\lambda_u-1}}\right)=O(1),$$ 
for $0<\beta\leq \frac{\lambda_u-1}{\lambda_u^4\ell}$, see Lemma \ref{eighvector}.
\end{proof}

\begin{prop}
Let $1<\ell<2$. If $f,g\in\W$ with critical exponent $\ell$, then 
the topological conjugacy between $f$ and $g$ is $\mathcal C^{\beta} $ with $\beta=\frac{C_u(g)}{C_u(f)}\frac{\lambda_u-1}{\lambda_u^4\ell}>0$.
\end{prop}
\begin{proof}
 Let $x,y\in K_f$ and choose $n$ maximal such
that there exists $I\in\mathscr P_n$ containing $x$ and $y$. Because of the fact that $B_{n+1}^{i+q_{n-1}}=D_n^i$ and by the maximality of $n$, then $I\neq D_n^i$. Moreover, because $x,y\in K_f$, then $I\neq C_n^i$. It remains to study two cases: 
either $I=A_n^i$ or $I=B_n^i$. 
\\
Suppose that $I=A_n^i$, then by the maximality of $n$, $x\in D_{n+1}^i$ and $y\in B_{n+1}^i$. From this, the fact that 
$1-x_{4,n+1},x_{3,n+1}<<1$, Lemma \ref{affdiffeo}  and Lemma \ref{length} it follows that 
\begin{equation}\label{chol1}
|x-y|\geq\frac{1}{2}|A_n^i|\geq\frac{1}{2C}e^{C_u(f)\lambda_u^{n}\frac{\lambda_u}{\lambda_u-1}}.
\end{equation}
Observe now that, by Lemma \ref{lengthup},
\begin{equation}\label{chol2}
|h(y)-h(x)|\leq |h(I)|=O\left(e^{\frac{C_u(g)\lambda_u^{n-3}}{\ell}}\right).
\end{equation}
By (\ref{chol1}) and (\ref{chol2}) we find
$$\frac{|h(y)-h(x)|}{|x-y|^\beta}=O\left(e^{\frac{C_u(g)\lambda_u^{n-3}}{\ell}-\beta C_u(f)\lambda_u^n\frac{\lambda_u}{\lambda_u-1}}\right)=O(1,)$$ 
for $0<\beta\leq\frac{C_u(g)}{C_u(f)}\frac{\lambda_u-1}{\lambda_u^4\ell}$, see Lemma \ref{eighvector}.
\\
Suppose now that $I=B_n^i$, then $x\in A_{n+1}^i$ and $y\in D_{n+1}^{i+q_{n-1}}$. Moreover, $x\leq f^{q_{n+1}+i}(0)\leq y$. Let $m\geq n+1$ maximal such that 
$x\in A_{m}^{i+q_{n+1}}$ and $y\in B_{m}^{i+q_{n+1}}$. As a consequence, using Lemma \ref{length} we have
\begin{equation}\label{chol3}
|x-y|\geq\frac{1}{2}\min\left(|A_{m}^{i+q_{n+1}}|,|B_{m}^{i+q_{n+1}}|\right)\geq\frac{1}{2C}e^{C_u(f)\lambda_u^{m}\frac{\lambda_u}{\lambda_u-1}}.
\end{equation}
Observe now that, by Lemma \ref{lengthup}  we have
\begin{equation}\label{chol4}
|h(y)-h(x)|\leq |h(A_{m}^{i+q_{n+1}})|+|h(B_{m}^{i+q_{n+1}})|=O\left(e^{\frac{C_u(g)\lambda_u^{m-3}}{\ell}}\right).
\end{equation}
By (\ref{chol3}) and (\ref{chol4}) we find
$$\frac{|h(y)-h(x)|}{|x-y|^\beta}=O\left(e^{\frac{C_u(g)\lambda_u^{m-3}}{\ell}-\beta C_u(f)\lambda_u^m\frac{\lambda_u}{\lambda_u-1}}\right)=O(1),$$ 
for $0<\beta\leq\frac{C_u(g)}{C_u(f)}\frac{\lambda_u-1}{\lambda_u^4\ell}$, see Lemma \ref{eighvector}.
\end{proof}
\section{Appendix I: technical lemmas}\label{AppendixII}
We collect here technical lemmas which are needed in Subsections \ref{jumpoutdiff} and \ref{tophyp}. The proofs of the lemmas are based on calculations.
\\
 Let $D\subset L\subset \R^5\times X_{3+\epsilon}\times X_{3+\epsilon}\times X_{3+\epsilon}$ as defined in Section \ref{manifold1} or a bounded closed set. Refer to Section \ref{manifold1} for the definitions of $\epsilon^*$, $C^*$ and $C_u^*$ and to Definition \ref{zoomop} for the definition of the zoom operator $Z_{I}$.
 \begin{lem}\label{D111} On $D=D_{1,1,1}$ the following holds:
 $$
 \frac{\ell S_5 q^{-1}_s(S_1S_2S_3)}{S_1S_2S_3}=O(1).
 $$
 \end{lem}
 
 \begin{proof} A calculation shows that,
 $$
 \frac{\ell S_5 q^{-1}_s(S_1S_2S_3)}{S_1S_2S_3}\le \frac{\ell S_5}{Dq_s(0)}=
 \frac{1-s^\ell}{1-s}=O(1),
 $$
 where we used $\log s\leq {-{1}/{\ell\lambda_u}+{1}/{\ell\lambda_s}+{1}/{\ell-1}}$, see Lemma \ref{eighvector}. 
 \end{proof}
 \begin{lem}\label{C2}  For every $\delta>0$ and $C^*\geq 1$ we have
 $
\left |Z_{\left[q_s^{-1}\left( 1-e^{y_2} \right),1\right]}q_s\right|\le \delta
 $
 when $C_u^*$ is large enough.
 \end{lem}
 
 \begin{proof} A calculation shows
 \begin{equation}\label{ZS2}
  \left|Z_{\left[q_s^{-1}\left( 1-e^{y_2} \right),1\right]}q_s\right|=
   \left|Z_{\left[q_s^{-1}\left( 1-S_2 \right),1\right]}q_s\right|\le 
   \left|Z_{\left[1-S_2 ,1\right]}q_0\right|
   \le 2(\ell-1)\frac{S_2}{1-S_2}\le \delta
 \end{equation}
  where we used $S_2\leq e^{C^*}e^{-C_u^*}$, see Lemma \ref{eighvector}.
 \end{proof}
 
 \begin{lem}\label{s1235} For every $\delta>0$ and $C^*\geq 1$ we have
 $$
1-2\delta \le \frac{\ell S_5 q^{-1}_s(S_1S_2S_3)}{S_1S_2S_3}\le 1+2\delta
 $$
  when $C_u^*$ is large enough.
 \end{lem}
 
 \begin{proof} A calculation shows for $C_u^*$ is large enough,
 $$
 \frac{\ell S_5 q^{-1}_s(S_1S_2S_3)}{S_1S_2S_3}\le \frac{\ell S_5}{Dq_s(0)}=
 \frac{1-s^\ell}{1-s}\le 1+2\delta,
 $$
 where we used $\log s\leq {-C_u^*\left({1}/{\ell\lambda_u}\right)+C^*\left({1}/{\ell\lambda_s}\right)+C^*\left({1}/{\ell-1}\right)}$, see Lemma \ref{eighvector}. A calculation shows for $C_u^*$ is large enough,
 $$
 \frac{\ell S_5 q^{-1}_s(S_1S_2S_3)}{S_1S_2S_3}\ge 
 \frac{\ell S_5}{Dq_s(q^{-1}_s(S_1S_2S_3))}\ge 
 \left(1-\frac{S_1S_2S_3}{S_5^{\frac{\ell}{\ell-1}}}\right)^{\frac{\ell-1}{\ell}}\ge 
 1-\frac{S_1S_2S_3}{S_5^{\frac{\ell}{\ell-1}}}\ge 1-2\delta,
 $$
 where we used ${\ell S_1^{\ell}}/{S_2}\leq{3}/{2}$ and $\log\left({S_1S_2S_3}/{S_5^{\frac{\ell}{\ell-1}}}\right)\leq -C_u^*\left({\lambda_u^2(\ell+1)-1}/{\ell\lambda_u(1+\lambda_u)}\right)+O(C^*)$, see Lemma \ref{eighvector}.
 \end{proof}
 
 \begin{lem}\label{C9} For every $\delta>0$ and $C^*\geq 1$ we have
 $
\left|Z_{\left[0,q_s^{-1}\left(y_1e^{y_2}e^{y_3}\right)\right]}q_s\right|\le \delta.
 $
  when $C_u^*$ is large enough.
 \end{lem}
 
 \begin{proof} A calculation shows for $C_u^*$ is large enough,
 \begin{equation}\label{ZSS}
 \left|Z_{\left[0,q_s^{-1}\left(y_1e^{y_2}e^{y_3}\right)\right]}q_s\right|= 
 \left|Z_{\left[0,q_s^{-1}\left(S_1 S_2 S_3\right)\right]}q_s\right|\le 
  2S_1 S_2 S_3 \frac{\ell-1}{\ell} S_5^{-\frac{\ell}{\ell-1}}\le \delta,
 \end{equation}
 where we used ${\ell S_1^{\ell}}/{S_2}\leq {3}/{2}$ and $\log\left({S_1S_2S_3}/{S_5^{\frac{\ell}{\ell-1}}}\right)\leq -C_u^*\left({\lambda_u^2(\ell+1)-1}/{\ell\lambda_u(1+\lambda_u)}\right)+O(C^*)$, see Lemma \ref{eighvector}.
 \end{proof}
  \begin{lem}\label{squarefactors} For every  $C^*\geq 1$  
  we have 
 $$
  \begin{aligned}
1.& \text{  }  \frac12\le \left[\frac{\ell S_5 \left(\varphi^{-1}\circ q_s^{-1}\right)\left(S_{1}S_{2}S_{3}\right)}{S_{1}S_{2}S_{3}}\cdot\frac{1}{1-\left(\varphi^{l}\left( S_1\right)\right)^{\ell}}\right]\le 2,\\
 2.& \text{  }   \frac12\le \left[\frac{S_{1}S_{3} \left(1-\left(\varphi^{-1}\circ q_s^{-1}\right)\left(1-S_{2}\right)\right)}{S_5 \left(\varphi^{-1}\circ q_s^{-1}\right)\left(S_{1}S_{2}S_{3}\right)}\right]\le 2,\\
3.& \text{  }    \frac12 \le \left[\left(\frac{\varphi^{l}\left( S_1\right)}{S_{1}}\right)^{\ell}\right]\le 2,\\
 4.&  \text{  }  \frac12\le \left[\left(\frac{\varphi^{l}\left( S_1\right)}{S_{1}}\right)^{\ell -1}\right]\le 2,
\end{aligned}
 $$ 
 when  $C_u^*>1$ is large enough and $\epsilon^*<1$ is small enough.
\end{lem}

 \begin{proof} Properties $3$ and $4$ hold when $\epsilon^*<1$ is small enough. Moreover, for $C_u^*$ large enough,
 \begin{equation}\label{top2}
\frac{S_{1}S_{3} \left(1-\left(\varphi^{-1}\circ q_s^{-1}\right)\left(1-S_{2}\right)\right)}{\ell S_{1}S_{2}S_{3}}=1 + O(\epsilon^*),
 \end{equation}
 where we used $S_2\leq e^{C^*}e^{-C_u^*}$.
For $C_u^*$ large enough and  Lemma \ref{s1235} we get
  \begin{equation}\label{bottom2}
\frac{S_{1}S_{2}S_{3}}     {\ell S_5 \left(\varphi^{-1}\circ q_s^{-1}\right)\left(S_{1}S_{2}S_{3}\right)}=1 + O(\epsilon^*).
 \end{equation}
 The equations (\ref{top2}) and (\ref{bottom2}) imply property $2$ by taking $\epsilon^*<1$ small enough. Finally, for $C_u^*$ large enough,
  \begin{equation}\label{top1}
\frac{1}{1-\left(\varphi^{l}\left( S_1\right)\right)^{\ell}}=1 + O(\epsilon^*).
\end{equation}
The esitimates  (\ref{bottom2}) and (\ref{top1}) imply property $1$ by taking $\epsilon^*<1$ small enough.
 \end{proof}
 The following lemma is needed in Appendix II.
\begin{lem}\label{deltaqs} On $D$,
$$
 \begin{aligned}
 \left|\frac{\partial q_s^{-1}(1-S_2)}{\partial s}\right| &=O\left(S_2\right)=O\left(1\right),\\
 \left|\frac{\partial q_s^{-1}(S_1S_2S_3)}{\partial s}\right| &=O\left(S_2^{\frac{\ell+1}{\ell}}S_3 S_5^{-\frac{\ell}{\ell-1}}\right)=O\left(1\right),\\
 \left| S_2S_3 D q_s^{-1}(S_1S_2S_3)\right| &=O\left(S_2^{\frac{1}{\ell^2}}S_3^{\frac{1}{\ell}}\right)=O\left(1\right),\\
 \left| S_1S_2S_3 D q_s^{-1}(S_1S_2S_3)\right| |q_s|_3&=O\left(S_2^{\frac{\ell+1}{\ell}} S_3 S_5^{-\frac{\ell+1}{\ell-1}}\right)=O\left(1\right),\\
  \left| S_5^{\frac{1}{\ell-1}}\frac{\partial q_s^{-1}(S_1S_2S_3)}{\partial s}\right|  |q_s|_3&=O\left(S_2^{\frac{\ell+1}{\ell}} S_3 S_5^{-\frac{\ell+1}{\ell-1}}\right)=O\left(1\right).
 \end{aligned}
$$
 \end{lem}
 
 \begin{proof} A calculation shows
 $$
 \frac{\partial q_s^{-1}(y)}{\partial s}=
 \frac{(1-y)(1-s)s^{\ell-1}\left[ (1-s)q^{-1}_s(y)+s     \right]^{1-\ell}-1 +  (1-s)q^{-1}_s(y)+s}
 {(1-s)^2}.
 $$
 Observe, $q^{-1}_s(1-S_2)=1-O(S_2)$. This implies $\left|{\partial q_s^{-1}(1-S_2)}/{\partial s}\right|=O(S_2)$. The first statement of the lemma follows.
  From Lemma \ref{s1235} we have $q_s^{-1}(S_1S_2S_3)=O({S_1S_2S_3}/{S_5})$. Using the expression for ${\partial q_s^{-1}(y)}/{\partial s}$ above we obtain
 \begin{equation}\label{parts qs}
 \left|\frac{\partial q_s^{-1}(S_1S_2S_3)}{\partial s}\right|=O\left(S_2^{\frac{\ell+1}{\ell}}S_3 S_5^{-\frac{\ell}{\ell-1}}\right),
\end{equation}
where we used that $S_1^{\ell}=O(S_2)$ and $s^{\ell-1}=S_5$. By expressing $\log S_2, \log S_3, \log S_5$ in the base of the eigenvectors of the matrix $M$, see Lemma \ref{eighvector}, we have
$$
S_2^{\frac{\ell+1}{\ell}}S_3 S_5^{-\frac{\ell}{\ell-1}}=e^{O(C_u^*)}=O(1).
$$ 
A calculation shows that 
$$Dq_s^{-1}\left(S_1S_2S_3\right)=O\left(\left[S_1S_2S_3 (1-s^{\ell})+ s^{\ell}\right]^{\frac{1}{\ell}-1}\right).$$
With a similar procedure as before one can show that
$$
S_1S_2S_3=e^{-KC_u^*}s^{\ell}<< s^{\ell},
$$ 
and as consequence,
\begin{equation}\label{Ds}
Dq_s^{-1}\left(S_1S_2S_3\right)=O\left(s^{1-\ell}\right).
\end{equation}  
Notice that 
\begin{equation}\label{qs3}
|q_s|_3=O\left(\frac{1}{s^2}\right).
\end{equation}
The third and fourth equations are achieved by using (\ref{Ds}),  (\ref{qs3}) and the previous procedure of using coordinates corresponding to the base of the eigenvectors of $M$. For the last bound one should use (\ref{parts qs}), (\ref{qs3}) and the eigenvectors procedure.
 \end{proof} 
\section{Appendix II: differentiability properties of renormalization}\label{Appendix} Let $D\subset L\subset \R^5\times X_{3+\epsilon}\times X_{3+\epsilon}\times X_{3+\epsilon}$ as defined in Section \ref{manifold1} or a bounded closed set. We show that
$$
R:D\to \R^5\times X_2\times X_2\times X_2,
$$
is differentiable and the derivative in a point ${\underline{f}}$ extends to a bounded operator 
$$
DR_{\underline{f}}: \R^5\times X_2\times X_2\times X_2\to \R^5\times X_2\times X_2\times X_2,
$$
and
${\underline{f}}\mapsto DR_{\underline{f}}$ is continuous. Let us start introducing some notation. \\
We denote by $\underline f=\left({y}_1,{y}_2,{y}_3,{y}_4,{y}_5,{\underline\varphi},
{\underline\varphi}^{l},{\underline\varphi}^{r}\right)=\left(\underline{y}, {\underline\varphi},
{\underline\varphi}^{l},{\underline\varphi}^{r}\right)\in L_0$, the composition by $f=O(\underline f)=\left({y}_1,{y}_2,{y}_3,{y}_4,{y}_5,{O(\underline\varphi}),
O({\underline\varphi}^{l}),O({\underline\varphi}^{r})\right)=\left(y, \varphi,
\varphi^{l},\varphi^{r}\right)\in \L_0$ and the renormalization by
$R\underline f=\left(\tilde{y}_1,\tilde{y}_2,\tilde{y}_3,\tilde{y}_4,\tilde{y}_5,\tilde{\underline\varphi},
\tilde{\underline\varphi}^{l},\tilde{\underline\varphi}^{r}\right)=\left(\underline{\tilde{y}},\tilde{\underline\varphi},
\tilde{\underline\varphi}^{l},\tilde{\underline\varphi}^{r}\right)$. The partial derivatives are denoted accordingly. 
\begin{equation}\label{matrix}DR_{\underline{f}}=\left(\begin{matrix}
     A & B_s & B_l & B_r\\
     C_s & D_{ss} & D_{sl} &D_{sr}\\
   C_l & D_{ls} & D_{ll} &D_{lr}\\
    C_r & D_{rs} & D_{rl} &D_{rr}\\
      \end{matrix}\right)=\left(\begin{matrix}
     A & B_s & B_l & 0\\
     C_s & 0 & D_{sl} &0\\
   C_l & D_{ls} & 0 &D_{lr}\\
    C_r & D_{rs} & D_{rl} &0\\
      \end{matrix}\right)
\end{equation}
where, for example,
$$ 
\begin{aligned}
A&=\frac{\partial \tilde{y}}{\partial y}: \R^5\to \R^5,\\
B_l&=\frac{\partial \tilde{y}}{\partial \underline\varphi^{l}}: X_2\to \R^5,\\
C_s&=\frac{\partial \tilde{\underline\varphi}}{\partial y}: 
\R^5\to X_2 ,\\
D_{lr}&=\frac{\partial \tilde{\underline\varphi}^{l}}{\partial \underline{\varphi}^r}: X_2\to X_2.
\end{aligned}
$$
 Observe, $B_r=0$, $D_{ss}=D_{sr}=D_{ll}=D_{rr}=0$. 
%and for the renormalization 
%$R\underline f=\left(\tilde{y}_1,\tilde{y}_2,\tilde{y}_3,\tilde{y}_4,\tilde{y}_5,\tilde{\underline\varphi},
%\tilde{\underline\varphi}^{l},\tilde{\underline\varphi}^{r}\right)=\left(\tilde{y},\tilde{\underline\varphi},
%\tilde{\underline\varphi}^{l},\tilde{\underline\varphi}^{r}\right)$. 

\subsection{The partial derivative $A$}\label{A} The partial derivative $A$ can be calculated explicitly by using Lemma \ref{ss}. In particular, $A$ depends continuously on $\underline{f}$. The operator $A$ does not cause that $R:L_0\to L$ is not differentiable.  In this subsection, the expressions for the partial derivatives follow from short explicit calculations which are left to the reader. The bounds on the norms follow from the definition of the domain $D$, Lemma \ref{s1235} and Lemma \ref{deltaqs}.

First observe that 
$$
\frac{\partial \tilde{y}_1}{\partial y_3}=
\frac{\partial \tilde{y}_1}{\partial y_4}=
\frac{\partial \tilde{y}_2}{\partial y_4}=
\frac{\partial \tilde{y}_3}{\partial y_4}=0,
$$ 
$$
\frac{\partial \tilde{y}_4}{\partial y_j}=0 \text{ when } j=2,3,5
$$
and
$$
\frac{\partial \tilde{y}_5}{\partial y_j}=0 \text{ when } j=2,3,4,5.
$$
\begin{lem}\label{partial y1 partial yj} Let $\underline{f}\in D$. The partial derivatives 
of $$
\Delta y_j\mapsto \tilde{y}_1(\underline{f}+\Delta  y_j),
$$
 $i=1,2,5$,  are given by
$$
\begin{aligned}
\frac{\partial \tilde{y}_1}{\partial y_1}&=
-\frac{\ell}{S_1}\cdot\left[ 
\frac{D\varphi^{l}(S_1)\varphi^l(S_1)^{\ell-1}S_1}{1-\varphi^{-1}\circ q_s^{-1}(1-S_2)}\right],\\
\frac{\partial \tilde{y}_1}{\partial y_2}&= \left[
\frac{D(\varphi^{-1}\circ q_s^{-1})(1-S_2)\varphi^l(S_1)^{\ell} S_2}
{\left(1-\varphi^{-1}\circ q_s^{-1}(1-S_2)\right)^2}
\right],\\
\frac{\partial \tilde{y}_1}{\partial y_5}&=-1\cdot\left[\frac{1}{\ell-1}\frac{D\varphi^{-1}(q_s^{-1}(1-S_2)) \varphi^l(S_1)^{\ell} S_5^{\frac{1}{\ell-1}}}
{\left(1-\varphi^{-1}\circ q_s^{-1}(1-S_2)\right)^2}\cdot\frac{\partial q^{-1}_s(1-S_2)}{\partial s}
\right],\\
\end{aligned}
$$
with norms
$$
\begin{aligned}
\left|\frac{\partial \tilde{y}_1}{\partial y_1}\right|&=O\left(\frac{1}{S_1}\right),\\
\left|\frac{\partial \tilde{y}_1}{\partial y_2}\right|&=O\left(1\right),\\
\left|\frac{\partial \tilde{y}_1}{\partial y_5}\right|&=O\left(S_5^{\frac{1}{\ell-1}}\right),\\
\end{aligned}
$$
and there exists $E>0$ such that 
$$
\left|\frac{\partial \tilde{y}_1}{\partial y_1}\right|>E\frac{1}{S_1}.$$
The dependance $\underline{f}\mapsto {\partial \tilde{y}_1}/{\partial y_j}$ is continuous. 
\end{lem}

\begin{lem}\label{partial y2 partial yj} Let $\underline{f}\in D$. The partial derivatives 
of $$
\Delta y_j\mapsto \tilde{y}_2(\underline{f}+\Delta  y_j),
$$
 $i=1,2,3,5$,  are given by
$$
\begin{aligned}
\frac{\partial \tilde{y}_2}{\partial y_1}&=\frac{1}{S_1}\cdot \left[
\frac{D(\varphi^{-1}\circ q^{-1}_s)(S_1S_2S_3)S_1S_2S_3}{\varphi^{-1}\circ q^{-1}_s(S_1S_2S_3)}+\frac{\ell \varphi^l(S_1)^{\ell-1} D\varphi^l(S_1) S_1}{1-\varphi^l(S_1)^\ell}\right],\\
\frac{\partial \tilde{y}_2}{\partial y_2}&=
1\cdot \left[
\frac{D(\varphi^{-1}\circ q^{-1}_s)(S_1S_2S_3)S_1S_2S_3}{\varphi^{-1}\circ q^{-1}_s(S_1S_2S_3)}\right],\\
\frac{\partial \tilde{y}_2}{\partial y_3}&=
1\cdot \left[
\frac{D(\varphi^{-1}\circ q^{-1}_s)(S_1S_2S_3)S_1S_2S_3}{\varphi^{-1}\circ q^{-1}_s(S_1S_2S_3)}\right],\\
\frac{\partial \tilde{y}_2}{\partial y_5}&=-1\cdot \left[
\frac{1}{1-\ell}\frac{D\varphi^{-1}(q_s^{-1}(S_1S_2S_3)S_5^{\frac{1}{\ell-1}}}
{\varphi^{-1}\circ q_s^{-1}(S_1S_2S_3)}  \cdot \frac{\partial q_s^{-1}(S_1S_2S_3)}{\partial s}\right] , \\
\end{aligned}
$$
with norms
$$
\begin{aligned}
\left|\frac{\partial \tilde{y}_2}{\partial y_1}\right|&=O\left(\frac{1}{S_1}\right),\\
\left|\frac{\partial \tilde{y}_2}{\partial y_2}\right|&=O\left(1\right),\\
\left|\frac{\partial \tilde{y}_2}{\partial y_3}\right|&=O\left(1\right),\\
\left|\frac{\partial \tilde{y}_2}{\partial y_5}\right|&=O\left(1\right).\\
\end{aligned}
$$

The dependance $\underline{f}\mapsto {\partial \tilde{y}_1}/{\partial y_j}$ is continuous. 
\end{lem}

\begin{lem}\label{partial y3 partial yj} Let $\underline{f}\in D$. The partial derivatives 
of $$
\Delta y_j\mapsto \tilde{y}_3(\underline{f}+\Delta  y_j),
$$
 $i=1,2,3,5$,  are given by
$$
\begin{aligned}
\frac{\partial \tilde{y}_3}{\partial y_1}&=
-\frac{1}{S_1}\cdot \left[
\frac{D(\varphi^{-1}\circ q^{-1}_s)(S_1S_2S_3)S_1S_2S_3}{ \varphi^{-1}\circ q^{-1}_s(S_1S_2S_3)}\right],\\
\frac{\partial \tilde{y}_3}{\partial y_2}&= \left[
\frac{D(\varphi^{-1}\circ q_s^{-1})(1-S_2) S_2}
{1-\varphi^{-1}\circ q_s^{-1}(1-S_2)}
-\frac{D(\varphi^{-1}\circ q^{-1}_s)(S_1S_2S_3)S_1S_2S_3}{\varphi^{-1}\circ q^{-1}_s(S_1S_2S_3)}\right],\\
\frac{\partial \tilde{y}_3}{\partial y_3}&=
-1\cdot \left[
\frac{D(\varphi^{-1}\circ q^{-1}_s)(S_1S_2S_3)S_1S_2S_3}{ \varphi^{-1}\circ q^{-1}_s(S_1S_2S_3)}
\right],\\
\frac{\partial \tilde{y}_3}{\partial  y_5}&=
1\cdot \left[
\frac{1}{1-\ell}\left\{
\frac{D\varphi^{-1}(q_s^{-1}(S_1S_2S_3)S_5^{\frac{1}{\ell-1}}}
{\varphi^{-1}\circ q_s^{-1}(S_1S_2S_3)}  \cdot \frac{\partial q_s^{-1}(S_1S_2S_3)}{\partial s}\right\}\right]+\\
&+ 1\cdot\left[
\frac{1}{1-\ell}\left\{\frac{D\varphi^{-1}(q_s^{-1}(1-S_2))  S_5^{\frac{1}{\ell-1}}}
{1-\varphi^{-1}\circ q_s^{-1}(1-S_2)}\cdot\frac{\partial q^{-1}_s(1-S_2)}
{\partial s}
\right\}\right],
\end{aligned}
$$
with norms
$$
\begin{aligned}
\left|\frac{\partial \tilde{y}_3}{\partial y_1}\right|&=O\left(\frac{1}{S_1}\right),\\
\left|\frac{\partial \tilde{y}_3}{\partial y_2}\right|&=O\left(1\right),\\
\left|\frac{\partial \tilde{y}_3}{\partial y_3}\right|&=O\left(1\right),\\
\left|\frac{\partial \tilde{y}_3}{\partial y_5}\right|&=O\left(1\right).\\
\end{aligned}
$$
The dependance $\underline{f}\mapsto {\partial \tilde{y}_1}/{\partial y_j}$ is continuous. 
\end{lem}

\begin{lem}\label{partial y4 partial yj} Let $\underline{f}\in D$. The partial derivatives 
of $$
\Delta y_j\mapsto \tilde{y}_4(\underline{f}+\Delta  y_j),
$$
 $j=1,4$,  are given by
$$
\begin{aligned}
\frac{\partial \tilde{y}_4}{\partial y_1}&=
\frac{\ell}{S_1}\cdot \left[\frac{D\varphi^l(S_1)}{\varphi^l(S_1)}\cdot S_1\right],\\
\frac{\partial \tilde{y}_4}{\partial y_4}&=-1,\\
\end{aligned}
$$
with norms
$$
\begin{aligned}
\left|\frac{\partial \tilde{y}_4}{\partial y_1}\right|&=O\left(\frac{1}{S_1}\right),\\
\left|\frac{\partial \tilde{y}_4}{\partial y_4}\right|&=O\left(1\right).\\
\end{aligned}
$$
The dependance $\underline{f}\mapsto {\partial \tilde{y}_4}/{\partial y_j}$, $j=1,4$, is continuous. 
\end{lem}

\begin{lem}\label{partial y5 partial yj} Let $\underline{f}\in D$. The partial derivative 
of $$
\Delta y_1\mapsto \tilde{y}_5(\underline{f}+\Delta  y_1),
$$
is given by
$$
\frac{\partial \tilde{y}_5}{\partial y_1}=
\frac{\ell-1}{S_1}\cdot \left[\frac{D\varphi^l(S_1)}{\varphi^l(S_1)}\cdot S_1\right],
$$
with norm
$$
\begin{aligned}
\left|\frac{\partial \tilde{y}_5}{\partial y_1}\right|&=O\left(\frac{1}{S_1}\right).\\
\end{aligned}
$$
The dependance $\underline{f}\mapsto {\partial \tilde{y}_5}/{\partial y_1}$ is continuous. 
\end{lem}

Let  $A=A(\underline{f})=\left({\partial \tilde{y}_i}/{\partial y_j}\right): \R^5\to \R^5$.

\begin{prop}\label{Aderivative} The dependance $\underline{f}\mapsto A(\underline{f})$ is continuous and
$$
\left| \tilde{y}(\underline{f}+\Delta y)-
\left[ \tilde{y}(\underline{f})+A \Delta y\right]\right|=O_{\underline{f}}(|\Delta y |^2).
$$
\end{prop}

\subsection{The partial derivative $B_l$}\label{Bl}

Recall the identification of $\text{Diff }^r([0,1])$ with $\mathcal{C}^{r-2}([0,1])$, $r\ge 2$. Given a non-linearity $\eta\in \mathcal{C}^{r-2}([0,1])$ denote the corresponding diffeomorphism with 
$$\varphi_\eta(x)=\frac{\int_0^x e^{\int_0^s\eta} ds}{\int_0^1 e^{\int_0^s\eta} ds}.$$
The following lemma is obtained by a straight forward calculation. The proof is left to the reader.

\begin{lem}\label{evaluation} Let $x\in [0,1]$. The evaluation operator 
$E: \text{Diff }^2([0,1])=\mathcal{C}^{0}([0,1])\to \R$ 
$$
E: \varphi\mapsto \varphi(x)
$$
is differentiable with derivative ${\partial \varphi(x)}/{\partial \varphi}: \mathcal{C}^{0}([0,1])\to \R$ given by
$$
\frac{\partial \varphi(x)}{\partial \varphi}(\Delta \eta)=
\left( \frac{\int_0^x[\int_0^s \Delta \eta] e^{\int_0^s \eta } ds}{\int_0^x e^{\int_0^s \eta } ds}- \frac{\int_0^1[\int_0^s \Delta \eta] e^{\int_0^s \eta } ds}{\int_0^1 e^{\int_0^s \eta } ds}\right)\varphi(x),
$$
where $\varphi=\varphi_\eta$. The norm is bounded by
$$
\left| \frac{\partial \varphi(x)}{\partial \varphi}\right|\le 2\min\left\{\varphi(x), 1-\varphi(x)\right\}.
$$
Moreover\footnote{ We identify $\Delta \varphi$ with $\Delta \eta$},
$$
\left| E(\varphi+\Delta \varphi)-\left[E(\varphi)+\frac{\partial \varphi(x)}{\partial \varphi}(\Delta \varphi)\right]\right|=O(|\Delta \varphi|_2^2).
$$
\end{lem}

\begin{cor}\label{corphil} Let $\psi^+,\psi^-\in \text{Diff }^2([0,1])$ and $x\in [0,1]$. The evaluation operator 
$E^{\psi^+,\psi^-}: \text{Diff }^2([0,1])=\mathcal{C}^{0}([0,1])\to \R,$ 
$$
E^{\psi^+,\psi^-}: \varphi\mapsto \psi^+\circ \varphi\circ \psi^-(x)
$$
is differentiable with derivative $
{\partial (\psi^+\circ \varphi\circ \psi^-(x))}/{\partial \varphi}: \mathcal{C}^{0}([0,1])\to \R$ given by
$$
\frac{\partial (\psi^+\circ \varphi\circ \psi^-(x))}{\partial \varphi}=D\psi^+(\varphi\circ \psi^-(x))\frac{\partial \varphi(\psi^-(x))}{\partial \varphi},
$$
with norm
$$
\left| \frac{\partial (\psi^+\circ \varphi\circ \psi^-(x))}{\partial \varphi}\right|\le 2 
\max \left\{D\psi^+\right\} \cdot \varphi(\psi^-(x)).
$$
Moreover,
$$
\left| E^{\psi^+,\psi^-}(\varphi+\Delta \varphi)-\left[E^{\psi^+,\psi^-}(\varphi)+\frac{\partial (\psi^+\circ \varphi\circ \psi^-(x))}{\partial \varphi}(\Delta \varphi)\right]\right|=O(|\Delta \varphi|_2^2).
$$
\end{cor}

 Let $\tau\in T$ 
and recall that $\pi^{\tau}: X_3\to X_3$ is defined as
$$\pi^{\tau}\underline{\varphi}_{\tau'}=\left\{
    \begin{matrix}
    
    0  & \tau'>\tau\\
      
    \varphi_{\tau'}  & \tau'\leq\tau 
     
      \end{matrix}\right. $$
Also define  $\pi_{\tau}: X_3\to X_3$ is defined as
$$\pi_{\tau}\underline{\varphi}_{\tau'}=\left\{
    \begin{matrix}
    \varphi_{\tau'}  & \tau'>\tau \\ 
    0  & \tau'\leq \tau\\
 \end{matrix}\right. 
 $$ 
 Let
 $$
 \varphi^l_{+,\tau}=O\circ \pi_\tau(\underline{\varphi}^l) \text{ and }  
 \varphi^l_{-,\tau}=O\circ \pi^\tau(\underline{\varphi}^l) .
 $$    
 Observe, $O\underline{\varphi}^l=\varphi^l_{+,\tau}\circ \varphi^l_\tau\circ \varphi^l_{-,\tau}$.
The proof of the following lemma  is a straight forward calculation of derivatives using Lemma \ref{ss}. The estimates on the norms follow from Corollary \ref{corphil}. 

\begin{lem}\label{partial ypartial phil} Let $\underline{f}\in D$ and $\tau\in T$. The maps
$$
\text{Diff }^{3+\epsilon}([0,1])\ni \Delta \varphi^l_\tau \mapsto \tilde{y}_i(\underline{f}+\Delta \varphi^l_\tau)\in \R
$$
are differentiable, $i=1,2,3,4,5$. The derivatives are given by
$$
\begin{aligned}
\frac{\partial \tilde{y}_1}{\partial \varphi^l_\tau}&=
D\varphi^l_{+,\tau}(\varphi^l_\tau\circ \varphi^l_{-,\tau}(S_1))\cdot 
\frac{\ell \varphi^l(S_1)^{\ell-1}}{1-\varphi^{-1}\circ q_s^{-1}(1-S_2)}\cdot 
\frac{\partial \varphi^l_\tau(\varphi^l_{-,\tau}(S_1))}{\partial \varphi^l_\tau},\\
\frac{\partial \tilde{y}_2}{\partial \varphi^l_\tau}&=D\varphi^l_{+,\tau}(\varphi^l_\tau\circ \varphi^l_{-,\tau}(S_1))\cdot 
\frac{\ell \varphi^l(S_1)^{\ell-1}}{1-\left(\varphi^l(S_1)\right)^\ell}\cdot 
\frac{\partial \varphi^l_\tau(\varphi^l_{-,\tau}(S_1))}{\partial \varphi^l_\tau},\\
\frac{\partial \tilde{y}_3}{\partial \varphi^l_\tau}&=0,\\
\frac{\partial \tilde{y}_4}{\partial \varphi^l_\tau}&=
D\varphi^l_{+,\tau}(\varphi^l_\tau\circ \varphi^l_{-,\tau}(S_1))\cdot 
\frac{\ell }{\varphi^l(S_1)}\cdot 
\frac{\partial \varphi^l_\tau(\varphi^l_{-,\tau}(S_1))}{\partial \varphi^l_\tau},\\
\frac{\partial \tilde{y}_5}{\partial \varphi^l_\tau}&=D\varphi^l_{+,\tau}(\varphi^l_\tau\circ \varphi^l_{-,\tau}(S_1))\cdot 
\frac{\ell -1}{\varphi^l(S_1)}\cdot 
\frac{\partial \varphi^l_\tau(\varphi^l_{-,\tau}(S_1))}{\partial \varphi^l_\tau}.\\
\end{aligned}
$$
In particular, the partial derivatives extend to bounded functionals on $\text{Diff }^2([0,1])$ with norms
$$
\begin{aligned}
\left|\frac{\partial \tilde{y}_1}{\partial \varphi^l_\tau}\right|&=
O(1),\\
\left|\frac{\partial \tilde{y}_2}{\partial \varphi^l_\tau}\right|&=O(S_1^{\ell}),\\
\left|\frac{\partial \tilde{y}_3}{\partial \varphi^l_\tau}\right|&=
0,\\
\left|\frac{\partial \tilde{y}_4}{\partial \varphi^l_\tau}\right|&=
O(1),\\
\left|\frac{\partial \tilde{y}_5}{\partial \varphi^l_\tau}\right|&=
O(1).\\
\end{aligned}
$$
The dependance $\underline{f}\mapsto {\partial \tilde{y}_i}/{\partial \varphi^l_\tau}$ is continuous and 
$$
\left| \tilde{y}_i(\underline{f}+\Delta \varphi^l_\tau)-
\left[ \tilde{y}_i(\underline{f})+\frac{\partial \tilde{y}_i}{\partial \varphi^l_\tau}\Delta \varphi^l_\tau\right]\right|=O_{\underline{f}}(|\Delta \varphi^l_\tau |_{3+\epsilon}^2).
$$
\end{lem}
The functionals ${\partial \tilde{y}_i}/{\partial \varphi^l_\tau}$ define the operator ${\partial \tilde{y}}/{\partial \varphi^l_\tau}: \text{Diff }^2([0,1])\to \R^5$.
Let the operator $B_l:X_2 \to \R^5$ be defined by
$$
B_l: \Delta \underline{\varphi}^l\mapsto \sum_{\tau\in T} \frac{\partial \tilde{y}}{\partial \varphi^l_\tau} \Delta \varphi^l_\tau.
$$
This is a uniformly bounded operator. This follows from Lemma \ref{partial ypartial phil}. We obtain the following.

\begin{prop}\label{Blderivative} The dependance $\underline{f}\mapsto B_l(\underline{f})$ is continuous and
$$
\left| \tilde{y}(\underline{f}+\Delta \underline{\varphi}^l)-
\left[ \tilde{y}(\underline{f})+B_l \Delta \underline{\varphi}^l\right]\right|=O_{\underline{f}}(|\Delta \underline{\varphi}^l |_{3+\epsilon}^2).
$$
Moreover, the partial derivative extends to a uniformly bounded operator $B_l:X_2 \to \R^5$. 
\end{prop}

\subsection{The partial derivative $B_s$}\label{Bs}

The following lemma is obtained by a straight forward calculation. The proof is left to the reader.

\begin{lem}\label{evaluationinv} Let $x\in [0,1]$. The evaluation operator 
$E_-: \text{Diff }^2([0,1])=\mathcal{C}^{0}([0,1])\to \R$,
$$
E_-: \varphi\mapsto \varphi^{-1}(x)
$$
is differentiable with derivative ${\partial \varphi^{-1}(x)}/{\partial \varphi}: \mathcal{C}^{0}([0,1])\to \R$ given by
$$
\frac{\partial \varphi^{-1}(x)}{\partial \varphi}=
-\frac{1}{D\varphi(\varphi^{-1}(x))}\cdot \frac{\partial \varphi(\varphi^{-1}(x))}{\partial \varphi}.
$$
 The norm is bounded by
$$
\left| \frac{\partial \varphi^{-1}(x)}{\partial \varphi}\right|\le 2 
\max \left\{\frac{1}{D\varphi}\right\} \cdot \min\left\{x, 
1-x\right\}.
$$
Moreover,
$$
\left| E_-(\varphi+\Delta \varphi)-\left[E_-(\varphi)+\frac{\partial \varphi^{-1}(x)}{\partial \varphi}(\Delta \varphi)\right]\right|=O(|\Delta \varphi|_2^2).
$$
\end{lem}

\begin{cor}\label{corphilinv} Let $\psi^+,\psi^-\in \text{Diff }^2([0,1])$ and $x\in [0,1]$. The evaluation operator 
$E^{\psi^+,\psi^-}_-: \text{Diff }^2([0,1])=\mathcal{C}^{0}([0,1])\to \R,$ 
$$
E^{\psi^+,\psi^-}_-: \varphi\mapsto \psi^+\circ \varphi^{-1}\circ \psi^-(x)
$$
is differentiable with derivative $
{\partial (\psi^+\circ \varphi^{-1}\circ \psi^-(x))}/{\partial \varphi}: \mathcal{C}^{0}([0,1])\to \R$ given by
$$
\frac{\partial (\psi^+\circ \varphi^{-1}\circ \psi^-(x))}{\partial \varphi}=D\psi^+(\varphi^{-1}\circ \psi^-(x))\frac{\partial \varphi^{-1}(\psi^-(x))}{\partial \varphi},
$$
with norm
$$
\left| \frac{\partial (\psi^+\circ \varphi^{-1}\circ \psi^-(x))}{\partial \varphi}\right|\le 2 
\frac{\max D\psi^+}{\min D\varphi}\cdot \min\left\{\psi^-(x), 
1-\psi^-(x)\right\}.
$$
Moreover,
$$
\left| E^{\psi^+,\psi^-}_-(\varphi+\Delta \varphi)-\left[E^{\psi^+,\psi^-}_-(\varphi)+\frac{\partial (\psi^+\circ \varphi^{-1}\circ \psi^-(x))}{\partial \varphi}(\Delta \varphi)\right]\right|=O(|\Delta \varphi|_2^2).
$$
\end{cor}

  Let
 $$
 \varphi_{+,\tau}=O\circ \pi_\tau(\underline{\varphi}) \text{ and }  
 \varphi_{-,\tau}-=O\circ \pi^\tau(\underline{\varphi}) .
 $$    
 Observe, $O\underline{\varphi}=\varphi_{+,\tau}\circ \varphi_\tau\circ \varphi_{-,\tau}$.
The proof of the following lemma is a straight forward calculation of derivatives using Lemma \ref{ss}. The estimates on the norms follow from Corollary \ref{corphilinv}. 

\begin{lem}\label{partial ypartial philinv} Let $\underline{f}\in D$ and $\tau\in T$. The maps
$$
\text{Diff }^3([0,1])\ni \Delta \varphi_\tau \mapsto \tilde{y}_i(\underline{f}+\Delta \varphi_\tau)\in \R
$$
are differentiable, $i=1,2,3,4,5$. The derivatives are given by
$$
\begin{aligned}
\frac{\partial \tilde{y}_1}{\partial \varphi_\tau}&=
D\varphi_{-,\tau}^{-1}(
\varphi^{-1}_\tau(\varphi_{+,\tau}^{-1} (q_s^{-1}(1-S_2)))         )\cdot 
\frac{\varphi^l(S_1)^{\ell}}{\left[1-\varphi^{-1}(q_s^{-1}(1-S_2))\right]^2}\cdot 
\frac{\partial \varphi_\tau(\varphi_{+,\tau}^{-1} (q_s^{-1}(1-S_2)))}{\partial \varphi_\tau},\\
\frac{\partial \tilde{y}_2}{\partial \varphi_\tau}&=
D\varphi_{-,\tau}^{-1}(
\varphi^{-1}_\tau(\varphi_{+,\tau}^{-1} (q_s^{-1}(S_1S_2S_3))))\cdot 
\frac{1}{\varphi^{-1}(q_s^{-1}(S_1S_2S_3))}\cdot
\frac{\partial \varphi_\tau(\varphi_{+,\tau}^{-1} (q_s^{-1}(S_1S_2S_3)))}{\partial \varphi_\tau},\\
\frac{\partial \tilde{y}_3}{\partial \varphi_\tau}&=D\varphi_{-,\tau}^{-1}(
\varphi^{-1}_\tau(\varphi_{+,\tau}^{-1} (q_s^{-1}(1-S_2)))         )\cdot 
\frac{\varphi^l(S_1)^{\ell}}{1-\varphi^{-1}(q_s^{-1}(1-S_2))}\cdot 
\frac{\partial \varphi_\tau(\varphi_{+,\tau}^{-1} (q_s^{-1}(1-S_2)))}{\partial \varphi_\tau}+\\
&\text{    }-D\varphi_{-,\tau}^{-1}(
\varphi^{-1}_\tau(\varphi_{+,\tau}^{-1} (q_s^{-1}(S_1S_2S_3))))\cdot 
\frac{1}{\varphi^{-1}(q_s^{-1}(S_1S_2S_3))}\cdot
\frac{\partial \varphi_\tau(\varphi_{+,\tau}^{-1} (q_s^{-1}(S_1S_2S_3)))}{\partial \varphi_\tau},\\
\frac{\partial \tilde{y}_4}{\partial \varphi_\tau}&=0,\\
\frac{\partial \tilde{y}_5}{\partial \varphi_\tau}&=0.\\
\end{aligned}
$$
In particular, the partial derivatives extend to bounded functionals on $\text{Diff }^2([0,1])$ with norms
$$
\begin{aligned}
\left|\frac{\partial \tilde{y}_1}{\partial \varphi_\tau}\right|&=
O(1),\\
\left|\frac{\partial \tilde{y}_2}{\partial \varphi_\tau}\right|&=O(1),\\
\left|\frac{\partial \tilde{y}_3}{\partial \varphi_\tau}\right|&=O(1),\\
\left|\frac{\partial \tilde{y}_4}{\partial \varphi_\tau}\right|&=0,\\
\left|\frac{\partial \tilde{y}_5}{\partial \varphi_\tau}\right|&=0.\\
\end{aligned}
$$
The dependance $\underline{f}\mapsto {\partial \tilde{y}_i}/{\partial \varphi_\tau}$ is continuous and 
$$
\left| \tilde{y}_i(\underline{f}+\Delta \varphi_\tau)-
\left[ \tilde{y}_i(\underline{f})+\frac{\partial \tilde{y}_i}{\partial \varphi_\tau}\Delta \varphi_\tau\right]\right|=O_{\underline{f}}(|\Delta \varphi_\tau |_{3+\epsilon}^2).
$$
\end{lem}

The functionals ${\partial \tilde{y}_i}/{\partial \varphi_\tau}$ define the operator ${\partial \tilde{y}}/{\partial \varphi_\tau}: \text{Diff }^2([0,1])\to \R^5$. Define the operator $B_s:X_2 \to \R^5$ by
$$
B_s: \Delta \underline{\varphi}\mapsto \sum_{\tau\in T} \frac{\partial \tilde{y}}{\partial \varphi_\tau} \Delta \varphi_\tau.
$$
This is a uniformly bounded operator, see Lemma \ref{partial ypartial philinv}. We obtain the following.

\begin{prop}\label{Bsderivative} The dependance $\underline{f}\mapsto B_s(\underline{f})$ is continuous and
$$
\left| \tilde{y}(\underline{f}+\Delta \underline{\varphi})-
\left[ \tilde{y}(\underline{f})+B_s \Delta \underline{\varphi}\right]\right|=O_{\underline{f}}(|\Delta \underline{\varphi}|_{3+\epsilon}^2)
$$
Moreover, the partial derivative extends to a uniformly bounded operator $B_s:X_2 \to \R^5$. 
\end{prop}

\subsection{The partial derivatives $C_s$}\label{Cs}

The proof of the following lemma is a straightforward calculation and it is left to the reader. Recall that we identify $\text{Diff }^r([0,1])$ with $C^{r-2}([0,1])$, diffeomorphisms with their non-linearities.

\begin{lem}\label{zoompartials} Let $\varphi\in \text{Diff }^{3+\epsilon}([0,1]).$ The zoom curve $Z:[0,1]^2 \ni (a,b)\mapsto Z_{[a,b]}\varphi\in \text{Diff }^2([0,1])$ is differentiable with partial derivatives
$$
\frac{\partial  Z_{[a,b]}\varphi}{\partial a} =
(b-a)(1-x)D\eta((b-a)x+a)-\eta((b-a)x+a),
$$
and
$$
\frac{\partial  Z_{[a,b]}\varphi}{\partial b} =
(b-a)xD\eta((b-a)x+a)+\eta((b-a)x+a).
$$
The norms are bounded by
$$
\left|  \frac{\partial  Z_{[a,b]}\varphi}{\partial a}  \right|_{2},
\left|  \frac{\partial  Z_{[a,b]}\varphi}{\partial b}  \right|_{2}\le 2 |\varphi |_{3},
$$
and
$$
\left|  Z_{[a+\Delta a, b+\Delta b]}\varphi- 
\left[  Z_{[a,b]}\varphi+\frac{\partial  Z_{[a,b]}\varphi}{\partial a}\Delta a +
\frac{\partial  Z_{[a,b]}\varphi}{\partial b} \Delta b\right]
\right|
$$
$$
\le |\varphi |_{3+\epsilon}
\left( | \Delta a|^{1+\epsilon}+| \Delta b|^{1+\epsilon}|\right).
$$
\end{lem}

Observe, ${\partial \tilde{\varphi}_\tau}/{\partial y_i}=0$ when $i=2,3,4,5$.

\begin{lem}\label{partial phi partial y} Let $\underline{f}\in D$ and $\tau\in T$. The map
$$
\R \ni \Delta y_1 \mapsto \tilde{\varphi}_\tau(\underline{f}+\Delta y_1)\in \text{Diff}^2([0,1])
$$
is differentiable. The derivative is given by
$$
\frac{\partial \tilde{\varphi}_\tau}{\partial y_1}=D(O\circ \pi^\tau(\underline{\varphi}^l))(y_1)\cdot \frac{\partial Z_{[y_1(\tau,1]}\varphi^l_\tau}{\partial y_1(\tau)},
$$
with norm
$$
\left|  \frac{\partial \tilde{\varphi}_\tau}{\partial y_1}\right|_2\le 
2e^{\frac{\delta}{8}} | \varphi^l_\tau|_{3}=O(1).
$$
The dependance $\underline{f}\mapsto {\partial \tilde{\varphi}_\tau}/{\partial y_1}$ is continuous and 
$$
\left| \tilde{\varphi}_\tau(\underline{f}+\Delta y_1)-
\left[\tilde{\varphi}_\tau(\underline{f})+\frac{\partial \tilde{\varphi}_\tau}{\partial y_1}\Delta y_1\right]\right|_2\le O_{\underline{f}}\left(  | \varphi^l_\tau|_{3+\epsilon} |\Delta y_1|^{1+\epsilon}\right).
$$
\end{lem}

Define the operator $C_s:\R^5 \to X_2$ by
$$
C_s(\Delta \underline{y})_\tau= \frac{\partial \tilde{\varphi}_\tau}{\partial y_1}\Delta y_1.
$$
This is a uniformly bounded operator, see Lemma \ref{partial phi partial y}. We obtain the following.

\begin{prop}\label{Csderivative} The dependance $\underline{f}\mapsto C_s(\underline{f})$ is continuous and
$$
\left| \tilde{\varphi}(\underline{f}+\Delta \underline{y})-
\left[\tilde{\varphi}(\underline{f})+C_s \Delta \underline{y}\right]\right|_2
=O_{\underline{f}}(|\Delta \underline{y}|^{1+\epsilon}).
$$
Moreover, the partial derivative extends to a uniformly bounded operator $C_s:\R^5 \to X_2$. 
\end{prop}

\subsection{The partial derivatives $C_l$}\label{Cl}

Observe, ${\partial \tilde{\varphi}^l_\tau}/{\partial y_i}=0$ when $i=1,3,4$ and $\tau\in T$. Moreover, ${\partial \tilde{\varphi}^l_\tau}/{\partial y_i}=0$ when $i=1,2,3,4,5$ and $\tau> 1/2$. The following lemmas describe the nonzero partial derivatives. The bounds on the norms of the derivatives follow from Lemma \ref{deltaqs}.
\begin{lem}\label{partial phi l min 12 partial y} Let $\underline{f}\in D$. The maps
$$
\R \ni \Delta y_i \mapsto \tilde{\varphi}^l_\tau(\underline{f}+\Delta y_i)\in \text{Diff }^2([0,1])
$$
are differentiable, $i=2,5$ and $\tau<1/2$. The derivatives are given by
$$
\begin{aligned}
\frac{\partial \tilde{\varphi}^l_\tau}{\partial y_2}&=
-S_2 D(\varphi_{+,\theta \tau}^{-1}\circ q_s^{-1})(1-S_2)\cdot
\frac{\partial Z_{[\varphi^{-1}\circ q^{-1}_s(1-S_2)(\theta \tau),1]}\varphi_{\theta \tau}}
{\partial \varphi^{-1}\circ q^{-1}_s(1-S_2)(\theta \tau)}, \\
\frac{\partial \tilde{\varphi}^l_\tau}{\partial y_5}&=\frac{S_5^\frac{1}{\ell-1}}{\ell-1}
D\varphi_{+,\theta \tau}^{-1}(q_s^{-1}(1-S_2))
\frac{\partial q^{-1}_s(1-S_2)}{\partial s} \cdot
\frac{\partial Z_{[\varphi^{-1}\circ q^{-1}_s(1-S_2)(\theta \tau),1]}\varphi_{\theta \tau}}
{\partial \varphi^{-1}\circ q^{-1}_s(1-S_2)(\theta \tau)},
\\
\end{aligned}
$$
with norms
$$
\begin{aligned}
\left|\frac{\partial \tilde{\varphi}^l_\tau}{\partial y_2}\right|_{2}&\le
2S_2 D(\varphi_{+,\theta \tau}^{-1}\circ q_s^{-1})(1-S_2)|\varphi_{\theta \tau}|_3=O(1),\\
\left|\frac{\partial \tilde{\varphi}^l_\tau}{\partial y_5}\right|_{2}&\le 2\frac{S_5^\frac{1}{\ell-1}}{\ell-1}
D\varphi_{+,\theta \tau}^{-1}(q_s^{-1}(1-S_2))
\frac{\partial q^{-1}_s(1-S_2)}{\partial s} |\varphi_{\theta \tau}|_3=O(1).
\\
\end{aligned}
$$
The dependance $\underline{f}\mapsto {\partial \tilde{\varphi}^l_{\tau}}/{\partial y_i}$ is continuous and 
$$
\left| \tilde{\varphi}^l_\tau(\underline{f}+\Delta \underline{y})-
\left[\tilde{\varphi}^l_\tau(\underline{f})+\frac{\partial \tilde{\varphi}^l_\tau}{\partial y_1}\Delta y_1+\frac{\partial \tilde{\varphi}^l_\tau}{\partial y_5}\Delta y_5\right]\right|_2\le 
$$
$$
\le O_{\underline{f}}\left( |\varphi_{\theta \tau}|_{3+\epsilon}\left\{|\Delta y_2|^{1+\epsilon}+|\Delta y_5|^{1+\epsilon}\right\}\right).
$$
\end{lem}

\begin{lem}\label{partial phi l12 partial y} Let $\underline{f}\in D$. The maps
$$
\R \ni \Delta y_i \mapsto \tilde{\varphi}^l_{\frac12}(\underline{f}+\Delta y_i)\in \text{Diff }^2([0,1])
$$
are differentiable, $i=2,5$. The derivatives are given by
$$
\begin{aligned}
\frac{\partial \tilde{\varphi}^l_{\frac12}}{\partial y_2}&=
-S_2 Dq^{-1}_s(1-S_2)\cdot\frac{\partial Z_{[q^{-1}_s(1-S_2),1]}q_s}{\partial q^{-1}_s(1-S_2)}, \\
\frac{\partial \tilde{\varphi}^l_\frac12}{\partial y_5}&=\frac{S_5^\frac{1}{\ell-1}}{\ell-1}
\frac{\partial q^{-1}_s(1-S_2)}{\partial s} \cdot \frac{\partial Z_{[q^{-1}_s(1-S_2),1]}q_s}{\partial q^{-1}_s(1-S_2)},
\\
\end{aligned}
$$
with norms
$$
\begin{aligned}
\left|\frac{\partial \tilde{\varphi}^l_\frac12}{\partial y_2}\right|_{2}&=
2S_2 Dq^{-1}_s(1-S_2) 2\ell\le 4S_2=O(1),\\
\left|\frac{\partial \tilde{\varphi}^l_\frac12}{\partial y_5}\right|_{2}&=\frac{4\ell}{\ell-1}
\frac{\partial q^{-1}_s(1-S_2)}{\partial s} S_5^\frac{1}{\ell-1}=O(1).
\\
\end{aligned}
$$
The dependance $\underline{f}\mapsto {\partial \tilde{\varphi}^l_\frac12}/{\partial y_i}$ is continuous and 
$$
\left| \tilde{\varphi}^l_\frac12(\underline{f}+\Delta \underline{y})-
\left[\tilde{\varphi}^l_\frac12(\underline{f})+\frac{\partial \tilde{\varphi}^l_\frac12}{\partial y_2}\Delta y_2+\frac{\partial \tilde{\varphi}^l_\frac12}{\partial y_5}\Delta y_5\right]\right|_2\le 
$$
$$
\le O_{\underline{f}}\left( |\Delta y_2|^{1+\epsilon}+|\Delta y_5|^{1+\epsilon}\right).
$$
\end{lem}
Define the operator $C_l:\R^5 \to X_2$ by
$$
C_l(\Delta \underline{y})_\tau= \frac{\partial \tilde{\varphi}^l_\tau}{\partial y_2}\Delta y_2+  \frac{\partial \tilde{\varphi}^l_\tau}{\partial y_5}\Delta y_5.
$$
This is a uniformly bounded operator, see Lemma \ref{partial phi l min 12 partial y}, Lemma \ref{partial phi l12 partial y}. We obtain the following.

\begin{prop}\label{Clderivative} The dependance $\underline{f}\mapsto C_l(\underline{f})$ is continuous and
$$
\left| \tilde{\varphi}^l(\underline{f}+\Delta \underline{y})-
\left[\tilde{\varphi}^l(\underline{f})+C_l\Delta \underline{y}\right]\right|_2
=O_{\underline{f}}(|\Delta \underline{y}|^{1+\epsilon}).
$$
Moreover, the partial derivative extends to a uniformly bounded operator $C_l:\R^5 \to X_2$. 
\end{prop}

\subsection{The partial derivatives $C_r$}\label{Cr}
Observe that, ${\partial \tilde{\varphi}^l_\tau}/{\partial y_4}=0$ when $\tau\in T$. Moreover, ${\partial \tilde{\varphi}^l_\tau}/{\partial y_i}=0$ when $i=2,3, 5$ and $\tau> 1/2$. The following lemmas describe the nonzero partial derivatives. The bounds on the norms of the derivatives follow from Lemma \ref{deltaqs}.

\begin{lem}\label{partial phi r min 12 partial y} Let $\underline{f}\in D$. The maps
$$
\R \ni \Delta y_i \mapsto \tilde{\varphi}^r_\tau(\underline{f}+\Delta y_i)\in \text{Diff }^2([0,1])
$$
are differentiable, $i=1,2,3,5$ and $\tau<1/2$. The derivatives are given by
$$
\begin{aligned}
\frac{\partial \tilde{\varphi}^r_\tau}{\partial y_1}&=
S_2S_3 D(\varphi_{+,\theta \tau}^{-1}\circ q_s^{-1})(S_1S_2S_2)\cdot
\frac{\partial Z_{[0,\varphi^{-1}\circ q^{-1}_s(S_1S_2S_3)(\theta \tau)]}\varphi_{\theta \tau}}
{\partial \varphi^{-1}\circ q^{-1}_s(S_1S_2S_3)(\theta \tau)}, \\
\frac{\partial \tilde{\varphi}^r_\tau}{\partial y_2}&=
S_1S_2S_3 D(\varphi_{+,\theta \tau}^{-1}\circ q_s^{-1})(S_1S_2S_2)\cdot
\frac{\partial Z_{[0,\varphi^{-1}\circ q^{-1}_s(S_1S_2S_3)(\theta \tau)]}\varphi_{\theta \tau}}
{\partial \varphi^{-1}\circ q^{-1}_s(S_1S_2S_3)(\theta \tau)}, \\
\frac{\partial \tilde{\varphi}^r_\tau}{\partial y_3}&=
S_1S_2S_3 D(\varphi_{+,\theta \tau}^{-1}\circ q_s^{-1})(S_1S_2S_2)\cdot
\frac{\partial Z_{[0,\varphi^{-1}\circ q^{-1}_s(S_1S_2S_3)(\theta \tau)]}\varphi_{\theta \tau}}
{\partial \varphi^{-1}\circ q^{-1}_s(S_1S_2S_3)(\theta \tau)}, \\
\frac{\partial \tilde{\varphi}^r_\tau}{\partial y_5}&=\frac{S_5^\frac{1}{\ell-1}}{\ell-1}
D\varphi_{+,\theta \tau}^{-1}(q_s^{-1}(S_1S_2S_3))
\frac{\partial q^{-1}_s(S_1S_2S_3)}{\partial s} \cdot
\frac{\partial Z_{[0,\varphi^{-1}\circ q^{-1}_s(S_1S_2S_3)(\theta \tau)]}\varphi_{\theta \tau}}
{\partial \varphi^{-1}\circ q^{-1}_s(S_1S_2S_3)(\theta \tau)},
\\
\end{aligned}
$$
with norms
$$
\begin{aligned}
\left|\frac{\partial \tilde{\varphi}^r_\tau}{\partial y_1}\right|_{2}&\le
2S_2S_3 D(\varphi_{+,\theta \tau}^{-1}\circ q_s^{-1})(S_1S_2S_3)|\varphi_{\theta \tau}|_3=O\left(|\varphi_{\theta \tau}|_3\right),\\
\left|\frac{\partial \tilde{\varphi}^r_\tau}{\partial y_2}\right|_{2}&\le
2S_1S_2S_3 D(\varphi_{+,\theta \tau}^{-1}\circ q_s^{-1})(S_1S_2S_3)|\varphi_{\theta \tau}|_3=O\left(|\varphi_{\theta \tau}|_3\right),\\
\left|\frac{\partial \tilde{\varphi}^r_\tau}{\partial y_3}\right|_{2}&\le
2S_1S_2S_3 D(\varphi_{+,\theta \tau}^{-1}\circ q_s^{-1})(S_1S_2S_3)|\varphi_{\theta \tau}|_3=O\left(|\varphi_{\theta \tau}|_3\right),\\
\left|\frac{\partial \tilde{\varphi}^r_\tau}{\partial y_5}\right|_{2}&\le 2\frac{S_5^\frac{1}{\ell-1}}{\ell-1}
D\varphi_{+,\theta \tau}^{-1}(q_s^{-1}(S_1S_2S_3))
\frac{\partial q^{-1}_s(S_1S_2S_3)}{\partial s} |\varphi_{\theta \tau}|_3=O\left(|\varphi_{\theta \tau}|_3\right).
\\
\end{aligned}
$$
The dependance $\underline{f}\mapsto {\partial \tilde{\varphi}^r(\tau)}/{\partial y_i}$ is continuous and 
$$
\left| \tilde{\varphi}^r_\tau(\underline{f}+\Delta \underline{y})-
\left[\tilde{\varphi}^r_\tau(\underline{f})+\frac{\partial \tilde{\varphi}^r_\tau}{\partial y_1}\Delta y_1+\frac{\partial \tilde{\varphi}^r_\tau}{\partial y_2}\Delta y_2+\frac{\partial \tilde{\varphi}^r_\tau}{\partial y_3}\Delta y_3+\frac{\partial \tilde{\varphi}^r_\tau}{\partial y_5}\Delta y_5\right]\right|_2\le 
$$
$$
\le O_{\underline{f}}\left( |\varphi_{\theta \tau}|_{3+\epsilon}\left\{|\Delta y_1|^{1+\epsilon}+|\Delta y_2|^{1+\epsilon}+|\Delta y_3|^{1+\epsilon}+|\Delta y_5|^{1+\epsilon}\right\}\right).
$$
\end{lem}
\begin{lem}\label{partial phi r12 partial y} Let $\underline{f}\in D$. The maps
$$
\R \ni \Delta y_i \mapsto \tilde{\varphi}^r_{\frac12}(\underline{f}+\Delta y_i)\in \text{Diff }^2([0,1])
$$
are differentiable, $i=1,2,3,5$. The derivatives are given by
$$
\begin{aligned}
\frac{\partial \tilde{\varphi}^r_{\frac12}}{\partial y_1}&=
S_2S_3 Dq^{-1}_s(S_1S_2S_3)\cdot\frac{\partial Z_{[0,q^{-1}_s(S_1S_2S_3)]}q_s}{\partial q^{-1}_s(S_1S_2S_3)}, \\
\frac{\partial \tilde{\varphi}^r_{\frac12}}{\partial y_2}&=
S_1S_2S_3 Dq^{-1}_s(S_1S_2S_3)\cdot\frac{\partial Z_{[0,q^{-1}_s(S_1S_2S_3)]}q_s}{\partial q^{-1}_s(S_1S_2S_3)}, \\
\frac{\partial \tilde{\varphi}^r_{\frac12}}{\partial y_3}&=
S_1S_2S_3 Dq^{-1}_s(S_1S_2S_3)\cdot\frac{\partial Z_{[0,q^{-1}_s(S_1S_2S_3)]}q_s}{\partial q^{-1}_s(S_1S_2S_3)}, \\
\frac{\partial \tilde{\varphi}^l_\frac12}{\partial y_5}&=\frac{S_5^\frac{1}{\ell-1}}{\ell-1}
\frac{\partial q^{-1}_s(S_1S_2S_3)}{\partial s} \cdot \frac{\partial Z_{[0,q^{-1}_s(S_1S_2S_3)]}q_s}{\partial q^{-1}_s(S_1S_2S_3)},
\\
\end{aligned}
$$
with norms
$$
\begin{aligned}
\left|\frac{\partial \tilde{\varphi}^r_\frac12}{\partial y_1}\right|_{2}&=
2S_2S_3 Dq^{-1}_s(S_1S_2S_3) |q_s|_3,\\
\left|\frac{\partial \tilde{\varphi}^r_\frac12}{\partial y_2}\right|_{2}&=
2S_1S_2S_3 Dq^{-1}_s(S_1S_2S_3) |q_s|_3=O(1),\\
\left|\frac{\partial \tilde{\varphi}^r_\frac12}{\partial y_3}\right|_{2}&=
2S_1S_2S_3 Dq^{-1}_s(S_1S_2S_3) |q_s|_3=O(1),\\
\left|\frac{\partial \tilde{\varphi}^r_\frac12}{\partial y_5}\right|_{2}&=2\frac{S_5^\frac{1}{\ell-1}}{\ell-1}
\frac{\partial q^{-1}_s(S_1S_2S_3)}{\partial s} |q_s|_3=O(1).
\\
\end{aligned}
$$
The dependance $\underline{f}\mapsto {\partial \tilde{\varphi}^r_\frac12}/{\partial y_i}$ is continuous and 
$$
\left| \tilde{\varphi}^r_\frac12(\underline{f}+\Delta \underline{y})-
\left[\tilde{\varphi}^r_\frac12(\underline{f})+\frac{\partial \tilde{\varphi}^r_\frac12}{\partial y_1}\Delta y_1+\frac{\partial \tilde{\varphi}^r_\frac12}{\partial y_2}\Delta y_2+\frac{\partial \tilde{\varphi}^r_\frac12}{\partial y_3}\Delta y_3+\frac{\partial \tilde{\varphi}^r_\frac12}{\partial y_5}\Delta y_5\right]\right|_2\le 
$$
$$
\le O_{\underline{f}}\left( |\Delta y_1|^{1+\epsilon}+|\Delta y_2|^{1+\epsilon}+|\Delta y_3|^{1+\epsilon}+|\Delta y_5|^{1+\epsilon}\right).
$$
\end{lem}
\begin{lem}\label{partial phi r max 12 partial y} Let $\underline{f}\in D$. The maps
$$
\R \ni \Delta y_1\mapsto \tilde{\varphi}^r_\tau(\underline{f}+\Delta y_1)\in \text{Diff }^2([0,1])
$$
are differentiable, $\tau>1/2$. The derivative is given by
$$
\begin{aligned}
\frac{\partial \tilde{\varphi}^r_\tau}{\partial y_1}&= D\left(\varphi^l_{-,\theta \tau}(y_1)\right)\cdot
\frac{\partial Z_{[0, y_1(\theta \tau)]}\varphi^l_{\theta \tau}}
{\partial y_1(\theta \tau)},
\end{aligned}
$$
with norm
$$
\begin{aligned}
\left|\frac{\partial \tilde{\varphi}^r_\tau}{\partial y_1}\right|_{2}&\le 
2 D\left(\varphi^l_{-,\theta \tau}(y_1)\right) |\varphi^l_{\theta \tau}|_3
\end{aligned}.
$$
The dependance $\underline{f}\mapsto {\partial \tilde{\varphi}^r(\tau)}/{\partial y_1}$ is continuous and 
$$
\left| \tilde{\varphi}^r_\tau(\underline{f}+\Delta \underline{y})-
\left[\tilde{\varphi}^r_\tau(\underline{f})+\frac{\partial \tilde{\varphi}^r_\tau}{\partial y_1}{\Delta y_1}\right]\right|_2\le 
$$
$$
\le O_{\underline{f}}\left( |\varphi_{\theta \tau}^l|_{3+\epsilon}\left\{|\Delta y_1|^{1+\epsilon}\right\}\right).
$$
\end{lem}
Define the operator $C_r:\R^5 \to X_2$ by
$$
C_r(\Delta \underline{y})_\tau= \frac{\partial \tilde{\varphi}^r_\tau}{\partial y_1}\Delta y_1+ \frac{\partial \tilde{\varphi}^r_\tau}{\partial y_2}\Delta y_2+ \frac{\partial \tilde{\varphi}^r_\tau}{\partial y_3}\Delta y_3+  \frac{\partial \tilde{\varphi}^r_\tau}{\partial y_5}\Delta y_5.
$$

 This is a bounded operator, see Lemma \ref{partial phi r min 12 partial y}, Lemma \ref{partial phi r12 partial y} and Lemma \ref{partial phi r max 12 partial y} . We obtain the following.

\begin{prop}\label{Crderivative} The dependance $\underline{f}\mapsto C_r(\underline{f})$ is continuous and
$$
\left| \tilde{\varphi}^r(\underline{f}+\Delta \underline{y})-
\left[\tilde{\varphi}^r(\underline{f})+C_r\Delta \underline{y}\right]\right|_2
=O_{\underline{f}}(|\Delta \underline{y}|^{1+\epsilon}).
$$
The operator $C_r:\R^5 \to X_2$ is bounded. Moreover the operator $C_r:\R^4 \to X_2$ with
$$
C_r(\Delta \underline{y})_\tau= \frac{\partial \tilde{\varphi}^r_\tau}{\partial y_2}\Delta y_2+ \frac{\partial \tilde{\varphi}^r_\tau}{\partial y_3}\Delta y_3+  \frac{\partial \tilde{\varphi}^r_\tau}{\partial y_5}\Delta y_5
$$ is uniformly bounded.
\end{prop}

\subsection{The partial derivatives $D_{sl}$}\label{Dsl}
Observe that, ${\partial \tilde{\varphi}_\tau}/{\partial {\varphi}^l_{\tau'}}=0$ when $\tau'>\tau$. The following lemmas describe the nonzero partial derivatives.

\begin{lem}\label{partial phi tau=tau' partial phil} Let $\underline{f}\in D$. The maps
$$
\R \ni \Delta{\varphi}^l_\tau \mapsto \tilde{\varphi}_\tau(\underline{f}+\Delta {\varphi}^l_\tau)\in \text{Diff }^2([0,1])
$$
are linear. Namely,
$$
\tilde{\varphi}_\tau(\underline{f}+\Delta {\varphi}^l_\tau)=\tilde{\varphi}_\tau(\underline{f})+\frac{\partial \tilde{\varphi}_\tau}{\partial {\varphi}^l_{\tau}}\Delta{\varphi}^l_\tau,
$$
where
$$
\frac{\partial \tilde{\varphi}_\tau}{\partial {\varphi}^l_{\tau}}=Z_{\left[y_1(\tau),1\right]},
$$
with norm
$$
\begin{aligned}
\left|\frac{\partial \tilde{\varphi}_\tau}{\partial {\varphi}^l_{\tau}}\right|_{2}&\le 1.
\end{aligned}
$$
\end{lem}
Let $\tau>\tau'$. For $\underline\varphi\in X_3$, we define
$$
\underline\varphi_{[\tau,\tau')}=O\circ\pi_{\tau'}\circ\pi^{\tau}\left(\underline\varphi\right).
$$
By Lemma \ref{zoompartials} and Lemma \ref{evaluation} we get the following.
\begin{lem}\label{partial phi tau>tau' partial phil} Let $\underline{f}\in D$ and let $\tau>\tau'$. The maps
$$
\R \ni \Delta{\varphi}^l_{\tau'} \mapsto \tilde{\varphi}_\tau(\underline{f}+\Delta {\varphi}^l_{\tau'})\in \text{Diff }^2([0,1])
$$
are differentiable. The derivatives are given by
$$
\begin{aligned}
\frac{\partial \tilde{\varphi}_\tau}{\partial {\varphi}^l_{\tau'}}&=D\underline\varphi_{[\tau,\tau')}^l\left(\varphi^l_{\tau'}\left(y_1(\tau')\right)\right)\cdot\frac{\partial \tilde{\varphi}_\tau}{\partial y_1({\tau})}\circ\frac{\partial {\varphi}^l_{\tau'}\left(y_1({\tau'})\right)}{\partial {\varphi}^l_{\tau'}},\\
\end{aligned}
$$
with norms
$$
\begin{aligned}
\left|\frac{\partial \tilde{\varphi}_\tau}{\partial {\varphi}^l_{\tau'}}\right|_{2}&=O\left(\left|\varphi_\tau^l\right|_3S_1\right).
\end{aligned}
$$
The dependance $\underline{f}\mapsto {\partial \tilde{\varphi}_\tau}/{\partial {\varphi}^l_{\tau'}}$ is continuous and 
$$
\left| \tilde{\varphi}_\tau(\underline{f}+\Delta{\varphi}_{\tau'}^l)-
\left[\tilde{\varphi}_\tau(\underline{f})+\frac{\partial \tilde{\varphi}_\tau}{\partial {\varphi}^l_{\tau'}}\Delta {\varphi}^l_{\tau'} \right]\right|_2
= O_{\underline{f}}\left( \left|{\varphi}^l_{\tau}\right|_{3+\epsilon}\left|\Delta{\varphi}_{\tau'}^l\right|_{2}^{1+\epsilon}\right).
$$
\end{lem}
Define the operator $D_{sl}: X_2 \to X_2$ by
$$
\left(D_{sl}\left(\Delta \underline{\varphi}^l\right)\right)_\tau=\sum_{\tau'\leq\tau}\frac{\partial \tilde{\varphi}_\tau}{\partial {\varphi}^l_{\tau'}}\Delta\varphi^l_{\tau'}.
$$

 This is a uniformly bounded operator, see Lemma \ref{partial phi tau=tau' partial phil} and Lemma \ref{partial phi tau>tau' partial phil} . We obtain the following.

\begin{prop}\label{Dslderivative} The dependance $\underline{f}\mapsto D_{sl}(\underline{f})$ is continuous and
$$
\left| \tilde{\varphi}(\underline{f}+\Delta \underline{\varphi}^l)-
\left[\tilde{\varphi}(\underline{f})+D_{sl}\Delta \underline{\varphi}^l\right]\right|_2
=O_{\underline{f}}(|\Delta \underline{\varphi}^l|_2^{1+\epsilon}).
$$
Moreover, the partial derivative extends to a uniformly bounded operator $D_{sl}: X_2 \to X_2$. 
\end{prop}
\subsection{The partial derivatives $D_{lr}$}\label{Dlr}
Observe that, ${\partial \tilde{\varphi}^l_\tau}/{\partial {\varphi}^r_{\tau'}}=0$ when $\tau\leq\frac{1}{2}$ and when $\tau>\frac{1}{2}, \theta\tau\neq\tau'$. The following lemma describe the nonzero partial derivatives.

\begin{lem}\label{partial phil tau=thetatau' partial phir} Let $\underline{f}\in D$, let $\tau>\frac{1}{2}$ and $\theta\tau=\tau'$. The maps
$$
\R \ni \Delta{\varphi}^r_{\tau'} \mapsto \tilde{\varphi}^l_\tau(\underline{f}+\Delta {\varphi}^r_{\tau'})\in \text{Diff }^2([0,1])
$$
are linear. Namely,
$$
\tilde{\varphi}^l_\tau(\underline{f}+\Delta {\varphi}^r_{\tau'})=\tilde{\varphi}^l_\tau(\underline{f})+\frac{\partial \tilde{\varphi}^l_\tau}{\partial {\varphi}^r_{\tau'}}\Delta{\varphi}^r_{\tau'},
$$
where
$$
\frac{\partial \tilde{\varphi}^l_\tau}{\partial {\varphi}^r_{\tau'}}=\text{Id},
$$
with norms
$$
\begin{aligned}
\left|\frac{\partial \tilde{\varphi}^l_\tau}{\partial {\varphi}^r_{\tau'}}\right|_{2}&=1.
\end{aligned}
$$
\end{lem}
Define the operator $D_{lr}: X_2 \to X_2$ by
$$
\left(D_{lr}\left(\Delta \underline{\varphi}^r\right)\right)_\tau=\Delta\varphi^r_{\theta\tau},
$$
where $\tau>\frac{1}{2}$.
 This is a uniformly bounded operator, see Lemma \ref{partial phil tau=thetatau' partial phir}. We obtain the following.

\begin{prop}\label{Dlrderivative} The dependance $\underline{f}\mapsto D_{lr}(\underline{f})$ is continuous and
$$
\left| \tilde{\varphi}^l(\underline{f}+\Delta \underline{\varphi}^r)-
\left[\tilde{\varphi}^l(\underline{f})+D_{lr}\Delta \underline{\varphi}^r\right]\right|_2
=0.
$$
Moreover, the partial derivative extends to a uniformly bounded operator $D_{lr}: X_2 \to X_2$. 
\end{prop}
\subsection{The partial derivatives $D_{ls}$}\label{Dls}
Observe that, ${\partial \tilde{\varphi}^l_\tau}/{\partial {\varphi}_{\tau'}}=0$ when $\tau\geq\frac{1}{2}$ and when $\tau<\frac{1}{2}, \tau'<\theta\tau$. The following lemmas describe the nonzero partial derivatives.

\begin{lem}\label{partial phil thetatau=tau' partial phi} Let $\underline{f}\in D$ and let $\tau<\frac{1}{2}$. The maps
$$
\R \ni \Delta{\varphi}_{\theta\tau} \mapsto \tilde{\varphi}^l_\tau(\underline{f}+\Delta {\varphi}_{\theta\tau})\in \text{Diff }^2([0,1])
$$
are differentiable.  The derivatives are given by
$$
\begin{aligned}
\frac{\partial \tilde{\varphi}^l_\tau}{\partial {\varphi}_{\theta\tau}}&=Z_{\left[\varphi^{-1}\circ q_s^{-1}(1-S_2)(\theta\tau),1\right]}+\\&+\frac{\partial \tilde{\varphi}^l_\tau}{\partial \varphi^{-1}\circ q_s^{-1}(1-S_2)(\theta\tau)}\circ\frac{\partial {\varphi}_{\theta\tau}^{-1}\left(\varphi_{\theta\tau}\left(\varphi^{-1}\circ q_s^{-1}(1-S_2)(\theta\tau)\right)\right)}{\partial {\varphi}_{\theta\tau}}
\end{aligned},
$$
with norms
$$
\begin{aligned}
\left|\frac{\partial \tilde{\varphi}^l_\tau}{\partial {\varphi}_{\theta\tau}}\right|_{2}&=O(S_2).
\end{aligned}
$$
The dependance $\underline{f}\mapsto {\partial \tilde{\varphi}^l_\tau}/{\partial {\varphi}_{\theta\tau}}$ is continuous and 
$$
\left| \tilde{\varphi}^l_\tau(\underline{f}+\Delta{\varphi}_{\theta\tau})-
\left[\tilde{\varphi}^l_\tau(\underline{f})+\frac{\partial \tilde{\varphi}^l_\tau}{\partial {\varphi}_{\theta\tau}}\Delta {\varphi}_{\theta\tau} \right]\right|_2
= O_{\underline{f}}\left( \left|{\varphi}_{\theta\tau}\right|_{3+\epsilon}\left|\Delta\underline{\varphi}\right|_{2}^{1+\epsilon}\right).
$$
\end{lem}

By Lemma \ref{zoompartials} and Lemma \ref{evaluationinv} we get the following.
\begin{lem}\label{partial phil thetatau<tau' partial phi} Let $\underline{f}\in D$, $\tau<\frac{1}{2}$ and $\theta\tau<\tau'$. The maps
$$
\R \ni \Delta{\varphi}_{\tau'} \mapsto \tilde{\varphi}^l_\tau(\underline{f}+\Delta {\varphi}_{\tau'})\in \text{Diff }^2([0,1])
$$
are differentiable. The derivatives are given by
$$
\begin{aligned}
\frac{\partial \tilde{\varphi}^l_\tau}{\partial {\varphi}_{\tau'}}&=\frac{1}{D\underline\varphi_{[\tau',\theta\tau)}\left(\varphi_{\theta\tau}\left(\varphi^{-1}\circ q_s^{-1}(1-S_2)(\theta\tau)\right)\right)}\cdot\\&\cdot\frac{\partial \tilde{\varphi}^l_\tau}{\partial \varphi^{-1}\circ q_s^{-1}(1-S_2)(\theta\tau)}\circ\frac{\partial {\varphi}_{\tau'}^{-1}\left(\varphi_{\tau'}\left(\varphi^{-1}\circ q_s^{-1}(1-S_2)(\tau')\right)\right)}{\partial {\varphi}_{\tau'}},\\
\end{aligned}
$$
with norms
$$
\begin{aligned}
\left|\frac{\partial \tilde{\varphi}^l_\tau}{\partial {\varphi}_{\tau'}}\right|_{2}&=O\left(\left|\varphi_{\theta\tau}\right|_3S_2\right).
\end{aligned}
$$
The dependance $\underline{f}\mapsto {\partial \tilde{\varphi}^l_\tau}/{\partial {\varphi}_{\tau'}}$ is continuous and 
$$
\left| \tilde{\varphi}^l_\tau(\underline{f}+\Delta{\varphi}_{\tau'})-
\left[\tilde{\varphi}^l_\tau(\underline{f})+\frac{\partial \tilde{\varphi}^l_\tau}{\partial {\varphi}_{\tau'}}\Delta {\varphi}_{\tau'} \right]\right|_2
= O_{\underline{f}}\left( \left|{\varphi}_{\theta\tau}\right|_{3+\epsilon}\left|\Delta{\varphi}_{\tau'}\right|_{2}^{1+\epsilon}\right).
$$
\end{lem}
Define the operator $D_{ls}: X_2 \to X_2$ by
$$
\left(D_{ls}\left(\Delta \underline{\varphi}\right)\right)_\tau=\sum_{\tau'>\theta\tau,\\\tau<\frac{1}{2}}\frac{\partial \tilde{\varphi}^l_\tau}{\partial {\varphi}_{\tau'}}\Delta\varphi_{\tau'}.
$$

 This is a uniformly bounded operator, see Lemma \ref{partial phil thetatau=tau' partial phi} and Lemma \ref{partial phil thetatau<tau' partial phi} . We obtain the following.

\begin{prop}\label{Dlsderivative} The dependance $\underline{f}\mapsto D_{ls}(\underline{f})$ is continuous and
$$
\left| \tilde{\varphi}^l(\underline{f}+\Delta \underline{\varphi})-
\left[\tilde{\varphi}^l(\underline{f})+D_{ls}\Delta \underline{\varphi}\right]\right|_2
=O_{\underline{f}}(|\Delta \underline{\varphi}|_2^{1+\epsilon}).
$$
Moreover, the partial derivative extends to a uniformly bounded operator 
$D_{ls}: X_2 \to X_2$.
\end{prop}
\subsection{The partial derivatives $D_{rs}$}\label{Drs}
Observe that, ${\partial \tilde{\varphi}^r_\tau}/{\partial {\varphi}_{\tau'}}=0$ when $\tau\geq\frac{1}{2}$ and when $\tau<{1}/{2}, \tau'<\theta\tau$. The following lemmas describe the nonzero partial derivatives.

\begin{lem}\label{partial phir thetatau=tau' partial phi} Let $\underline{f}\in D$ and let $\tau<\frac{1}{2}$. The maps
$$
\R \ni \Delta{\varphi}_{\theta\tau} \mapsto \tilde{\varphi}^r_\tau(\underline{f}+\Delta {\varphi}_{\theta\tau})\in \text{Diff }^2([0,1])
$$
are differentiable.  The derivatives are given by
$$
\begin{aligned}
\frac{\partial \tilde{\varphi}^r_\tau}{\partial {\varphi}_{\theta\tau}}&=Z_{\left[0,\varphi^{-1}\circ q_s^{-1}(S_1S_2S_3)(\theta\tau)\right]}+\\&+\frac{\partial \tilde{\varphi}^r_\tau}{\partial \varphi^{-1}\circ q_s^{-1}(S_1S_2S_3)(\theta\tau)}\circ\frac{\partial {\varphi}_{\theta\tau}^{-1}\left(\varphi_{\theta\tau}\left(\varphi^{-1}\circ q_s^{-1}(S_1S_2S_3)(\theta\tau)\right)\right)}{\partial {\varphi}_{\theta\tau}},
\end{aligned}
$$
with norms
$$
\begin{aligned}
\left|\frac{\partial \tilde{\varphi}^r_\tau}{\partial {\varphi}_{\theta\tau}}\right|_{2}&=O(q_s^{-1}\left(S_1S_2S_3\right)).
\end{aligned}
$$
The dependance $\underline{f}\mapsto {\partial \tilde{\varphi}^r_\tau}/{\partial {\varphi}_{\theta\tau}}$ is continuous and 
$$
\left| \tilde{\varphi}^r_\tau(\underline{f}+\Delta{\varphi}_{\theta\tau})-
\left[\tilde{\varphi}^r_\tau(\underline{f})+\frac{\partial \tilde{\varphi}^r_\tau}{\partial {\varphi}_{\theta\tau}}\Delta {\varphi}_{\theta\tau} \right]\right|_2
= O_{\underline{f}}\left( \left|{\varphi}_{\theta\tau}\right|_{3+\epsilon}\left|\Delta\underline{\varphi}\right|_{2}^{1+\epsilon}\right).
$$
\end{lem}

By Lemma \ref{zoompartials} and Lemma \ref{evaluationinv} we get the following.
\begin{lem}\label{partial phir thetatau<tau' partial phi} Let $\underline{f}\in D$, let $\tau<\frac{1}{2}$ and $\theta\tau<\tau'$. The maps
$$
\R \ni \Delta{\varphi}_{\tau'} \mapsto \tilde{\varphi}^r_\tau(\underline{f}+\Delta {\varphi}_{\tau'})\in \text{Diff }^2([0,1])
$$
are differentiable. The derivatives are given by
$$
\begin{aligned}
\frac{\partial \tilde{\varphi}^r_\tau}{\partial {\varphi}_{\tau'}}&=\frac{1}{D\underline\varphi_{[\tau',\theta\tau)}\left(\varphi_{\theta\tau}\left(\varphi^{-1}\circ q_s^{-1}(S_1S_2S_3)(\theta\tau)\right)\right)}\cdot\\&\cdot\frac{\partial \tilde{\varphi}^r_\tau}{\partial \varphi^{-1}\circ q_s^{-1}(S_1S_2S_3)(\theta\tau)}\circ\frac{\partial {\varphi}_{\tau'}^{-1}\left(\varphi_{\tau'}\left(\varphi^{-1}\circ q_s^{-1}(S_1S_2S_3)(\tau')\right)\right)}{\partial {\varphi}_{\tau'}},\\
\end{aligned}
$$
with norms
$$
\begin{aligned}
\left|\frac{\partial \tilde{\varphi}^r_\tau}{\partial {\varphi}_{\tau'}}\right|_{2}&=O\left(\left|\varphi_{\theta\tau}\right|_3q_s^{-1}\left(S_1S_2S_3\right)\right).
\end{aligned}
$$
The dependance $\underline{f}\mapsto {\partial \tilde{\varphi}^r_\tau}/{\partial {\varphi}_{\tau'}}$ is continuous and 
$$
\left| \tilde{\varphi}^r_\tau(\underline{f}+\Delta{\varphi}_{\tau'})-
\left[\tilde{\varphi}^r_\tau(\underline{f})+\frac{\partial \tilde{\varphi}^r_\tau}{\partial {\varphi}_{\tau'}}\Delta {\varphi}_{\tau'} \right]\right|_2
= O_{\underline{f}}\left( \left|{\varphi}_{\theta\tau}\right|_{3+\epsilon}\left|\Delta{\varphi}_{\tau'}\right|_{2}^{1+\epsilon}\right).
$$
\end{lem}
Define the operator $D_{rs}: X_2 \to X_2$ by
$$
\left(D_{rs}\left(\Delta \underline{\varphi}\right)\right)_\tau=\sum_{\tau'>\theta\tau,\\\tau<\frac{1}{2}}\frac{\partial \tilde{\varphi}^r_\tau}{\partial {\varphi}_{\tau'}}\Delta\varphi_{\tau'}.
$$

 This is a uniformly bounded operator, see Lemma \ref{partial phir thetatau=tau' partial phi} and Lemma \ref{partial phir thetatau<tau' partial phi} . We obtain the following.

\begin{prop}\label{Drsderivative} The dependance $\underline{f}\mapsto D_{rs}(\underline{f})$ is continuous and
$$
\left| \tilde{\varphi}^r(\underline{f}+\Delta \underline{\varphi})-
\left[\tilde{\varphi}^r(\underline{f})+D_{rs}\Delta \underline{\varphi}\right]\right|_2
=O_{\underline{f}}(|\Delta \underline{\varphi}|_2^{1+\epsilon}).
$$
Moreover, the partial derivative extends to a uniformly bounded operator 
$D_{rs}: X_2 \to X_2$.
\end{prop}
\subsection{The partial derivatives $D_{rl}$}\label{Drl}
Observe that, ${\partial \tilde{\varphi}^r_\tau}/{\partial {\varphi}^l_{\tau'}}=0$ when $\tau\leq{1}/{2}$ and when $\tau>{1}/{2}$ with $\tau'>\theta\tau$. The following lemmas describe the nonzero partial derivatives.

\begin{lem}\label{partial phir thetatau=tau' partial phil} Let $\underline{f}\in D$ and let $\tau>\frac{1}{2}$. The maps
$$
\R \ni \Delta{\varphi}^l_{\theta\tau} \mapsto \tilde{\varphi}^r_\tau(\underline{f}+\Delta {\varphi}^l_{\theta\tau})\in \text{Diff }^2([0,1])
$$
are linear. Namely,
$$
\tilde{\varphi}^r_\tau(\underline{f}+\Delta {\varphi}^l_{\theta\tau})=\tilde{\varphi}^r_\tau(\underline{f})+\frac{\partial \tilde{\varphi}^r_\tau}{\partial {\varphi}^l_{\theta\tau}}\Delta{\varphi}^l_{\theta\tau},
$$
where
$$
\frac{\partial \tilde{\varphi}^r_\tau}{\partial {\varphi}^l_{\theta\tau}}=Z_{\left[0,y_1(\theta\tau)\right]},
$$
with norms
$$
\begin{aligned}
\left|\frac{\partial \tilde{\varphi}^r_\tau}{\partial {\varphi}^l_{\theta\tau}}\right|_{2}&=O(S_1).
\end{aligned}
$$
\end{lem}
By Lemma \ref{zoompartials} and Lemma \ref{evaluation} we get the following.
\begin{lem}\label{partial phir thetatau>tau' partial phil} Let $\underline{f}\in D$ and let $\tau>\frac{1}{2}$ with $\theta\tau>\tau'$. The maps
$$
\R \ni \Delta{\varphi}^l_{\tau'} \mapsto \tilde{\varphi}^r_\tau(\underline{f}+\Delta {\varphi}^l_{\tau'})\in \text{Diff }^2([0,1])
$$
are differentiable. The derivatives are given by
$$
\begin{aligned}
\frac{\partial \tilde{\varphi}^r_\tau}{\partial {\varphi}^l_{\tau'}}&=D\underline\varphi_{[\theta\tau,\tau')}^l\left(\varphi^l_{\tau'}\left(y_1(\tau')\right)\right)\cdot\frac{\partial \tilde{\varphi}^r_\tau}{\partial y_1({\theta\tau})}\circ\frac{\partial {\varphi}^l_{\tau'}\left(y_1({\tau'})\right)}{\partial {\varphi}^l_{\tau'}},\\
\end{aligned}
$$
with norms
$$
\begin{aligned}
\left|\frac{\partial \tilde{\varphi}^r_\tau}{\partial {\varphi}^l_{\tau'}}\right|_{2}&=O\left(\left|\varphi_{\theta\tau}^l\right|_3S_1\right).
\end{aligned}
$$
The dependance $\underline{f}\mapsto {\partial \tilde{\varphi}^r_\tau}/{\partial {\varphi}^l_{\tau'}}$ is continuous and 
$$
\left| \tilde{\varphi}^r_\tau(\underline{f}+\Delta{\varphi}_{\tau'}^l)-
\left[\tilde{\varphi}^r_\tau(\underline{f})+\frac{\partial \tilde{\varphi}^r_\tau}{\partial {\varphi}^l_{\tau'}}\Delta {\varphi}^l_{\tau'} \right]\right|_2
= O_{\underline{f}}\left( \left|{\varphi}^l_{\theta\tau}\right|_{3+\epsilon}\left|\Delta{\varphi}_{\tau'}^l\right|_{2}^{1+\epsilon}\right).
$$
\end{lem}
Define the operator $D_{rl}: X_2 \to X_2$ by
$$
\left(D_{rl}\left(\Delta \underline{\varphi}^l\right)\right)_\tau=\sum_{\tau'\leq\theta\tau}\frac{\partial \tilde{\varphi}^r_\tau}{\partial {\varphi}^l_{\tau'}}\Delta\varphi^l_{\tau'}.
$$

 This is a uniformly bounded operator, see Lemma \ref{partial phir thetatau=tau' partial phil}, and Lemma \ref{partial phir thetatau>tau' partial phil} . We obtain the following.

\begin{prop}\label{Drlderivative} The dependance $\underline{f}\mapsto D_{rl}(\underline{f})$ is continuous and
$$
\left| \tilde{\varphi}^r(\underline{f}+\Delta \underline{\varphi}^l)-
\left[\tilde{\varphi}^r(\underline{f})+D_{rl}\Delta \underline{\varphi}^l\right]\right|_2
=O_{\underline{f}}(|\Delta \underline{\varphi}^l|_2^{1+\epsilon}).
$$
Moreover, the partial derivative extends to a uniformly bounded operator 
$D_{rl}: X_2 \to X_2$.
\end{prop}

\end{document}